\documentclass[a4paper, 11pt]{article}

\usepackage{amsmath,amssymb,amsthm}
\usepackage{dsfont}
\usepackage{graphicx}
\usepackage{subcaption}
\usepackage[british]{babel}
\usepackage[utf8]{inputenc}
\usepackage{xcolor}
\usepackage[colorlinks=true,linkcolor=blue,citecolor=red]{hyperref}
\usepackage{authblk}
\usepackage{blindtext}
\usepackage{epstopdf}
\usepackage{cite}

\newtheorem{remark}{Remark}

\newtheorem{theorem}{Theorem}
\newtheorem{algorithm}{Algorithm}
\newtheorem{lemma}{Lemma}

\usepackage{mathtools}

\newcommand{\be}{\begin{equation}}
\newcommand{\ee}{\end{equation}}
\newcommand{\z}{\mathbf{z}}
\newcommand{\nt}{n_{\textrm{TOT}}}
\newcommand{\Sv}{S_{\Delta v}}
\newcommand{\V}{\mathbb{V}}

\begin{document}
\title{Monte Carlo stochastic Galerkin methods for non-Maxwellian kinetic models of  multiagent systems with uncertainties}

\author[1]{Andrea Medaglia \thanks{andrea.medaglia02@universitadipavia.it}}
\author[2]{Andrea Tosin \thanks{andrea.tosin@polito.it}}
\author[1]{Mattia Zanella \thanks{mattia.zanella@unipv.it}}
\affil[1]{Department of Mathematics "F. Casorati", University of Pavia, Italy}
\affil[2]{Department of Mathematical Sciences "G. L. Lagrange", Politecnico di Torino, Italy}

\date{}

\maketitle

\abstract{
In this paper, we focus on the construction of a hybrid scheme for the approximation of non-Maxwellian kinetic models with uncertainties. In the context of multiagent systems, the introduction of a kernel at the kinetic level is useful to avoid unphysical interactions. The methods here proposed, combine a direct simulation Monte Carlo (DSMC) in the phase space together with stochastic Galerkin (sG) methods in the random space. The developed schemes preserve the main physical properties of the solution together with accuracy in the random space. The consistency of the methods is tested with respect to surrogate Fokker-Planck models that can be obtained in the quasi-invariant regime of parameters. Several applications of the schemes to non-Maxwellian models of multiagent systems are reported. 
}
\\[+.2cm]

\tableofcontents

\section{Introduction}

Kinetic equations are often studied to describe aggregate trends of large systems of interacting particles and have shown a remarkable effectivity in different research fields, ranging from classical rarefied gas dynamics to socio-economic and traffic flow dynamics. Without reviewing the huge literature on these topics, we mention \cite{CDOP,cercignani1994BOOK,CFRT,DM08,dimarco_acta,pareschi2013BOOK,perthame,prigogine1971BOOK} and the reference therein for an introduction to the subject. The contributions have to be further distinguished depending on the type of kernel characterizing the interaction frequency between particles or agents. It is worth mentioning that the introduction of a state-dependent kernel represents an essential tool in kinetic theory to enforce physical properties of rarefied gases \cite{desvillettes}, whereas it is currently underexplored in less classical applications to multi-agent systems. In this direction, we mention the following recent contributions \cite{during2019,during2021,DPTZ,Toscani_nonMax}. 

The deterministic description of multi-agent phenomena has often to face the lack of essential information on microscopic dynamics, initial states, or boundary conditions. Hence, it is of paramount importance to quantify and control possible deviations from expected trends and patterns due to unavoidable uncertainties in the model parameters and initial distributions. An established idea relies on considering these quantities as random variables influencing the evolution of the kinetic distribution, increasing, therefore, the dimensionality of the problem.  In recent years, we experienced a growing interest in the construction of numerical methods for kinetic equations with uncertainties, see the collection \cite{Jin2017}. Among the most popular techniques for uncertainty quantification, stochastic Galerkin (sG) methods are based on the construction of deterministic solvers and are capable to guarantee spectral convergence in the random field under suitable regularity assumptions \cite{LJ18,hu2019,Zhu2017}. Anyway, their computational cost is generally high due to the  {curse} of dimensionality of kinetic equations and they are highly intrusive with respect to the original formulation of the model. Furthermore, the main physical properties of the solution, like its positivity, entropy dissipation, and hyperbolicity, are lost. Besides sG methods, we find nonintrusive approaches to UQ that do not require significant modifications to the numerical scheme of the deterministic problem and are based on collocation strategies. Therefore, these latter methods are easy to parallelize and do not require any knowledge of the class of probability distributions of random parameters. In this direction, multi-fidelity approaches have been recently developed using control variate techniques, see \cite{bertaglia_etal,dimarco2019,dimarco2020,hu2021,gamba2021,pareschi2022}.

In this work, we follow a different path that is inspired by the novel approach proposed in the seminal work for mean-field equations \cite{Carrillo2019} and further extended to the homogeneous Boltzmann equation in \cite{pareschi2020JCP}. The proposed approach is capable to combine the efficiency of Direct Simulation Monte Carlo (DSMC) methods for nonlinear kinetic equations in the phase space \cite{babovsky1986,babovsky1989,nanbu1980,pareschi2001ESAIMP} with the accuracy of sG methods in the parameter space. The DSMC-sG method preserves the main physical properties of the kinetic solution along with spectral accuracy in the random space provided minimal regularity assumptions. Anyway, in a non-Maxwellian framework, the numerical formulation of the DSMC-sG method requires the introduction of step functions at the particles' level. As shown in \cite{pareschi2020JCP} for the variable hard-spheres (VHS) model, this fact can break spectral convergence of sG. For this reason, it has been shown that a mollification of the step function coupled with a thermalization of particles is capable to restore the physical validity of the model together with spectral accuracy in the random space. 

In models for collective phenomena, the equilibrium distribution of Boltzmann-type models is unknown and typically only mass is conserved. For this reason, we introduce a surrogate Fokker-Planck model that can be formally derived from the original model in the quasi-invariant limit \cite{toscani2006CMS}. In the case of non-Maxwellian interactions, we will obtain a nonlocal nonlinear Fokker-Planck class of equations whose equilibrium distribution can be approximated numerically through suitable deterministic methods \cite{Pareschi2018}. In particular, we will investigate the effects of a mollification of step functions introduced at the Monte Carlo level, coupled with a correction of nonconserved quantities computed by the approximation of the corresponding surrogate Fokker-Planck model. Numerical tests for kinetic models of wealth distribution and traffic flow have been performed. 

The rest of the paper is organized as follows. In Section \ref{sec:1} we introduce non-Maxwellian models for multi-agent systems with random inputs and we formally derive their corresponding Fokker-Planck models. Regularity of the solutions in the random space of the surrogate models has been investigated in Section \ref{subsec:reg}. In Section \ref{sec:nonM_models} we briefly review the basic features of some existing kinetic models for pure gambling, wealth distribution and traffic flow dynamics with non-Maxwellian kernels.   In Section \ref{sec:DSMC} then we construct the DSMC-sG methods and we provide results on the consistency of the method. Finally, in Section \ref{sec:numerics} several numerical results are presented which show the efficiency and accuracy of the introduced method. 

\section{Non-Maxwellian models with uncertain parameters} 
\label{sec:1}
To introduce the modelling setting we consider a binary interaction model with uncertain mixing \cite{Dimarco2017,tosin2018CMS,pareschi2022}. If two sampled particles that are  characterized by the pre-interaction states $v,w\in \mathbb{V} \subseteq \mathbb R$ interact, then their post-interaction states $v^\prime,w^\prime \in \mathbb{V}$ are obtained following the scheme
\begin{equation}
\label{eq:inter}
\begin{split}
v^\prime = v - \epsilon I_1(v,w,\z) + D_1(v)\eta_\epsilon,\\
w^\prime = w - \epsilon I_2(v,w,\z) + D_2(w)\eta_\epsilon,
\end{split}
\end{equation}
with $\epsilon>0$ a given constant, $I_1$, $I_2$ suitable interaction functions depending on the pre-interaction states and on the random quantity $\z \in \mathbb R^{d_\z}$, $d_{\z}\ge 1$.  Furthermore, $\eta_\epsilon$ is a random variable with zero mean and variance $\sigma_\epsilon^2$, and the functions $D_1$ and $D_2$ define the local relevance of the diffusion. Under suitable assumptions on the strength of the diffusion it is possible to show that the post-interaction states $v^\prime$, $w^\prime$ remain in $\mathbb{V}$. 

We adopt a classical kinetic theory approach based on the one-dimensional space-homogeneous Boltzmann equation, that describes the time evolution of the one-body distribution function $f=f(t,v,\z)$. The function $f(t,v,\z)$ identifies the state of the system, such that $f(t,v,\z)dv$ is the fraction of agents characterized by a state comprised between $v$ and $v + dv$ at time $t>0$ and parametrised by uncertainties defined in the random vector $\z$ with joint distribution $p(\z)$. The evolution of $f$ is given by the non-Maxwellian Boltzmann-type model 
\begin{equation}\label{eq:nonmax}
\begin{split}
&\dfrac{d}{dt} \int_{\mathbb V} \varphi(v) f(t,v,\z)dv  \\
&\qquad=\dfrac{1}{2} \left \langle \iint_{\mathbb V\times \mathbb V}B(v,w,\z)(\varphi(v^\prime) + \varphi(w^\prime)-\varphi(v)-\varphi(w))f(t,v,\z)f(t,w,\z)dvdw \right\rangle,
\end{split}
\end{equation}
 {where $\varphi:\mathbb V\rightarrow \mathbb R$ is a test function}, while the symmetric function $B(v,w,\z)$ denotes the collision kernel that  {characterizes} the collision frequency of agents with states $v$ and $w$. The notation $\langle \cdot \rangle$ expresses the expectation with respect to the random variable $\eta_\epsilon$. The model introduced in \eqref{eq:nonmax} can be complemented with uncertain initial condition $f(0,v,\z) = f_0(v,\z)$. 

It is worth noting that in the VHS framework of classical kinetic theory for rarefied gas dynamics the collisional kernel is assumed to be function of the relative velocity $|v-w|$, see \cite{cercignani1994BOOK}. The model \eqref{eq:nonmax} greatly simplifies in the Maxwellian case corresponding to $B(v,w,\z) \equiv 1$. In the description of collective models in the Maxwellian simplification of considering a constant interaction kernel has been largely considered \cite{during2015,pareschi2013BOOK,tosin2019MMS,Visconti2017}. Within the Maxwellian simplification it is possible to argue on the existence and uniqueness of large time behavior of the resulting model. In some cases, explicit solutions can be obtained, as in the famous one-dimensional Kac model \cite{kac59}. In models for collective phenomena, a precise analytical description of the kinetic emerging equilibrium distribution is very difficult to obtain. A possible way to overcome this difficulty relies on the possibility to study surrogate models, that are approximations of the kinetic model \eqref{eq:nonmax} in some limit and whose large time behavior is easily available. 

\subsection{Fokker-Planck approximation}
\label{subsec:FP}

The computation of the emerging equilibrium density of the Boltzmann model introduced in \eqref{eq:nonmax} is very challenging. An established way to overcome this difficulty relies on the introduction of the quasi-invariant limit \cite{cordier2005JSP,toscani2006CMS,tosin2018CMS} under which it is possible to derive a surrogate Fokker-Planck model for the interaction dynamics. The introduced scaling has connections with the grazing collision limit in classical kinetic theory. The main idea is to introduce a new time scale $\tau = \epsilon t$ and to define
\[
f_\epsilon(\tau,v,\z) = f(\tau/\epsilon,v,\z)
\]
that is solution to
\begin{equation}
\label{eq:nonmax_rescale}
\begin{split}
&\dfrac{d}{d\tau} \int_{\mathbb V}\varphi(v)f_\epsilon(\tau,v,\z)dv  \\
&\qquad=\dfrac{1}{2\epsilon} \left\langle \iint_{\mathbb V^2}B(v,w,\z)(\varphi(v^\prime)+\varphi(w^\prime)-\varphi(v)-\varphi(w))f_\epsilon(\tau,v,\z)f_\epsilon(\tau,w,\z)dv\,dw \right\rangle. 
\end{split}
\end{equation}
Hence, scaling the variance of the random variable $\eta_\epsilon$ as $\sigma_\epsilon^2 = \epsilon \sigma^2$ we have that for $\epsilon\ll1$ the interaction dynamics in \eqref{eq:inter} are quasi-invariant, since $v^\prime-v \ll1$ and $w^\prime-w\ll1$. Assuming then $\varphi$ smooth enough and at least $\varphi \in \mathbb C_0^3(\mathbb V)$, we can perform the following Taylor expansions
\[
\begin{split}
\varphi(v^\prime)-\varphi(v)& = (v^\prime-v)\dfrac{d}{dv}\varphi(v)+ \dfrac{1}{2}(v^\prime-v)^2\dfrac{d^2}{dv^2}\varphi(v) + \dfrac{1}{6}(v^\prime-v)^3\dfrac{d^3}{dv^3}\varphi(\bar v) \\
\varphi(w^\prime)-\varphi(w) &= (w^\prime-w)\dfrac{d}{dw}\varphi(w)+ \dfrac{1}{2}(w^\prime-w)^2\dfrac{d^2}{dw^2}\varphi(w) + \dfrac{1}{6}(w^\prime-w)^3\dfrac{d^3}{dw^3}\varphi(\bar w),
\end{split}
\]
with $\bar v\in (\textrm{min}\{v,v^\prime\},\textrm{max}\{v,v^\prime\})$, $\bar w\in (\textrm{min}\{w,w^\prime\},\textrm{max}\{w,w^\prime\})$. Plugging the above expansion in \eqref{eq:nonmax_rescale} we obtain 
\[
\begin{split}
&\dfrac{d}{d\tau} \int_{\mathbb V} \varphi(v)f_\epsilon(\tau,v,\z)dv \\
&= \dfrac{1}{2\epsilon} \left\{ \left\langle \iint_{\mathbb V^2}B(v,w,\z)\left[(v^\prime-v)\dfrac{d}{dv}\varphi(v) + (w^\prime-w)\dfrac{d}{dw}\varphi(w)\right] f_\epsilon(\tau,v,\z)f_\epsilon(\tau,w,\z)dv\,dw\right\rangle \right. \\
&+\dfrac{1}{2\epsilon}\left. \left \langle \iint_{\mathbb V^2}B(v,w,\z)\left[(v^\prime-v)^2 \dfrac{d^2}{dv^2}\varphi(v) + (w^\prime-w)^2 \dfrac{d^2}{dw^2}\varphi(w)\right]f_\epsilon(\tau,v,\z)f_\epsilon(\tau,w,\z)dv\,dw \right\rangle\right\}  \\
& +R_\varphi^\epsilon(f_\epsilon,f_\epsilon),
\end{split}\]
where $R_\varphi^\epsilon(f_\epsilon,f_\epsilon)$ is a reminder term of the following form 
\[
R_\varphi^\epsilon(f_\epsilon,f_\epsilon) = \dfrac{1}{6\epsilon} \left\langle \iint_{\mathbb V^2} B(v,w,\z) \left[(v^\prime-v)^3 \dfrac{d^3}{dv^3}\varphi(v)+(w^\prime-w)^3 \dfrac{d^3}{dw^3}\varphi(w) \right] f_\epsilon(\tau,v,\z)f_\epsilon(\tau,w,\z)dv\,dw\right\rangle.
\]
Thanks to the smoothness assumptions on $\varphi$ and the boundedness of the third order moment of $\eta_\epsilon$,  {in the limit $\epsilon\rightarrow 0$} we have that 
\[
|R_\varphi^\epsilon(f_\epsilon,f_\epsilon)| \rightarrow 0^+
\]
and we may assume that $f_\epsilon$ converges to a distribution $f(\tau,v,\z)$ at least formally, we point the interested reader to \cite{CFRT,toscani2006CMS} for related approaches in the Maxwellian and Boltzmann-Povzner frameworks.  {With a slight abuse of notation, we indicate with $f(\tau,v,\z)$ the limit distribution as $\epsilon\rightarrow 0^+$, hence $f(\tau,v,\z)$ is weak solution to the nonlinear nonlocal Fokker-Planck equation}
\begin{equation}
\label{eq:FP_general}
\begin{split}
&\partial_\tau f(\tau,v,\z) =  \partial_v \left[ \int_{\mathbb V} B(v,w,\z) (I_1(v,w,\z)+I_2(v,w,\z)) f(\tau,w,\z)dw f(\tau,v,\z)\right] \\
&\qquad {+}\dfrac{\sigma^2}{2}\partial_v^2 \left[ \int_{\mathbb V} B(v,w,\z)f(\tau,w,\z)dw (D^2_1(v)+D^2_2(v)) f(\tau,v,\z) \right], 
\end{split}
\end{equation}
complemented with the following boundary conditions for all $\z \in \mathbb R^d$
\begin{equation}
\label{eq:BC}
\begin{split}
& \int_{\V} B(v,w,\z) (I_1(v,w,\z)+I_2(v,w,\z)) f(\tau,w,\z)dw f(\tau,v,\z) \\
&\qquad\qquad+ \dfrac{\sigma^2}{2}\partial_v\left( (D^2_1(v)+D^2_2(v))\int_{\V} B(v,w,\z)f(\tau,w,\z)dw f(\tau,v,\z)\right) \Big|_{v \in \partial V} = 0	\\
&\int_{\V} B(v,w,\z)f(\tau,w,\z)dw (D^2_1(v)+D^2_2(v)) f(\tau,v,\z) \Big|_{v \in \partial V} = 0
\end{split}
\end{equation}

\subsection{Regularity of solutions in the random space}
\label{subsec:reg}
We recall that $p(\z): I_\z \rightarrow \mathbb R_+$ is the probability density of the random  {vector} $\z$. We define the weighted norm in $L^2_p(\V\times I_\z)$ as follows
\[
\|f(t) \|_{L^2_p(\V\times I_\z)} = \left(\int_{\V}\int_{I_\z} |f(t,v,\z)|^2 p(\z)d\z dv \right)^{1/2}.
\]
In the following we will provide sufficient conditions to guarantee regularity of the solution of the general Fokker-Planck model \eqref{eq:FP_general}. 

Let us rewrite first the Fokker-Planck models \eqref{eq:FP_general} as follows
\begin{equation}
\label{eq:FP_test}
\partial_t f(t,v,\z) = \partial_v \left[ \tilde{\mathcal B}[f](t,v,\z)f(t,v,\z) + \mathcal D[f](t,v,\z) \partial_v (f(t,v,\z)) \right],
\end{equation}
with 
\[
\begin{split}
\tilde{\mathcal B}[f](t,v,\z) =& \int_{\V} B(v,w,\z) (I_1(v,w,\z) + I_2(v,w,\z))f(t,w,\z)dw  \\
&+ \dfrac{\sigma^2}{2}(D_1^2(v)+D_2^2(v)) \partial_v \left( \int_{\V} B(v,w,\z)f(t,w,\z)dw\right)
\end{split}\]
and
\[
\mathcal D[f](t,v,\z) = \dfrac{\sigma^2}{2}(D_1^2(v)+D_2^2(v))\int_{\V} B(v,w,\z)f(t,w,\z)dw \ge 0
\]

We have
\begin{theorem}\label{th:1}
Given $B(v,w,\z)>0$, let $f(t,v,\z)$ be the solution of the Fokker-Planck model \eqref{eq:FP_test}. If $C_{\mathcal B} = \|\partial_v \tilde{ \mathcal B}[f] \|_{L^\infty_p(\V \times I_\z)}<+\infty$  {and $\|\partial_v^2 \tilde{ \mathcal B}[f] \|_{L^\infty_p(\V \times I_\z)}<+\infty$} we have
\begin{equation}
\label{eq:d0}
\|f(t) \|^{ {2}}_{L^2_p(\V\times I_\z)} \le e^{C_{\mathcal B} t}\|f(0) \|^{ {2}}_{L^2_p(\V\times I_\z)} 
\end{equation}
for all $t \ge0$, provided 
\[
f^2(t,v,\z)\tilde{\mathcal B}[f](t,v,\z)p(\z)\Big|_{\partial (\V\times I_\z)} = 0,\quad f(t,v,\z)\mathcal D[f](t,v,\z)\partial_v f(t,v,z)p(\z)\Big|_{\partial (\V\times I_\z)} = 0.
\] 
Furthermore, if $C_{\mathcal D} = \| \mathcal D[f]\|_{L^\infty_p(\V \times I_\z)} < +\infty$ we have
\begin{equation}
\label{eq:d1}
\| \partial_v f(t)\|^{ {2}}_{L^2_p(\V\times I_\z)} \le e^{2( {2} C_{\mathcal B} + C_{\mathcal D})t}\| \partial_v f(0)\|^{ {2}}_{L^2_p(\V\times I_\z)}.
\end{equation}
\end{theorem}

\begin{proof}
We multiply by $2 f(t,v,\z)p(\z)$ the nonlinear nonlocal Fokker-Planck equation \eqref{eq:FP_test} and we integrate it over $v$ and $\z$:
\[
\begin{split}
\dfrac{d}{dt} \| f\|_{L^2_p(\V\times I_\z)}^2 =&   \int_{I_\z \times V} 2f(t,v,\z)p(\z) \partial_v \left(\tilde{ \mathcal B}[f](t,v,\z)f(t,v,\z) \right) dv\,d\z\\
& + \int_{I_\z \times \V} 2f(t,v,\z)p(\z) \partial_v\left( \mathcal D[f](t,v,\z) \partial_v f(t,v,\z)  \right) dv\,d\z \\
=& I + II.
\end{split}\]
For the integral $I$ we have
\[
\begin{split}
 &\int_{ \V \times I_\z} 2f(t,v,\z)p(\z) \partial_v \left( \tilde{\mathcal B}[f](t,v,\z) f(t,v,\z) \right) dv\,d\z\\
 &= \int_{ \V\times I_\z} 2f^2(t,v,\z)\partial_v \tilde{ \mathcal B}[f](t,v,\z)p(\z) d\z\,dv + \int_{I_\z \times \V} 2f(t,v,\z) \tilde{ \mathcal B}[f](t,v,\z)\partial_v f(t,v,\z) p(\z) dv\,d\z\\
& \le 2C_{\mathcal B}  \| f \|_{L^2_p(\V\times I_\z)} - 2 \int_{ \V\times I_\z}f(t,v,\z)\partial_v\left(f(t,v,\z)\tilde{\mathcal B}[f](t,v,\z) \right)p(\z)dv\,d\z, 
\end{split}
\]
since $f^2(t,v,\z)\tilde{\mathcal B}[f](t,v,\z)p(\z)\Big|_{\partial (\V\times I_\z)} = 0$. Therefore, we get
\[
\int_{\V\times I_\z} f(t,v,\z)p(\z) \partial_v \left( \tilde{\mathcal B}[f](t,v,\z)f(t,v,\z) \right) dv\,d\z \le \dfrac{C_{\mathcal B} }{2} \| f\|^{ {2}}_{L^2_p(\V\times I_\z)}.
\]
For the integral $II$ we have
\[
\begin{split}
&\int_{\V\times I_\z} 2 f(t,v,\z)p(\z) \partial_v \left(\mathcal D[f](t,v,\z) \partial_v f(t,v,\z)\right) dv\,d\z \\
&\quad= -2 \int_{ \V\times I_\z} (\partial_v f(t,v,\z))^2 \mathcal D[f](t,v,\z) {p(\z)}dv\,d\z \le 0,
\end{split}
\]
since $f(t,v,\z)\mathcal D[f](t,v,\z)\partial_v f(t,v,z)p(\z)\Big|_{\partial (\V\times I_\z)} = 0$ and $\mathcal D[f] \ge 0$.  Hence, we have
\[
\dfrac{d}{dt} \| f(t)\|_{L^2_p(\V\times I_\z)}^{ {2}} \le C_{\mathcal B} \| f(t)\|_{L^2_p(\V\times I_\z)}^{ {2}}.
\]
Thanks to Gronwall's Lemma we obtain
\[
\| f(t)\|_{L^2_p(\V\times I_\z)}^{ {2}} \le e^{C_{\mathcal B} t}\| f(0)\|_{L^2_p(\V\times I_\z)}^{ {2}}.
\]

Next, we apply the $v$ derivative to both members of \eqref{eq:FP_test} 
\[
\partial_t \partial_v f(t,v,\z) = \partial_v^2 \left[\tilde{\mathcal B}[f](t,v,\z)f(t,v,\z) + \mathcal D[f](t,v,\z)\partial_v f(t,v,\z) \right]. 
\]
We multiply by $2p(\z)\partial_v f(t,v,\z)$ both members of the latter equation and we integrate over $\V\times I_\z $
\[
\begin{split}
\dfrac{d}{dt} \|\partial_v f(t)\|^2_{L^2_p(\V\times I_\z)}  &= \int_{\V\times I_\z} 2 \partial_v f(t,v,\z) \partial_v  {\Big[} \partial_v (\tilde{\mathcal B}[f](t,v,\z)f(t,v,\z) )  \\
&\quad + \partial_v(\mathcal D[f](t,v,\z) \partial_v f(t,v,\z))  {\Big]} p(\z)dvd\z. \end{split}\]
We have
\[
\begin{split}
&\int_{ \V\times I_\z}2\partial_v f(t,v,\z) \partial_v \left[\partial_v \left(\tilde{\mathcal B}[f](t,v,\z)f(t,v,\z) \right) \right]  {p(\z)d\z} \\
&= \int_{ \V\times I_\z} 2\partial_v f(t,v,\z)[ \tilde{ {2}\partial_v\mathcal B}[f](t,v,\z)\partial_v f(t,v,\z) + \tilde{\mathcal B}[f](t,v,\z)\partial_v^2f(t,v,\z) + f(t,v,\z)\partial_v^2\tilde{\mathcal B}[f](t,v,\z)]   {p(\z)d\z}\\
& \le  {4} C_{\mathcal B}\|\partial_vf(t) \|_{L^2_p(\V\times I_\z)}^2,
\end{split}\]
and
\[
\int_{\V\times I_\z} 2\partial_v f(t,v,\z)\partial_v^2 (\mathcal D[f](t,v,\z)\partial_vf(t,v,\z)) {p(\z)}dv d\z \le 2C_{\mathcal D}\|\partial_vf \|_{L^2_p(\V\times I_\z)}^2 - \|\partial_v^2 f \|^{ {2}}_{L^2_p(\V\times I_\z)}.
\]
Therefore, we have
\[
\dfrac{d}{dt} \|\partial_v f(t) \|^2_{L^2_p(\V\times I_\z)} \le 2( {2} C_{\mathcal B} +C_{\mathcal D}) \|\partial_vf(t)\|_{L^2_p(\V\times I_\z)}^2
\]
and from the Gronwall inequality we get
\[
 \|\partial_v f(t) \|^2_{L^2_p(\V\times I_\z)} \le e^{2( {2} C_{\mathcal B} +C_{\mathcal D})t} \|\partial_v f(0) \|^2_{L^2_p(\V\times I_\z)},
\]
corresponding to \eqref{eq:d1}, from which we conclude the proof.
\end{proof}

\begin{theorem}\label{th:2}
Given $B(v,w,\z) >0$, let $f(t,v,\z)$ the solution of the Fokker-Planck model \eqref{eq:FP_test} and let us consider the constants
$C_{\mathcal B} = \| \partial_v \tilde{\mathcal B}[f](t,v,\z)\|_{L^\infty_p(\V \times I_\z)}< +\infty$, 
$C_{\mathcal B,1} = \| \partial_\z \tilde{\mathcal B}[f](t,v,\z)\|_{L^\infty_p(\V \times I_\z)}< +\infty$ and $C_{\mathcal B,2} = \| \partial_\z\partial_v \tilde{\mathcal B}[f](t,v,\z)\|_{L^\infty_p(\V \times I_\z)}< +\infty$. Then, if $\sigma^2 = 0$ we have
 {
\[
\begin{split}
\| \partial_{\z} f(t) \|^2_{L^2_p(\V\times I_\z)} &\le 
e^{(2C_{\mathcal B} + C_{\mathcal{B},1}+ C_{\mathcal{B},2})t}\|\partial_\z f(0) \|_{L^2_p(\V \times I_\z)}^2  \\
&\quad+ \dfrac{C_{\mathcal{B},2}}{C_{\mathcal{B}} + C_{\mathcal B,1} + C_{\mathcal B,2}} \| f(0) \|^2_{L^2_p(\V \times I_\z)} (e^{(2C_{\mathcal B} + C_{\mathcal B,1} + C_{\mathcal B,2})t}-e^{C_{\mathcal B}t}) \\
&\quad+ \dfrac{C_{\mathcal{B},1}}{ C_{\mathcal B,1} + C_{\mathcal B,2} -2C_{\mathcal{B}}-2C_{\mathcal D}} \| \partial_v f(0) \|^2_{L^2_p(\V \times I_\z)} (e^{(2C_{\mathcal B} + C_{\mathcal B,1} + C_{\mathcal B,2})t}-e^{2(2C_{\mathcal B}+C_{\mathcal{D}})t}),
%e^{\left(C_{{\mathcal B}}+\frac{C_{{\mathcal B},2}}{2}\right)t}\|\partial_\z f(0) \|_{L^2_p(\V\times I_\z)}^2 + \dfrac{2C_{{\mathcal B},2} e^{C_{\mathcal B} t}(e^{C_{\mathcal B}t}-e^{C_{\mathcal B,2}t/2})}{2(2C_{\mathcal B}-C_{\mathcal B,2})}\| \partial_vf(0)\|_{L^2_p(\V\times I_\z)}
\end{split}
\]
where $C_{\mathcal D} = \| \mathcal D[f]\|_{L^\infty_p(\V \times I_\z)} < +\infty$.}
\end{theorem}

\begin{proof}

Let us consider the $\z$ derivative of the Fokker-Planck model \eqref{eq:FP_test} with $\sigma^2 = 0$
\[
\partial_t \partial_{\z} f(t,v,\z) = \partial_{\z}\partial_v \left[ \tilde{\mathcal B}[f](t,v,\z)f(t,v,\z) \right].% + \mathcal D[f](t,v,\z)\partial_v(f(t,v,\z)) \right]. 
\]
We multiply by $2 p(\z) \partial_{\z} f$ and we integrate over $I_\z \times \V$
\[
\begin{split}
&\int_{ \V\times I_\z} 2 p(\z) \partial_{\z} f(t,v,\z) \partial_t (\partial_{\z} f(t,v,\z)) dv\,d\z \\
& \qquad=\int_{ \V\times I_\z} 2p(\z)\partial_{\z} f(t,v,\z) \partial_{\z} \left[\tilde{\mathcal B}[f](t,v,\z)\partial_v f(t,v,\z) + (\partial_v \tilde{\mathcal B}[f](t,v,\z))f(t,v,\z) \right] dv\,d\z
% \\
%& + \int_{I_{\z}} \int_0^1 2p(\z) \partial_{\z} f \partial_v^2 [\mathcal D[f]\partial_{\z} f]dv\,d\z. 
\end{split}\]
Hence, we observe that 
\[
\begin{split}
&\int_{\V\times I_\z}  2p(\z)\partial_{\z} f \partial_{\z}((\partial_v\tilde{\mathcal B}[f](t,v,\z))f(t,v,\z))dv\,d\z   \\
&\quad\le  {2C_{{\mathcal B}} \| \partial_\z f(t) \|_{L^2_p(\V\times I_\z)}^2  + C_{\mathcal{B},2} \left[ \| \partial_{\z} f(t)\|^2_{L^2_p(\V\times I_\z)} + \| f(t)\|^2_{L^2_p(\V\times I_\z)}\right],}  \\
%&\quad =2\mathcal C_{{\mathcal B}} \| \partial_{\z}f(t)\|^2_{L^2_p(\V\times I_\z)}
\end{split}\]
 and
\[
\begin{split}
&\int_{\V\times I_{\z}} 2p(\z)\partial_{\z} f(t,v,\z) \partial_{\z}(\tilde{\mathcal B}[f](t,v,\z)\partial_vf(t,v,\z))dv\,d\z \\
&\qquad\le C_{{\mathcal B},1}\int_{\V\times I_\z }2p(\z)\partial_\z f(t,v,\z)\partial_v  f(t,v,\z) d\z dv  \\
&\qquad\qquad+  {C_{{\mathcal B}}} \int_{\V\times I_\z } 2p(\z)\partial_\z f(t,v,\z)\partial_\z\partial_v f(t,v,\z) d\z dv \\
&\qquad\le  { C_{{\mathcal B},1}} (\|\partial_\z f(t) \|_{L^2_p(\V\times I_\z)}^2 + \| \partial_v f(t)\|_{L^2_p(\V\times I_\z)}^2), 
\end{split}
\]
thanks to the Young's inequality. 
%Furthermore we have
%\[\begin{split}
%&\int_{I_\z} \int_0^1 2p(\z) \partial_\z f \partial_v\partial_\z \left(  \mathcal D[f]\partial_vf \right)dv\,d\z \\
%&\quad=  - \int_{I_\z} \int_0^1 2p(\z) (\partial_v\partial_\z f)(\partial_\z \mathcal D[f]\partial_v f )dv\,d\z - \int_0^1 \int_{I_\z} 2p(\z)( \partial_v \partial_\z f )^2\mathcal D[f] dv\,d\z \\
%&\quad\le -\int_0^1 \int_{I_\z} 2p(\z)\partial_v\partial_\z f \partial_\z\mathcal D[f] \partial_v fdv\,d\z 
%\end{split}\]
%since $ \mathcal D[f] \ge 0$ (NB $\partial_\z \mathcal D[f] \le 0$). 
%If $D(\cdot) = 0$ we can close the chain of inequalities. 
Hence, we obtained
\[
\dfrac{d}{dt} \| \partial_\z f(t)\|_{L^2_p(\V\times I_\z)}^{ {2}} \le  {\left(2C_{{\mathcal B}}+C_{{\mathcal B},1}+C_{{\mathcal B},2} \right)\| \partial_\z f(t)\|_{L^2_p(\V\times I_\z)}^2  +C_{{\mathcal B},2}\| f(t)\|_{L^2_p(\V\times I_\z)}^2 + C_{{\mathcal B},1}\| \partial_v f(t)\|_{L^2_p(\V\times I_\z)}^2}
\]
Thanks to the uniform Gronwall inequality, see \cite{temam88} p.88, we have
  {
\[
\begin{split}
\|\partial_\z f(t)\|^2_{L^2_p(\V\times I_\z)} &\le e^{\left(2C_{{\mathcal B}}+C_{{\mathcal B},1} + C_{{\mathcal B},2} \right)t}\|\partial_\z f(0) \|_{L^2_p(\V\times I_\z)}^2 \\
&\quad+C_{{\mathcal B},2} \int_0^t \| f(s)\|_{L^2_p(\V\times I_\z)}^2 e^{\left(2C_{{\mathcal B}}+C_{{\mathcal B},1}+C_{{\mathcal B},2}\right)(t-s)}ds \\
& \quad +C_{{\mathcal B},1} \int_0^t \| \partial_v f(s)\|_{L^2_p(\V\times I_\z)}^2 e^{\left(2C_{{\mathcal B}}+C_{{\mathcal B},1}+C_{{\mathcal B},2}\right)(t-s)}ds
\end{split}
\]
}
Taking into account Theorem \ref{th:1} we obtain
 {
\[
\begin{split}
\|\partial_\z f(t)\|^2_{L^2_p(\V\times I_\z)} &\le e^{\left(2C_{{\mathcal B}}+C_{{\mathcal B},1}+C_{{\mathcal B},2}\right)t}\|\partial_\z f(0) \|_{L^2_p(\V\times I_\z)}^2 \\
&\quad+C_{{\mathcal B},2}  \int_0^t  \| f(0)\|^2_{L^2_p(\V\times I_\z)} e^{C_{\mathcal{B}}s}e^{\left(2C_{{\mathcal B}}+C_{{\mathcal B},1}+C_{{\mathcal B},2}\right)(t-s)}ds  \\
&\quad + C_{\mathcal{B},1} \int_0^t \| \partial_v f(0)\|_{L^2_p(\V \times I_\z)}^2 e^{(C_{\mathcal{D}}+2C_{\mathcal B})s}e^{\left(2C_{{\mathcal B}}+C_{{\mathcal B},1}+C_{{\mathcal B},2}\right)(t-s)}ds, 
 \end{split}\]
 }
 from which we conclude.
\end{proof}

\begin{remark}
Theorem \ref{th:1} implies that, provided $f$, $\partial_v f$ are in $L^2_p(\V\times I_\z)$ initially, then under suitable assumptions $f$, $\partial_v f$  remain in $L^2_p(\V\times I_\z)$ for later times. Furthermore, in the hypotheses of Theorem \ref{th:2} we have that at least $f \in H^1_p = \{f : \| \partial_\z^\ell f\|_{L^2_p(\V\times I_\z)}<+\infty, \ell = 0,1\}$ exploiting the regularity of $f$, $\partial_v f$. Anyway, the estimates are not sharp as the ones obtained for linear equations, see e.g. \cite{JLM,QL}. Future research efforts will be dedicated to obtain sharper estimates for nonlinear Fokker-Planck equations. 
\end{remark}

\section{Examples in non-Maxwellian models for collective phenomena} \label{sec:nonM_models}
We  {briefly} present three non-Maxwellian kinetic models for collective phenomena, namely a model for pure gamble \cite{Bassetti2010}, a model for wealth distribution \cite{Toscani_nonMax} where the binary scheme is based on the Cordier-Pareschi-Toscani model \cite{cordier2005JSP} and, finally, a variation of the traffic model presented in \cite{tosin2019MMS} that includes a speed-dependent interaction kernel.
  
\subsection{Pure gambling}
\label{subsec:gambling}

In the kinetic models for pure gambling the state space is $\V = \mathbb R_+$. Preliminary Maxwellian models have been introduced in \cite{Bassetti2010} in which the nonlinear Boltzmann-type model \eqref{eq:nonmax} with $B\equiv 1$ has been considered.  In the pure gambling processes \cite{Dragulescu2000}, the entire wealth of two agents is at stake at each interaction and randomly shared between agents.  Therefore, assuming that the game is fair, the binary interactions are of the type \eqref{eq:inter} with $I_1(v,w,\z) = (1-\omega)v - \omega w$, $I_2(v,w,\z) = (\omega-1)v + \omega w$ where $\omega$ is a random variable symmetric with respect to $1/2$ and we considered $\epsilon = 1$. Furthermore, we consider vanishing diffusion functions $D_1(v,\z) = D_2(w,\z) = 0$. 

In an economic framework, an agent with zero wealth cannot gamble. To mimic this situation, in \cite{Toscani_nonMax} it has been proposed to modify the classical kinetic gamble model of \cite{Bassetti2010} through an interaction kernel of the following form
\begin{equation}
\label{eq:B_gamb}
B(v,w,\z) = \kappa (vw)^\delta,
\end{equation}
where the exponent of the kernel is an uncertain quantity, i.e.  $\delta = \delta(\z)$. For the introduced gambling rules and in presence of the interaction kernel \eqref{eq:B_gamb}, the wealth density $f(t,v,\z)$ satisfies a bilinear non-Maxwellian Boltzmann-type equation that in weak form reads
\begin{equation}
\label{eq:nonMax_gamb}
\dfrac{d}{dt} \int_{\mathbb R_+} f(t,v,\z)dv = \kappa \left\langle \int_{\mathbb R_+^2}(vw)^{\delta}(\varphi(v^\prime)-\varphi(v))f(t,v,\z)f(t,w,\z)\,dv\,dw \right\rangle.
\end{equation}
It is worth to observe that, for any $\delta$, equation \eqref{eq:nonMax_gamb} conserved mass and momentum. The mass conservation can be easily observed by taking $\varphi(v)=1$, whereas for momentum conservation we  consider $\varphi(v) = v$ and we get
\[
\begin{split}
\dfrac{d}{dt} \int_{\mathbb R_+}vf(t,v,\z)dv &= \kappa \left\langle \int_{\mathbb R_+^2} \left( \omega v^{\delta+1}w^{\delta} + (\omega-1)v^\delta w^{\delta+1}\right) f(t,v,\z)f(t,w,\z)\,dv\,dw\right\rangle \\
&= \kappa \left[  \left\langle \omega \right\rangle m_{\delta}m_{\delta+1} + \left\langle \omega-1 \right\rangle m_{\delta}m_{\delta+1} \right] =0
\end{split}\]
since $-\omega$ and $\omega-1$ are identically distributed. Assuming now $\omega \sim \mathcal U([0,1])$, it is possible to show that for any $0<\delta<1$ the large time distribution of the model \eqref{eq:nonMax_gamb} incorporated the kernel uncertainties and is a Gamma density of the form 
\[ %\label{eq:equilibrium_gambling}
f^\infty(v,\z) = \dfrac{(1-\delta)^{1-\delta}}{\Gamma(1-\delta)} v^{-\delta} \textrm{exp} \left\{ -(1-\delta)v\right\}.
\]
The uncertain parameter $\delta$ characterizing the interaction kernel has a great influence on the large time behavior of the system. Indeed, it is worth to remark that the variance of $X \sim f^\infty$ reads
\[
\textrm{Var}(X) = \dfrac{1}{1-\delta}, 
\]
and inequalities of the money distribution increase with $\delta$ and blow up in the limit $\delta\rightarrow 1$. 

\subsection{Wealth distribution}
\label{subsec:wealth}

In recent years, several kinetic models for wealth distribution have been proposed. Also in this case, the state space is $\V = \mathbb R_+$. In the following, we concentrate on the modelling setting proposed in \cite{cordier2005JSP} in the case of interaction with a background distribution. In particular, in \cite{Toscani_nonMax} it is assumed that elementary wealth changes of an agent  are determined by interactions \eqref{eq:inter} with $I_1(v,w,\z) =  \lambda (v-w)$, $D_1(v,\z) = v$, $I_2(v,w;\z) = D_2(w,\z) = 0$ and $w \sim \mathcal E$ a background distribution with $w \in \mathbb R_+$. Therefore, the interaction scheme reads
\[
v^\prime = (1-\epsilon\lambda)v + \epsilon\lambda w + \eta_\epsilon v.
\]
The quantity $\lambda \in (0,1]$ determines the saving propensity and it is assumed $\eta_\epsilon \ge -1+\epsilon\lambda$. Furthermore, in an economic framework, the probability of transactions in which one player has no wealth to exchange is very rare. To this end, in \cite{Toscani_nonMax} the authors proposed the  kernel
\begin{equation}
\label{eq:kernel_wealth}
B(v,w,\z) = \kappa (vw)^\delta,
\end{equation}
with $\delta \in (0,1]$ and $\kappa>0$. In the following, we will concentrate on the case $\delta = \delta(\z)$. The resulting non-Maxwellian kinetic model in weak form reads
\begin{equation}
\label{eq:nonMax_we}
\int_{\mathbb R_+} f(t,v,\z)  \varphi(v)dv = \left \langle \int_{\mathbb R_+^2} \kappa (vw)^{\delta(\z)} (\varphi(v^\prime)-\varphi(v))f(t,v,\z)\mathcal E(w)\,dv\,dw \right\rangle
\end{equation}
and do not conserve mean and energy. In particular, the following estimates hold 
\[
\begin{split}
m_1(t,\z) = \int_{\mathbb R_+} v f(t,v,\z)dv \le \textrm{max} \left\{ m_1(0), \dfrac{M_{1+\delta}}{M_\delta} \right\} \\
m_2(t,\z) = \int_{\mathbb R_+} v^2 f(t,v,\z)dv \le \textrm{max} \left\{ m_2(0), \bar m_2(\z) \right\},
\end{split}\]
where
\[
\bar m_2(\z) = \left( \dfrac{\epsilon\lambda(1-\epsilon\lambda)M_{1+\delta} + \sqrt{(\epsilon\lambda(1-\epsilon\lambda)M_{1+\delta})^2 + M_\delta M_{2+\delta}(2\epsilon\lambda - \sigma^2_\epsilon -\epsilon^2 \lambda^2)}}{M_\delta(2\epsilon\lambda- \sigma_\epsilon^2-\epsilon^2\lambda^2)} \right),
\]
provided $\sigma^2+ \lambda^2 < 2\lambda$. Information on the large time behavior can be obtained by relying to a Fokker-Planck model  approximating the kinetic model \eqref{eq:nonMax_we}
 \be
 \label{eq:FP_wealth}
\partial_\tau f(\tau,v,\z) = \kappa M_\delta \left[ \lambda \partial_v \left( v^\delta (v-M_{1+\delta}/M_{\delta}) f(\tau,v,\z) \right) + \dfrac{\sigma^2}{2} \partial_v^2 (v^{2+\delta}f(\tau,v,\z)) \right]
 \ee
where
\[
M_\delta(\z) = \int_{\mathbb R_+} v^{\delta(\z)} \mathcal E(v)dv < +\infty, \qquad \delta(\z) \le4,
\]
adding  boundary conditions  of the type \eqref{eq:BC}. The equilibrium distribution is now given by
\be \label{eq:equilibrium_wealth}
f^\infty(v,\z) = \dfrac{(\mu(\z) m(\z)^{1+\mu+\delta})}{\Gamma(1+\mu+\delta)} \dfrac{\textrm{exp}\left( -\dfrac{\mu m(\z)}{v} \right)}{v^{2+\mu+\delta}}, 
\ee
where
\[
 \mu(\z) = \dfrac{2\tilde \lambda(\z)}{\tilde \sigma^2(\z)},\qquad \tilde \lambda(\z) = \kappa M_\delta(\z) \lambda, \qquad \tilde \sigma^2(\z) = \kappa M_\delta(\z) \sigma^2,\qquad m(\z) = \dfrac{M_{1+\delta}(\z)}{M_\delta(\z)}.
\]
We point the interested reader to \cite{Toscani_nonMax} for additional details.

\subsection{Traffic flow} \label{subsec:traffic}

%In details, the random variable $z_{1}$ mimics the lack of knowledge in the interaction law between vehicles, whereas the random variable $z_{2}$ influences the frequency of interaction. These variables represent a \textit{structural uncertainty} due to the fact that the physics of the interactions among the vehicles is partly heuristic. For those reasons, it is not restrictive to assume that the vector $\z$ has independent components, i.e. $\Psi(z_{1},\,z_{2})=\Psi_{1}(z_{1})\Psi_{2}(z_{2})$:
%\begin{equation}
%	\textrm{Prob}(\z  \in  A)=\iint_{A}\Psi_{1}(z_{1})\Psi_{2}(z_{2})dz_{2}dz_{1},
%\end{equation}
%with $A \subseteq \mathbb R^2$. 

In kinetic traffic modelling, non-constant interaction kernels have been frequently  considered, see e.g. \cite{choi2021,helbing1996,Herty2005,Herty2010,paveri1975TR,prigogine1971BOOK} and the references therein. In the following, we study the influence of a cross section on a traffic model recently proposed in the Maxwellian framework \cite{tosin2019MMS} in which $\V = [0,1]$. 

The time evolution of the distribution $f(t,v,\z)$  is determined by microscopic binary interactions responsible for speed changes. Given normalized pre-interaction speeds  $(v,\,w) \in [0,1]\times [0,1]$, the post-interaction speeds $(v',\,w')$ are determined by \eqref{eq:inter} with 
\begin{equation}
\label{eq:Inter_traf}
I_1(v,w,\z) = - \big( P(\rho,\z)(1-v)+(1-P(\rho,\z))(P(\rho,\z)w-v) \big) ,\qquad  I_2(v,w,\z) = 0
\end{equation}
 being  $P(\rho,\z)= (1-\rho)^{\mu(\z)} \in [0,1]$, $\mu(\z)>0$, the probability to accelerate with a  traffic density $\rho\in [0,1]$. The presence of uncertain quantities in $P(\rho,\z)$ is associated with different responses of vehicles in heterogeneous traffic conditions, see \cite{tosin2021MCRF}.  Furthermore, we consider $D_1(v) = D(v,\rho)$ and $D_2(v,\z) = 0$. Hence, the speed changes are determined by 
\begin{equation}
	\label{eq:bin_int_traf}
	\begin{split}
		& v'=v+\epsilon I(v,\,w,\z) + D(v,\rho)\eta_\epsilon \\
		& w'=w,
	\end{split}
\end{equation}
we point the interested reader to \cite{tosin2019MMS} for further details on the modeling setting. The choice of the function $D$ has to ensure that $v^\prime,w^\prime \in [0,1]$ for any $v,w \in [0,1]$. In \cite{tosin2019MMS} the authors proposed 
\[
D(v,\rho) = a(\rho) \sqrt{\left((1+\epsilon)v(1-v) - \dfrac{\epsilon}{4}\right)_+}, \qquad a(\rho)\ge 0, 
\]
to guarantee the existence of a constant $c>0$ such that, considering $\eta_\epsilon$ with support $[-c(1-\epsilon),c(1-\epsilon)]$, the post-interaction speeds comply with the bound  $0\le v^\prime,w^\prime\le 1$ for any $v,w \in [0,1]$.

Hence, the evolution of the density $f$ follows the non-Maxwellian Boltzmann-type equation in weak form 
%in the weak form for the time variation of the expectation value of $\varphi(v)$, namely any observable function of the velocity $v$. Indeed, the time evolution is due, on average, to the variation of $\varphi$ in the binary interaction \eqref{binaryinteraction}:
\begin{equation}
	\label{phi}
	\begin{split}
		&\frac{d}{dt}\int_{0}^{1}\varphi(v)f(t,v,\z)dv  \\
		&=\frac{1}{2}\left\langle \int_{[0,1]^2} B(|v-w|,\z) (\varphi(v')-\varphi(v)+\varphi(w')-\varphi(w))f(t,v,\z)f(t,w,\z)dvdw \right \rangle.
	\end{split}
\end{equation}
In \eqref{phi} the uncertain interaction kernel $B(|v-w|,\z)$ describes the frequency of interactions and depends on the relative velocity $|v-w|$ as follows
\begin{equation}
\label{eq:B_traf}
	B(|v-w|,\z)=|v-w|^{\alpha(\z)}, \qquad \alpha(\z)\geq0.
\end{equation} 
%such that if $\alpha(\z) \equiv 0$ we collapse on the Maxwellian setting where all vehicles interact. 
%Early works on traffic dynamics have always taken into account different cross sections in order to identify more frequent speed changes, we mention in this direction the pioneering work of Prigogine and Paveri-Fontana \cite{paveri1975TR,prigogine1971BOOK} and the more recent survey \cite{Herty2010} together with the references therein. 
Since a priori information on the frequency of interaction is missing, it seem reasonable to introduce an additional uncertain exponent $\z$ of the cross section as an uncertain quantity. 

From equation \eqref{phi} we can compute the evolution of mean speed $V(t,\z)$ of the system. We fix $\varphi(v)=v$ and the evolution of $V(t,\z) = \int_0^1 vf(t,v,\z)dv$ is given by 
 \begin{equation}
	\label{Vmean}
	\frac{dV(t,\z)}{dt}=\frac{\epsilon}{2}\int_{[0,1]^2}|v-w|^{\alpha(\z)}I(v,w,\z)f(t,v,\z)f(t,w,\z)dvdw,
\end{equation}
whose solution depends on the specific interaction kernel considered. For Maxwellian particles, corresponding to the choice $\alpha(\z)\equiv0$ we recall the results of \cite{tosin2019MMS,tosin2021MCRF} from which we are able to find a close equation for the time evolution of the mean speed $V(t,\z)$. 

For any other $\alpha(\z)>0$, the mean speed $V(t,\z)$ depends on higher momenta and its evolution cannot be expressed in closed form. To investigate the influence of the uncertain interaction kernel on the evolution of $V(t,\z)$, we fix $\alpha(\z) \ne 0$. If $\alpha(\z) = 1$ we get  $V^-(t,\z) \le V(t,\z)\le V^+(t,\z)$ where

\begin{figure}
	\begin{subfigure}{0.5\textwidth}
		\centering
		\includegraphics[width=0.75\linewidth]{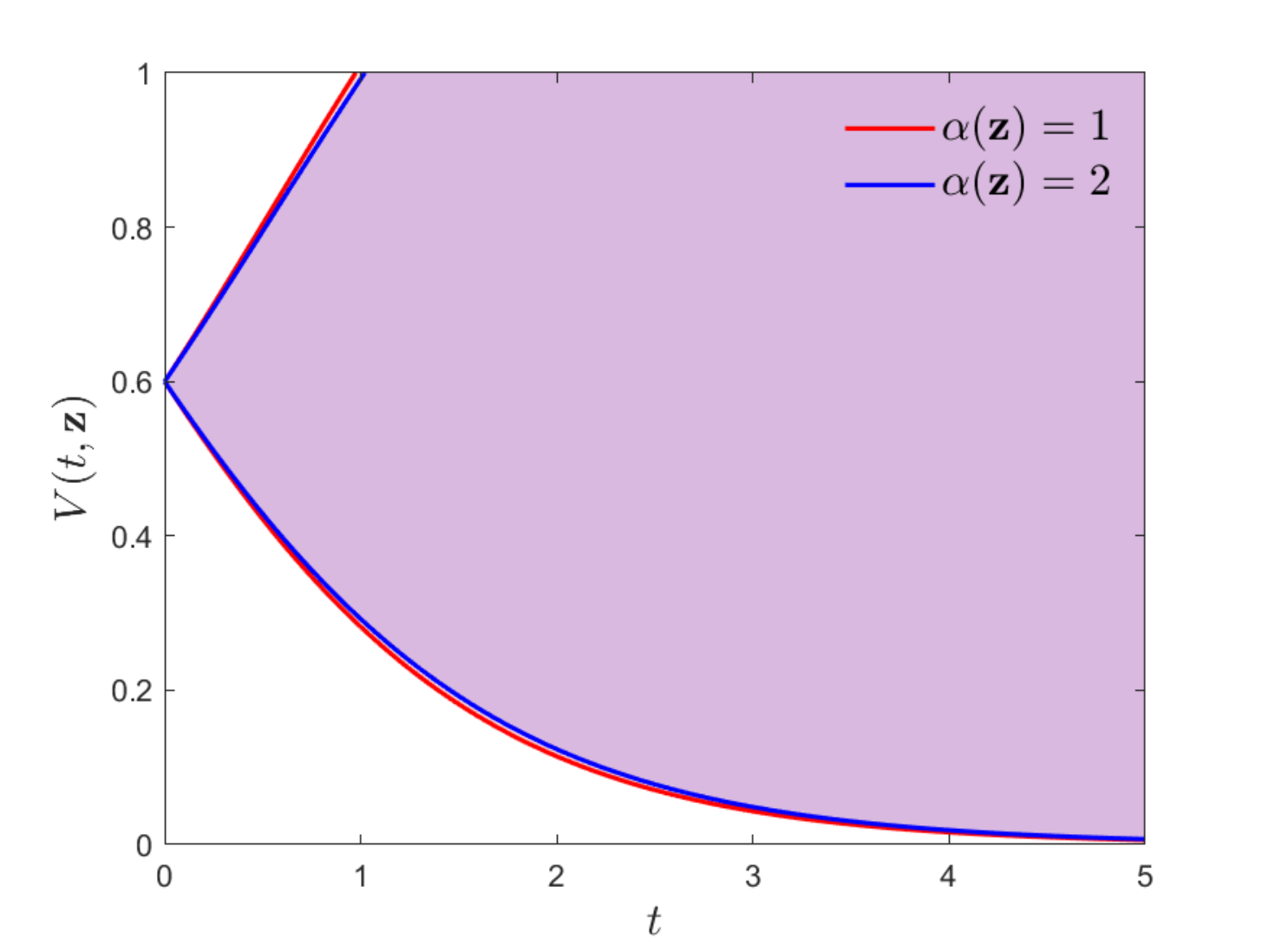}
		\caption{$\rho=0.4$, $\mu(\z)=1$}
		\label{fig:sub41}
	\end{subfigure}
	\begin{subfigure}{0.5\textwidth}
		\centering
		\includegraphics[width=0.75\linewidth]{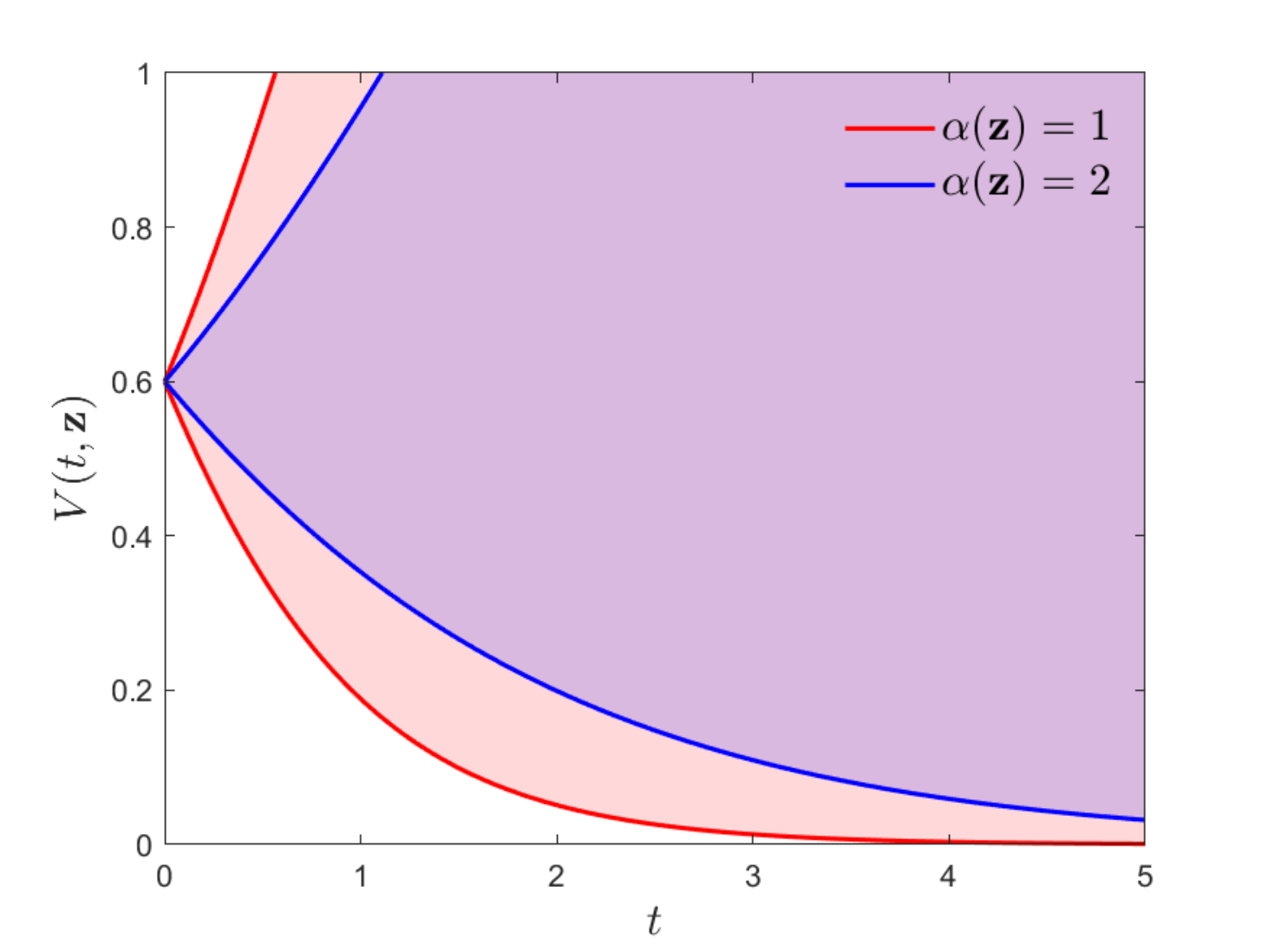}
		\caption{$\rho=0.4$, $\mu(\z)=3$}
		\label{fig:sub42}
	\end{subfigure}
	\begin{subfigure}{0.5\textwidth}
		\centering
		\includegraphics[width=0.75\linewidth]{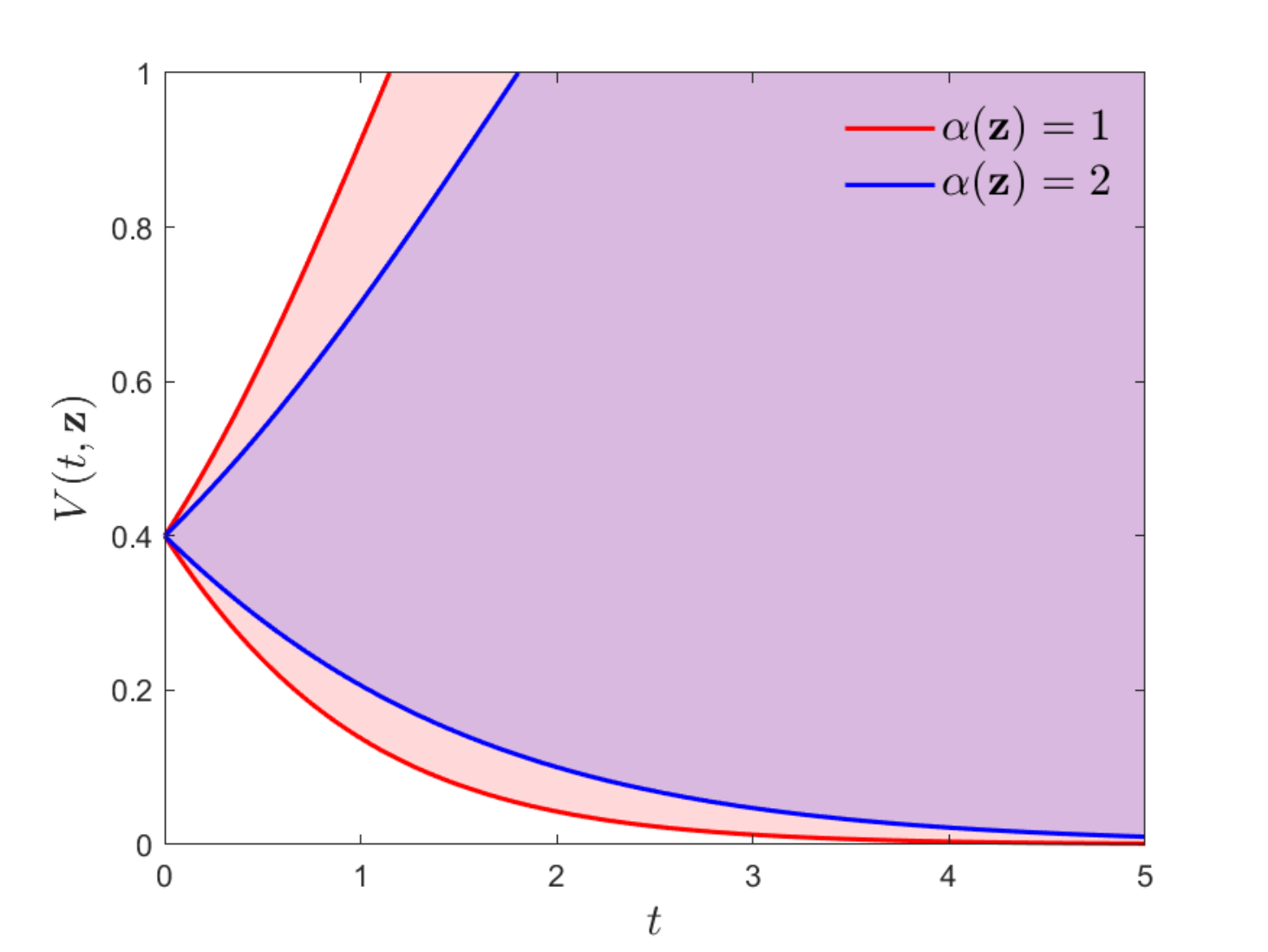}
		\caption{$\rho=0.6$, $\mu(\z)=1$}
		\label{fig:sub43}
	\end{subfigure}
	\begin{subfigure}{0.5\textwidth}
		\centering
		\includegraphics[width=0.75\linewidth]{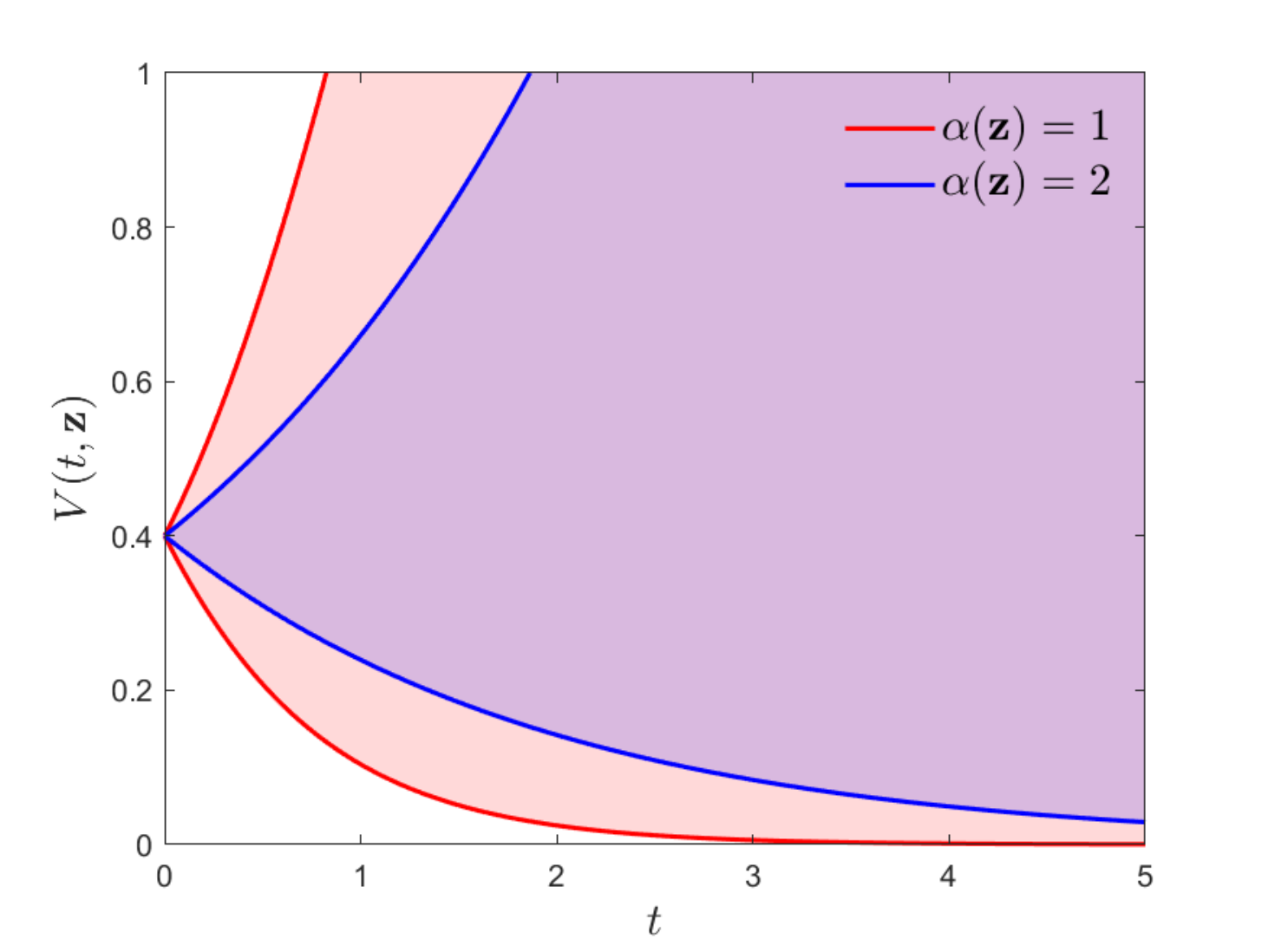}
		\caption{$\rho=0.6$, $\mu(\z)=3$}
		\label{fig:sub44}
	\end{subfigure}
	\caption{Upper and lower bounds for the time evolution of the mean velocity $V(t,\z)$ for $\alpha(\z)=1$ (red line) and $\alpha(\z)=2$ (blue line), for different values of the traffic density $\rho$ and the parameter $\mu(\z)$. The initial velocity is $V_{0}=1-\rho$.}
	\label{fig:test4}
\end{figure} 

%Both ODEs are of the Bernoulli-type with structure 
%\[
%\left| \frac{dV(t;\,z_{1})}{dt} \right| \leq aV(t;\,z_{1})-bV^2(t;\,z_{1}),
%\]
%with $a$ and $b$ coefficients. Removing the absolute value, upper $(+)$ and lower $(-)$ bounds for the time evolution of $V(t;\,z_{1})$ can be determined
%\begin{equation*}
%	bV^2(t;\,z_{1})-aV(t;\,z_{1})\leq\frac{dV(t;\,z_{1})}{dt}\leq aV(t;\,z_{1})-bV^2(t;\,z_{1}),
%\end{equation*}
%whose solutions are
\begin{equation}
	\label{bernoulli2}
	\begin{split}
		& V^{+}(t,\z)=\left[\frac{C_2(\z)}{C_1(\z)}+e^{-C_1(\z)t}\left(\frac{1}{V_{0}}-\frac{C_2(\z)}{C_1(\z)}\right) \right]^{-1} \\ 
		& V^{-}(t,\z)=\left[\frac{C_2(\z)}{C_1(\z)}+e^{C_1(\z)t}\left(\frac{1}{V_{0}}-\frac{C_2(\z)}{C_1(\z)}\right) \right]^{-1} 
	\end{split}
\end{equation} 
and 
\begin{equation}
	\label{bernoulli}
	\begin{split}
	C_1(\z) = 
	\begin{cases} 
		\frac{\epsilon}{2}(3-P(\rho,\z)-P^2(\rho,\z)) & \beta>0\\
		\epsilon P(\rho,\z) & \beta \le 0
	\end{cases}
	\\
	C_2(\z) = 
	\begin{cases} 
		\frac{\epsilon}{2}(1+P(\rho,\z)-P^2(\rho,\z)) & \beta> 0  \\ 
		\frac{\epsilon}{2}(1+P(\rho,\z)-P^2(\rho,\z)) & \beta\leq 0.
	\end{cases}
	\end{split}
\end{equation} 

Let us consider now $\alpha(\z) =2 $. We define the following constants $C_1(\z)=\epsilon P(\rho,\z)\geq0$ and $C_2(\z)=\frac{\epsilon}{2}(P^2(\rho,\z)+1-P(\rho,\z))>0$ and from \eqref{Vmean} we get $V^-(t,\z)\le V(t,\z) \le V^+(t,\z)$ where
\begin{equation}
	\begin{split}
		 V^{+}(t,\z)&=\left[\frac{C_1(\z)}{C_1(\z)+C_2(\z)}+e^{-(C_1(\z)+C_2(\z))t}\left(\frac{1}{V_{0}}-\frac{C_1(\z)}{C_1(\z)+C_2(\z)}\right) \right]^{-1} \\ 
		 V^{-}(t,\z)&=\left[\frac{C_1(\z)}{C_1(\z)+C_2(\z)}+e^{(C_1(\z)+C_2(\z))t}\left(\frac{1}{V_{0}}-\frac{C_1(\z)}{C_1(\z)+C_2(\z)}\right) \right]^{-1},
	\end{split}
\end{equation} 
with $ V(0,\z) = V_0 \in [0,1]$. We point the interested reader to Appendix \ref{append:A} for additional details. The impact of the interaction kernel \eqref{eq:B_traf} on the introduced traffic model is shown in Figure \ref{fig:test4}.

As in Section \ref{subsec:FP}, the Boltzmann model \eqref{phi} can be approximated through a surrogate Fokker-Planck model in the quasi-invariant limit. In particular, for the introduced traffic model with interaction kernel we get
\begin{equation}
	\label{FP_nocontrol}
	\begin{split}
		\partial_{t} f(t,v,\z)=&\frac{\sigma^2}{2}\partial_{v}^{2}\left[\left(\int_{0}^{1} B(|v-w|,\z)f(t,w,\z)dw \right)D^{2}(v,\rho)f(t,v,\z)\right] \\
		&-\partial_{v}\left[\left(\int_{0}^{1} B(|v-w|,\z)I(v,w,\z)f(t,w,\z)dw \right)f(t,v,\z)\right], 
	\end{split}
\end{equation} 
where $I$ has been defined in \eqref{eq:Inter_traf} and $B$ is the interaction kernel as in \eqref{eq:B_traf}. For the introduced Fokker-Planck equation, at variance with the cases in Sections \ref{subsec:gambling}-\ref{subsec:wealth}, we cannot compute analytically the equilibrium distribution of the Fokker-Planck model unless $\alpha(\z) = 0$, corresponding to the Maxwellian scenario. 

\section{DSMC stochastic Galerkin methods}
\label{sec:DSMC}

In this section, we revise the construction of a stochastic Galerkin version of the classical DSMC Algorithm for non-Maxwellian particles, see e.g. \cite{pareschi2001ESAIMP,pareschi2013BOOK}. In more detail, we extend the Direct Simulation Monte Carlo stochastic Galerkin (DSMC-sG) methods, introduced in the gas dynamic framework \cite{pareschi2020JCP}, to the models with uncertain parameters proposed in Section \ref{sec:nonM_models}. Next, we provide consistency results of the  DSMC-sG algorithm with respect to relevant observables and in the reconstruction of the kinetic density.

\subsection{DSMC-sG for non-Maxwellian models with uncertainties}\label{subsec:DSMC}
We first rewrite \eqref{eq:nonmax} in strong form to highlight the gain and loss part of the Boltzmann-type equation:
\begin{equation}
\label{eq:boltzmann_strong}
\begin{split}
	\frac{\partial}{\partial t}f(t,v,\z) & = Q(f,f)(t,v,\z) \\
	& = \left\langle \int_{\V} B(v,w,\z) \left( \dfrac{1}{J}f(t,{}^\prime v,\z)f(t,{}^\prime w,\z) - f(t,v,\z)f(t,w,\z) \right) dw \right\rangle, 
\end{split}
\end{equation}
where $J$ is the absolute value of the Jacobian of the considered transformation. We denote by $Q_\Sigma$ the operator obtained replacing the kernel $B(v,w,\z)$ with $B_\Sigma(v,w,\z)$ given by
\begin{equation}
	B_{\Sigma}(v,w,\z)=\min\{ B(v,w,\z) ,\Sigma\},
\end{equation}
where $\Sigma=\max\{B(v,w,\z)\}>0$ is an upper bound for the interaction kernel. By decomposing $Q_\Sigma(f,f)$ in its gain and loss part we can rewrite the interaction step as 
\[
\dfrac{\partial f}{\partial t}(t,v,\z) =  P(f,f)(t,v,\z) - \mu\Sigma f(t,v,\z),
\] 
with 
\[
P(f,f)(t,v,\z) =  Q^+_\Sigma(f,f)(t,v,\z) + f(t,v,\z) \left\langle \int_{\V}[\Sigma - B_\Sigma(v,w,\z)]f(t,w,\z)dw \right\rangle, 
\]
and 
$$Q^+_\Sigma(f,f) = \left \langle \int_{\V} B_\Sigma(v,w,\z) \dfrac{1}{J} f(t,{}^\prime v,\z)f(t,{}^\prime w,\z) dw \right \rangle,$$
with $\mu = \int_{\V} f(t,w,\z)dw$.
 
Let us now consider a time interval $[0,T]$ and let us discretize it in $\nt$ interval of size $\Delta t>0$. We denote by $f^n(v,\z)$ the approximation of $f(t^n,v,\z)$ and we consider the forward Euler scheme
\[
f^{n+1}=(1-\mu\Sigma\Delta t)f^n + \mu\Sigma\Delta t \dfrac{P(f^n,f^n)}{\mu\Sigma}, 
\]
where $f^{n+1}$ is a probability density provided $\mu\Sigma\Delta t\le 1$.

Then, we consider a sample of $N$ particles $v_{i}(\z,t)$, $i = 1,\dots,N$, from the kinetic solution of the Boltzmann model at time $t$, and we approximate $v_i(\z,t)$ by its generalized polynomial chaos expansion
\begin{equation}
	\label{eq:gPC}
	v_{i}^{M}(\z,t)=\sum_{h=0}^{M} \hat{v}_{i,h}(t)\Phi_{h}(\z),
\end{equation}
where $\{\Phi_{h}(\z)\}_{h=0}^{M}$, $M\in\mathbb{N}$, is a  {set of orthogonal polynomials} of degree less or equal to $M$, orthonormal with respect to the probability density function $p(\z)$
\begin{equation}
	\int_{I_\z}\Phi_{h}(\z) \Phi_{k}(\z) p(\z) d\z = \mathbb{E}_{\z}\left[\Phi_{h}(\cdot) \Phi_{k}(\cdot) \right] = \delta_{hk}, \qquad h,k=0,\dots,M,
\end{equation}
where $I_\z$ is the sample space and $\delta_{hk}$ is the Kronecker delta. The choice for the orthogonal polynomials obviously depends on the distribution of the parameters and follows the so-called Wiener-Askey scheme \cite{Xiu2010,Xiu2002}. In \eqref{eq:gPC}, $\hat{v}_{i,h}(t)$ is the projection of the velocity in the subspace generated by the polynomial of degree $h=0,\dots,M$
\begin{equation}
	\hat{v}_{i,h}(t)=\int_{I_\z}v_{i}(\z,t)\Phi_{h}(\z) p(\z) d\z =\mathbb{E}_{\z}\left[v_{i}(t,\cdot)\Phi_{h}(\cdot)\right].
\end{equation}

To perform collision with a non-Maxwellian kernel, we may rewrite the general binary interaction scheme \eqref{eq:inter} for two particles $v_i = v_{i}(\z,t)$, $w_j = w_{j}(\z,t)$, highlighting the acceptance-rejection process introduced by the classical Nanbu-Babovski method \cite{pareschi2001ESAIMP}
\begin{equation}
	\label{eq:interactionDSMC}
	\begin{split}
		& v_{i}'(\z,t) = v_{i} + \chi \left( \Sigma\xi < B(v_i,w_j,\z) \right) ( \epsilon I_1(v_i,w_j,\z) + D_1(v_{i})\eta_\epsilon ) \\
		& w_{j}'(\z,t) = w_{j} + \chi \left( \Sigma\xi < B(v_i,w_j,\z) \right) ( \epsilon I_2(v_i,w_j,\z) + D_2(w_{j})\eta_\epsilon )
	\end{split}
\end{equation}
where $\chi(\cdot)$ is the indicator function and $\xi$ a uniform random number in $(0,1)$. Then, we substitute the velocities $v_{i},\,w_{i}$ with their gPC approximation $v_{i}^{M},\,w_{j}^{M}$ and we project against $\Phi_{h}(\z)p(\z)d\z$ on $I_\z$ for every $h=0,\dots,M$. We obtain
\begin{equation}
	\label{eq:interactionDSMCproj}
	\begin{split}
		&\hat{v}_{i,h}'(t)=\hat{v}_{i,h}(t)+\hat{V}^{h}_{i,j} \\
		&\hat{w}_{j,h}'(t)=\hat{w}_{j,h}(t)+\hat{W}^{h}_{i,j},
	\end{split}
\end{equation}
where $\hat{V}^{h}_{i,j}$, $\hat{W}^{h}_{i,j}$ are the so-called collisional matrices
\begin{equation}
	\label{eq:collisionmatrix1}
	\begin{split}
		& \hat{V}^{h}_{i,j}=\int_{I_\z} \chi\left( \Sigma\xi < B(v^M_i,w^M_j,\z) \right) \left( ( \epsilon I_1(v^M_i,w^M_j,\z) + D_1(v^M_{i})\eta_\epsilon ) \right)\Phi_{h}(\z) p(\z) d\z \\
		& \hat{W}^{h}_{i,j}=\int_{I_\z} \chi\left( \Sigma\xi < B(v^M_i,w^M_j,\z) \right) \left( ( \epsilon I_2(v^M_i,w^M_j,\z) + D_2(w^M_{j})\eta_\epsilon ) \right)\Phi_{h}(\z) p(\z) d\z.
	\end{split}
\end{equation}
We stress the fact that the new binary interaction for the projections \eqref{eq:interactionDSMCproj} does not depend on the uncertain parameter $\z$. The DSMC-sG method is summarised in  Algorithm \ref{DSMC_sG_VHS}. 

\begin{algorithm}
	\label{DSMC_sG_VHS}
	(DSMC-sG).
	
	\begin{enumerate}
		\item compute the initial gPC expansion $\{v_{i}^{M,0},\,i=1,\dots,N\}$ from the  initial distribution $f(0,v,\z)$;
		\item for $n=0$ to $\nt-1$, \\
		given the projections $\{\hat{v}_{i,h}^{n},\,i=1,\dots,N,\,h=0,\dots,M\}$: 
		\begin{itemize}
			\item compute an upper bound $\Sigma$ of the kernel;
			\item set $N_{c}=\textrm{Sround}(\Sigma\mu\Delta t N/2)$; 			\item select $N_{c}$ dummy collision pairs, denoted by $(\hat{v}_{i,h},\,\hat{w}_{j,h})$ with $i=j=1,\dots,N_{c}$, uniformly among all possible pairs and for those: 
			\begin{itemize}
				\item select $\xi$ uniformly in $(0,\,1)$;
				\item compute the collision matrices $\hat{V}^{h}_{i,j}$, $\hat{W}^{h}_{i,j}$ \eqref{eq:collisionmatrix1} for $i,\,j=1,\dots,N_{c},\,h=0,\dots,M$; 
				\item perform the collision between $i$ and $j$ and compute $\hat{v}'_{i,h}$ and $\hat{w}'_{j,h}$ according to the collision law \eqref{eq:interactionDSMCproj};
				\item set $\hat{v}_{i,h}^{n+1}=\hat{v}'_{i,h}$ and $\hat{w}_{j,h}^{n+1}=\hat{w}'_{j,h}$;
			\end{itemize}
			\item set $\hat{v}_{i,h}^{n+1}=\hat{v}_{i,h}$ and $\hat{w}_{j,h}^{n+1}=\hat{w}_{j,h}$ for all the particles that have not been selected;
		\end{itemize}
		end for,
	\end{enumerate}
\end{algorithm}
where by $\textrm{Sround}(x)$ we denote the stochastic rounding of a positive real number $x$
\[
\textrm{Sround} = 
\begin{cases}
\lfloor x \rfloor + 1 & \textrm{with\, probability}\;x - \lfloor x \rfloor\\
\lfloor x \rfloor       & \textrm{with\, probability}\;1-x + \lfloor x \rfloor, 
\end{cases}
\]
where $\lfloor x \rfloor $ denotes the integer part of $x$. 
\subsection{Consistency estimates}
We want to evaluate the error produced by the DSMC-sG algorithm in the reconstructed distribution function and its moments. In the following, we denote by $f_\epsilon(t,v,\z)$ the solution of the Boltzmann equation \eqref{eq:nonmax_rescale} with binary updates \eqref{eq:inter} and by $f(t,v,\z)$ the corresponding mean field approximation, weak solution of the Fokker-Planck equation \eqref{eq:FP_general}. We introduce then the empirical density functions  
\[
f_{\epsilon,N}(t,v,\z) = \dfrac{1}{N} \sum_{i=1}^N \delta(v-v_i(t,\z)) \qquad f^M_{\epsilon,N}(t,v,\z) = \dfrac{1}{N} \sum_{i=1}^N \delta(v-v^M_i(t,\z)), 
\]
where $\delta(\cdot)$ is the Dirac delta function. Being $\varphi$ any test function, we denote by 
\[
\left\langle \varphi, f_\epsilon \right\rangle (\z, t) = \int_{\V} f_\epsilon(t,v,\z) dv 
\]
its expectation with respect to the distribution function $f_\epsilon(t,v,\z)$, so that we have 
\[
\left\langle \varphi, f_{\epsilon,N} \right\rangle (\z, t) = \dfrac{1}{N} \sum_{i=1}^{N} \varphi ( v_i (t,\z) ), \qquad  \left\langle \varphi, f^M_{\epsilon,N} \right\rangle (\z, t) = \dfrac{1}{N} \sum_{i=1}^{N} \varphi ( v^M_i (t,\z) ).
\]
From the central limit theorem we have the following result \cite{caflisch1998monte}
\begin{lemma} \label{lemma1}
	If we denote by $\mathbb{E}_V[\cdot]$ the expectation with respect to $f_\epsilon$ in the velocity space,  {for each $\z$} the root mean square error satisfies
	\[
	 {   \mathbb{E}_V \bigg[ \Big( \left\langle \varphi, f_\epsilon \right\rangle (\z, t) - \left\langle \varphi, f_{\epsilon,N} \right\rangle (t,\z)    \Big)^2 \bigg]^{1/2} } = \dfrac{\sigma_\varphi (t,\z)}{N^{1/2}}
	\]
	with
	\[
	\sigma^2_\varphi (t,\z)  =  \int_V   {\Big(} \varphi(v) - \left\langle \varphi, f_\epsilon \right\rangle (t,\z)  {\Big)^2}  f_\epsilon(t,v,\z) dv.
	\]
\end{lemma}
\noindent If $H^r(I_\z)$ is a weighted Sobolev space
\[
H^r_p(I_\z) = \bigg\{ v : I_\z \rightarrow \mathbb{R} \; : \; \dfrac{\partial^k v}{\partial z^k} \in L^2_p(I_\z), \; 0 \leq k \leq r \bigg\},
\]
from the polynomial approximation theory \cite{Xiu2010}, we have the following spectral estimate
\begin{lemma} \label{lemma2}
	For any $v(\z) \in H^r_p(I_\z), \, r\geq 0$, there exists a constant $C$ independent of $M>0$ such that 
	\[
	\| v - v^M \|_{L^2_p(I_\z)} \leq \dfrac{C}{M^r} \| v \|_{H^r_p(I_\z)}.
 	\]
\end{lemma}

Next, for any random variable $W(\z,t)$ taking values in $L^2_p(I_\z)$, we define
\[
\| W \|_{L^2(I_\z,L^2(\V))} = \| \mathbb{E}_V [W^2]^{1/2} \|_{L^2_p(I_\z)},
\]
and equivalently
\[
\| W \|_{L^2(\V,L^2(I_\z))} =  \mathbb{E}_V \left[ \| W \|^2_{L^2_p(I_\z)}  \right]^{1/2},   
\]
see \cite{dimarco2020,pareschi2020JCP,pareschi2022}. 
Then, we have the following result  % \cite{pareschi2020JCP}
\begin{theorem} \label{theorem:error_estimate}
	Let $f(t,v,\z)$ be a probability density function,  {solution of the weak Fokker-Planck equation \eqref{eq:FP_general}}, and $f^M_{\epsilon,N}(t,v,\z)$ the empirical measure obtained from the $N$-particles sG approximation $\{v^M_i(\z,t)\}_i$,  {solution of the time scaled Boltzmann equation \eqref{eq:nonmax_rescale}}. If $v_i(\z,t)\in H^r_p(I_\z)$ for every $i=1,\dots,N$, and in the quasi-invariant limit $\epsilon \rightarrow 0$, we have the following estimate 
	\[
	\| \left\langle \varphi, f \right\rangle - \left\langle \varphi, f^M_{\epsilon,N} \right\rangle \|_{L^2(\V,L^2_p(I_\z))} \leq \dfrac{\|\sigma_{\varphi}\|_{L^2_p(I_\z)}}{N^{1/2}} + \dfrac{C}{M^r} \left( \dfrac{1}{N} \sum_{i=1}^{N} \| \nabla \varphi(\xi_i) \|_{L^2_p(I_\z)} \right),
	\]
	where $\varphi$ is a test function, $C>0$ is a constant independent on $M$ and $\xi_i=(1-\theta)v_i+\theta v^M_i,\,\theta\in(0,1)$.
	\begin{proof}
		Thanks to the triangular inequality we have 
		\begin{equation}
			\begin{split}
				\| \left\langle \varphi, f \right\rangle - \left\langle \varphi, f^M_{\epsilon,N} \right\rangle \|_{L^2(\V,L^2_p(I_\z))}  \leq & \underbrace{  \| \left\langle \varphi, f \right\rangle   - \left\langle \varphi, f_{\epsilon} \right\rangle \|_{L^2(\V,L^2_p(I_\z))}   }_{I}  \\ 
				& +\quad \underbrace{  \| \left\langle \varphi, f_{\epsilon} \right\rangle - \left\langle \varphi, f_{\epsilon,N} \right\rangle \|_{L^2(\V,L^2_p(I_\z))}  }_{II} \\
				& +\quad \underbrace{  \| \left\langle \varphi, f_{\epsilon,N} \right\rangle - \left\langle \varphi, f^M_{\epsilon,N} \right\rangle \|_{L^2(\V,L^2_p(I_\z))}   }_{III}.
			\end{split}
		\end{equation} 
		In the quasi-invariant  regime $\epsilon\rightarrow 0$, up to the extraction of a subsequence, we have 
		\[
		\lim_{\epsilon\to 0} f_\epsilon(t,v,\z) = f(t,v,\z),
		\]
		as a consequence we have
		\[
		\lim_{\epsilon\to 0} \| \left\langle \varphi, f \right\rangle   - \left\langle \varphi, f_{\epsilon} \right\rangle \|_{L^2(\V,L^2(I_\z))} = 0,
		\]
		and the first term vanishes in the quasi-invariant limit. 
		The second term can be evaluated exploiting the result of Lemma \ref{lemma1}. Therefore, we have
		\[
		II = \dfrac{\| \sigma_{\varphi}(\z) \|_{L^2(I_\z)}}{N^{1/2}}.
		\]
Finally, we have
		\[
		 \left\| \dfrac{1}{N} \sum_{i=1}^{N} \big(  \varphi(v_i) - \varphi(v^M_i)  \big)\right\|_{L^2(\V,L^2_p(I_\z))} \leq \dfrac{1}{N} \sum_{i=1}^{N} \| \varphi(v_i) - \varphi(v^M_i)  \|_{L^2(\V,L^2_p(I_\z))},
		\]
		and from the mean value theorem $\varphi(v_i) - \varphi(v^M_i) = \nabla \varphi(\xi_i) \cdot (v_i - v^M_i)$, for $\xi_i=(1-\theta)v_i+\theta v^M_i,\,\theta\in(0,1)$. Thanks to Lemma \ref{lemma2} with $C=\max_i C_i \| v_i \|_{H^r_p(I_\z)}$ we have
		\[
		III \leq  \dfrac{1}{N} \sum_{i=1}^{N} \| \nabla\varphi(\xi_i)  \|_{L^2_p(I_\z)} \|v_i - v^M_i  \|_{L^2_p(I_\z)} \leq \dfrac{C}{M^r} \bigg( \dfrac{1}{N} \sum_{i=1}^{N} \| \nabla\varphi(\xi_i)  \|_{L^2_p(I_\z)} \bigg).
		\]
	\end{proof}
\end{theorem}
%We observe that Theorem \ref{theorem:error_estimate} with the choice $\varphi(v)=v^n, \, n\in\R$ gives an estimate for the moments of the distribution.

Next, we introduce a uniform grid in the domain $\V$, with $\Delta v > 0$ width of the cell, and we denote by $\Sv \geq 0$ a smoothing function such that
\[
\Delta v \int_\V \Sv (v) dv = 1.
\] 
We consider the approximations of the density function obtained by
\[
f_{\epsilon,N,\Delta v}(t,v,\z) = \dfrac{1}{N} \sum_{i=1}^N \Sv(v-v_i(t,\z)) \qquad f^M_{\epsilon,N,\Delta v}(t,v,\z) = \dfrac{1}{N} \sum_{i=1}^N \Sv(v-v^M_i(t,\z)),
\]
observing that the standard histogram reconstruction corresponds to the choice $\Sv(v)=\chi(|v|\leq\Delta v/2)/\Delta v$. Defining
\[
f_{\epsilon,\Delta v}(t,v,\z) = \int_\V \Sv(v-w)f(t,w,\z)dw,
\]
we have the following result
\begin{theorem} \label{theorem:error_f}
	The error introduced by the reconstruction of the distribution in the DSMC-sG method, in the grazing limit $\epsilon\rightarrow 0$, satisfies
	\[
	\begin{split}
		\|  f(t,v,\z) -  f^M_{\epsilon,N,\Delta v}(t,v,\z)  \|_{L^2(\V,L^2_p(I_\z))} \leq & \|B_{f_\epsilon}\|_{L^2_p(I_\z)}(\Delta v)^q + \dfrac{\|\sigma_{\Sv}\|_{L^2_p(I_\z)}}{N^{1/2}} \\
		& + \dfrac{C}{M^r} \bigg( \dfrac{1}{N} 	\sum_{i=1}^{N} \| \nabla \Sv(\xi_i) \|_{L^2_p(I_\z)} \bigg),
	\end{split}
	\]
	where $q>0$, $B_{f_\epsilon}>0$ is a constant, $C>0$ is a constant independent on $M$ and $\xi_i=(1-\theta)v_i+\theta v^M_i,\,\theta\in(0,1)$.
	\begin{proof}
		Thanks to the triangular inequality we have 
		\begin{equation*}
			\begin{split}
				\| f - f^M_{\epsilon,N,\Delta v} \|_{L^2(V,L^2_p(I_\z))}\quad&\leq\quad  \underbrace{  \| f - f_\epsilon \|_{L^2(V,L^2(I_\z))}   }_{I} \quad + \quad \underbrace{  \| f_\epsilon - f_{\epsilon,\Delta v}  \|_{L^2(V,L^2_p(I_\z))}   }_{II} \\
				& +\quad \underbrace{  \| f_{\epsilon,\Delta v} - f_{\epsilon,N,\Delta v} \|_{L^2(V,L^2_p(I_\z))}   }_{III} \quad + \quad \underbrace{  \|  f_{\epsilon,N,\Delta v} -  f^M_{\epsilon,N,\Delta v} \|_{L^2(V,L^2_p(I_\z))}   }_{IV} .
			\end{split}
		\end{equation*} 
	In the limit $\epsilon\rightarrow0$ we have shown that $f_\epsilon\rightarrow f$, so the first term vanishes. The second term represents the error introduced by the density reconstruction and is bounded by
	\[
	II \leq  \|B_{f_\epsilon}\|_{L^2_p(I_\z)}(\Delta v)^q,
	\]
	where $q>0$ depends on the accuracy of the reconstruction.
	For the last two terms, we observe that 
	\[
	f_{\epsilon,\Delta v} = \left\langle \varphi, f_\epsilon \right\rangle \qquad f_{\epsilon,N,\Delta v} = \left\langle \varphi, f_{\epsilon,N} \right\rangle \qquad f^M_{\epsilon,N,\Delta v} = \left\langle \varphi, f^M_{\epsilon,N} \right\rangle
	\]
	with $\varphi(\cdot)=\Sv(v - \cdot)$. Hence, we can apply the result of Theorem \ref{theorem:error_estimate} with the just mentioned choice for $\varphi$. 
	\end{proof}
\end{theorem}

\section{Numerical results}
 \label{sec:numerics}
 
In this section, we present several numerical tests for the DSMC-sG scheme for the non-Maxwellian models with uncertainties described in Section \ref{sec:nonM_models}. In all the subsequent tests we will consider $N= 10^5$ agents and the densities are reconstructed through standard histograms. 

%In particular, we consider 1D uncertainty in the kinetic models for gambling and wealth distribution, and uncorrelated $2$D uncertainty in the traffic model, corresponding to $\z=(z_1,z_2)$ where $\z\sim p_1(z_1)p_2(z_2)$. In the first two cases we consider $\delta=\delta(\z)$ in the interaction kernel \eqref{eq:B_gamb}, in the last we take $\mu=\mu(z_1)$ in the interaction term defined in  \eqref{eq:Inter_traf} and $\alpha=\alpha(z_2)$ in the kernel defined in \eqref{eq:B_traf}. 

In more detail, we first check the consistency of the DSMC-sG approximation of the Boltzmann equation with the exact equilibrium distribution for the kinetic model for gambling. Then we test the convergence to the equilibrium of the Fokker-Planck equation for the discussed wealth distribution and the traffic models. 

In the binary interactions \eqref{eq:interactionDSMC} we introduced the indicator function $\chi(\cdot)$. This term may deteriorate the overall convergence of the DSMC-sG scheme. Coherently with the approach proposed in \cite{pareschi2020JCP}, we introduce the following regularisation 
%of the acceptance-rejection technique has been considered, and the indicator function has been replaced with a sigmoid $K(\beta(\cdot))$ dependent on a parameter $\beta>0$

%Then, we check for the convergence of the scheme. However, as already observed in \cite{pareschi2020JCP}, we essentially loose the spectral accuracy of the scheme due to the presence of the indicator function at the particle level in \eqref{eq:interactionDSMC}.
 
 %To overcome this problem, we follow the ideas in the aforementioned work, where a regularisation of the acceptance-rejection technique has been considered, and the indicator function has been replaced with a sigmoid $K(\beta(\cdot))$ dependent on a parameter $\beta>0$
\begin{equation}
	\label{eq:bin_int_sigmoid}
	\begin{split}
		& v_{i}'(\z,t) = v_{i} + K \left( \beta (\Sigma\xi - B(v_i,w_j,\z)) \right) ( \epsilon I_1(v_i,w_j,\z) + D_1(v_{i})\eta_\epsilon ) \\
		& w_{j}'(\z,t) = w_{j} + K \left( \beta (\Sigma\xi - B(v_i,w_j,\z)) \right) ( \epsilon I_2(v_i,w_j,\z) + D_2(w_{j})\eta_\epsilon ), 
	\end{split}
\end{equation}
where  $K(\beta(\cdot))$ is a sigmoid function dependent on the parameter $\beta>0$. In particular, we consider
\begin{equation}
	\label{eq:K}
	K \left( \beta (\Sigma\xi - B(v_i,w_j,\z)) \right)=\frac{1}{2}\left(1+\tanh\left(-\beta\left(\Sigma\xi-B(v_i,w_j,\z)\right)\right)  \right). 
\end{equation} 
 {With this choice, we note that $\beta\gg1$ is associated with a sharp sigmoid function, on the contrary, a smaller $\beta$ is linked to a smoother sigmoid. We will return to the influence of the parameter $\beta$ in the following.}

This regularisation induces a different evolution of the relevant observables. Consequently, to keep the exact time evolution of the first two moments together with the spectral convergence,  we couple the DSMC-sG scheme with a scaling process of the form 
\be \label{eq:rescaling}
\begin{split}
	& v_{i}''(\z,t)=\left(v_{i}'(\z,t) - V_{K}(t,\z)\right)\sqrt{\dfrac{E_{\textrm{FP}}(t,\z)}{E_{K}(t,\z)}} + V_{\textrm{FP}}(t,\z) \\
	& w_{j}''(\z,t)=\left(w_{j}'(\z,t) - V_{K}(t,\z)\right)\sqrt{\dfrac{E_{\textrm{FP}}(t,\z)}{E_{K}(t,\z)}} + V_{\textrm{FP}}(t,\z),
\end{split}
\ee
where $V_{\textrm{FP}}(t,\z)$ and $E_{\textrm{FP}}(t,\z)$ are, respectively, the mean velocity and the energy computed from the corresponding surrogate Fokker-Planck model. Similarly, we indicated with $V_{K}(t,\z)$ and $E_{K}(t,\z)$ mean velocity and energy of the Boltzmann-type model with sigmoid function $K(\cdot)$ \eqref{eq:K} in equation \eqref{eq:bin_int_sigmoid}. The computation of $V_{\textrm{FP}}(t,\z)$ and $E_{\textrm{FP}}(t,\z)$ follows from the model \eqref{eq:FP_general} that is solved through standard sG method and for which we can guarantee sufficient regularity under the assumptions of Theorem \ref{th:2}. We highlight how the additional scaling process \eqref{eq:rescaling} will be consistent with the original Boltzmann-type model in the regime $\epsilon \ll 1$.

For clarity purposes, in the rest of the section we will indicate the numerical solution of the Fokker-Planck model $f_{\textrm{FP}}(t,v,\z)$ and the numerical solution of the Boltzmann-type model as $f_{\epsilon}(t,v,\z)$, whereas the solution of the Boltzmann-type model with additional scaling process \eqref{eq:rescaling} will be denoted by $\tilde f_\epsilon(t,v,\z)$. For the approximation of the Fokker-Planck model we will implement a standard sG collocation method based on a semi-implicit scheme presented in \cite{Pareschi2018} and further studied in \cite{Dimarco2017,tosin2021MCRF,zanella2020MCS}. 

%In this second case We adopt the following hybrid scheme, presented here in the general form
%\begin{algorithm} \label{DSMC_sG_VHS_hybrid}
%	(Hybrid DSMC-sG).
%	
%	\begin{itemize}
%		\item fix $\epsilon = \Delta t \ll 1$; 
%		\item compute the initial gPC expansion $\{v_i^{M,0}$, $i = 1,\dots,N\}$, from the deterministic initial distribution $f_0(v)$;
%		\item for $n=0$ to $\nt-1$ in Algorithm \ref{DSMC_sG_VHS}:
%		\begin{itemize}
%			\item compute the solutions of the Fokker-Planck model $g^n(t,v,\z_h)$ in each collocation point $h = 0,\dots,M$; 
%			\item compute $V_\chi$, $E_\chi$ from the ensamble $g^n(t,v,\z)$; 
%			\item perform point $(2)$ of Algorithm \ref{DSMC_sG_VHS} with $\hat v_{i,h}^{n+1}$, $\hat w_{j,h}^{n+1}$ for the selected particles as in \eqref{eq:bin_int_sigmoid};
%			\item rescale the post-interaction velocities according to \eqref{eq:rescaling}.
%		\end{itemize}
%	\end{itemize}  
%\end{algorithm}

\subsection{Test 1: gambling} \label{subsec:test_gambling}

\begin{figure}
	\centering
	\includegraphics[width = 0.45\linewidth]{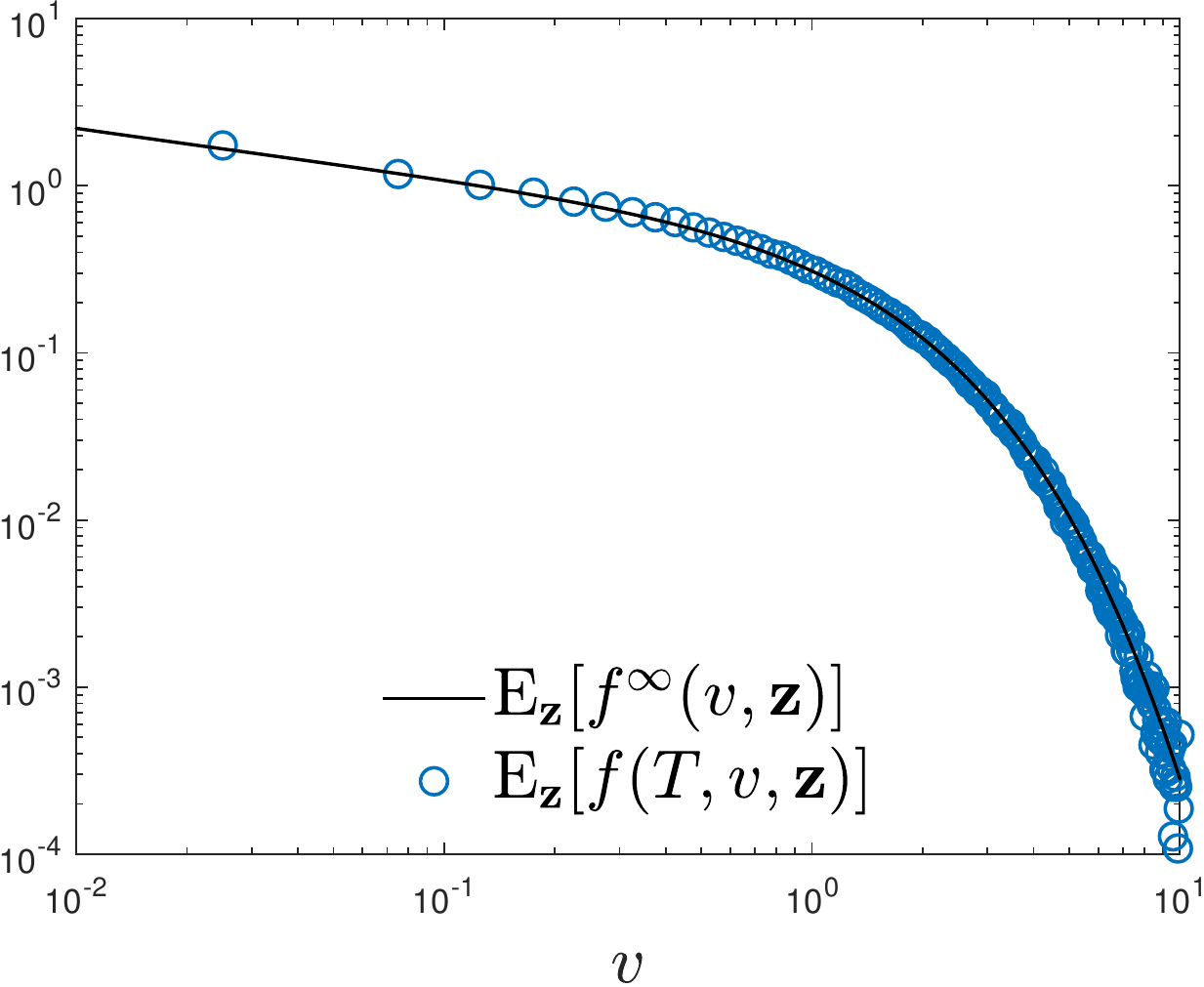}
	\includegraphics[width = 0.45\linewidth]{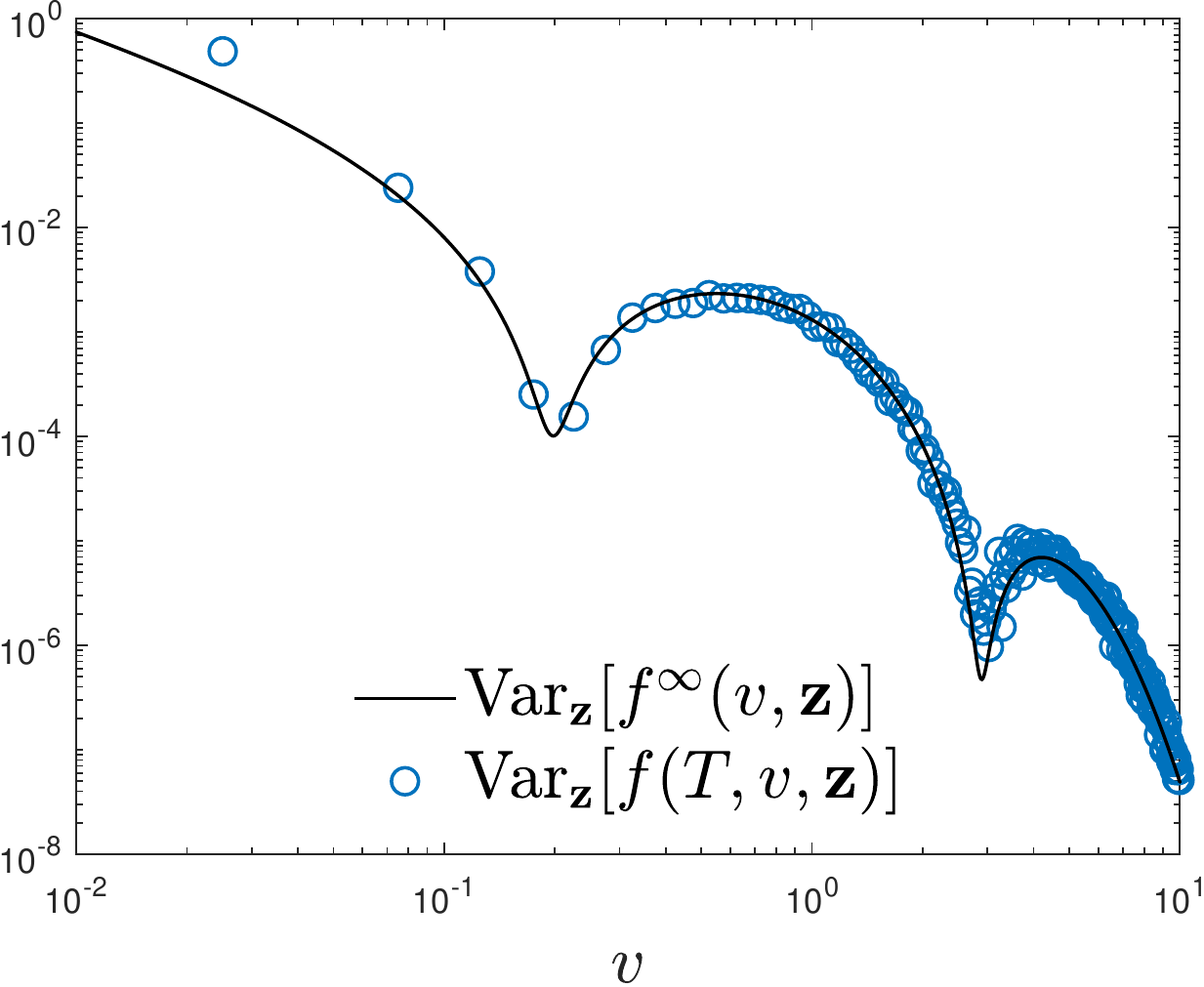}
	\caption{\textbf{Test 1}. Comparison of $f^\infty(v,\z)$ in \eqref{eq:finf_t1} and DSMC-sG approximation of $f(T,v,\z)$ at time $T=10$, in log-log scale, in terms of expectation (left) and variance (right) with respect to uncertain parameter $\delta(z) = \frac{1}{2}z$, $z\sim\mathcal{U}([0,1])$. We consider $N=10^5$ particles with $M=5$ Galerkin projections, time step $\Delta t = 0.1$ and initial density \eqref{eq:f0_t1_g}.}
	\label{fig:test1_gambling}
\end{figure}
We consider the kinetic model for gambling with 1D uncertainty in the collisional kernel. We choose $\delta(z) = \frac{1}{2}z$ and $z\sim\mathcal{U}([0,1])$ in \eqref{eq:B_gamb} fixing $\kappa = 1$. %, $\omega\sim\mathcal{U}([0,1])$ in the binary interaction. 
Since the random parameter is uniformly distributed, we use the Legendre polynomials in the gPC expansion. In all the simulations, we consider $M=5$ Galerkin projections and the time frame $[0,T]$ with $T = 10$ discretised with time step $\Delta t = 0.1$. The kinetic density is reconstructed in the interval $[0,10]$ with $\Delta v=0.05$. We consider the deterministic initial distribution
\begin{equation}
\label{eq:f0_t1_g}
f(0,v,\z)=
\begin{cases}
\dfrac{1}{2} & v \in [0,2] \\
0                 & \textrm{elsewhere}.  
\end{cases}
\end{equation}
In this test, we highlight that  the equilibrium solution of the Boltzmann-type model can be computed exactly and $\epsilon = 1$ and reads
\begin{equation}
\label{eq:finf_t1}
f^\infty(v,\z) = \dfrac{(1-\delta)^{1-\delta}}{\Gamma(1-\delta)}v^{-\delta} \exp\{-(1-\delta)v\},
\end{equation}
see Section \ref{subsec:gambling}. 
In Figure \ref{fig:test1_gambling}, we report expected value and variance with respect to the  random parameter $z$ of the DSMC-sG approximation. We may observe the good agreement of the considered quantities of interest with the analytical ones. 

\subsection{Test 2: wealth distribution}
We consider now the kinetic model for the wealth distribution described in Section \ref{subsec:wealth}. In particular, we consider the case where the interaction kernel \eqref{eq:kernel_wealth} is characterized by  $\delta(\z) = \z\sim\mathcal{U}([0,1])$ and $\kappa = 1$. Therefore, we adopt the Legendre polynomials in the gPC expansion of the velocities. In all the results of this test, we consider a background uniformly distributed as $\mathcal{E}\sim\mathcal{U}([0.9, 1.1])$. Furthermore, we consider the following deterministic initial distribution
\begin{equation}
\label{eq:f0_t1}
f(0,v,\z)=
\begin{cases}
\dfrac{1}{2} & v \in [0,2] \\
0                 & \textrm{elsewhere}.  
\end{cases}
\end{equation}
In Figure \ref{fig:test1_wealth} we show expectation and variance of $f_\epsilon(t,v,\z)$ computed through DSMC-sG method with respect to the analytical equilibrium distribution of the Fokker-Planck model \eqref{eq:equilibrium_wealth}, for various $\epsilon  = 5 \times 10^{-2}, 10^{-1}, 5 \times 10^{-1}$. In the last picture we report also the behavior of the expected mean wealth $\mathbb E_{\z}[V_\epsilon(t,\z)]$ for the introduced values of $\epsilon$. We consider $M=5$ Galerkin projections, $\Delta t=\epsilon/10$, $\lambda=\sigma^2=0.5$ and time frame $[0,T]$ with $T = 10$. The kinetic density is reconstructed through standard histogram over the interval $[0,10]$ with $\Delta v = 0.05$. 

%In the top left panel of Figure \ref{fig:test1_wealth}, we may observe the accordance between the expectations of the exact steady state of the Fokker Planck equation \eqref{eq:equilibrium_wealth} and the DSMC-sG approximation of $f_{\epsilon}(T,v,\z)$, for decreasing values of the Knudsen number $\epsilon$ and fixed time $T=10$. In the bottom centred panel, we show the time evolution of the first order moment $V_{\epsilon}(t,\z)=\int_V v f_{\epsilon}(t,v,\z)dv$ together with the time evolution of the first order moment $V(t,\z)$ of $f(v,t,\z)$, given by the numerical solution of equation \eqref{eq:FP_wealth} \textcolor{red}{se ne parla in dettaglio dopo, accennare già qui?}. 

%Since there is a good agreement, in the following, we fix $\epsilon=0.1$. In the top right panel of Figure \ref{fig:test1_wealth}, we show the time evolution of $\mathbb{E}_\z[f_{\epsilon}(t^n,v,\z)]$ at fixed times $t^n=n,\,n=0,\dots,10$.
\begin{figure}
	\centering
	\includegraphics[width = 0.31\linewidth]{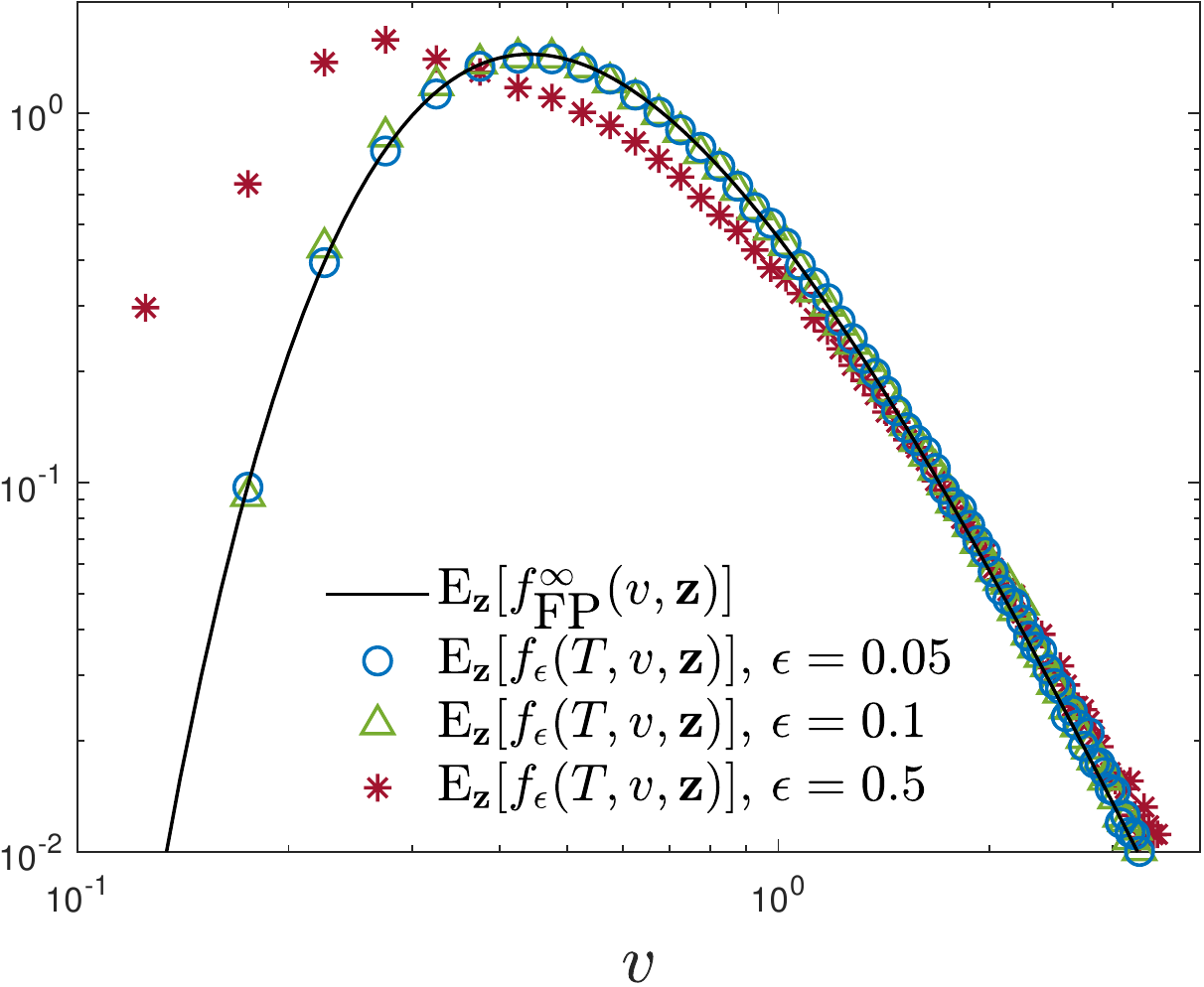}
	\includegraphics[width = 0.31\linewidth]{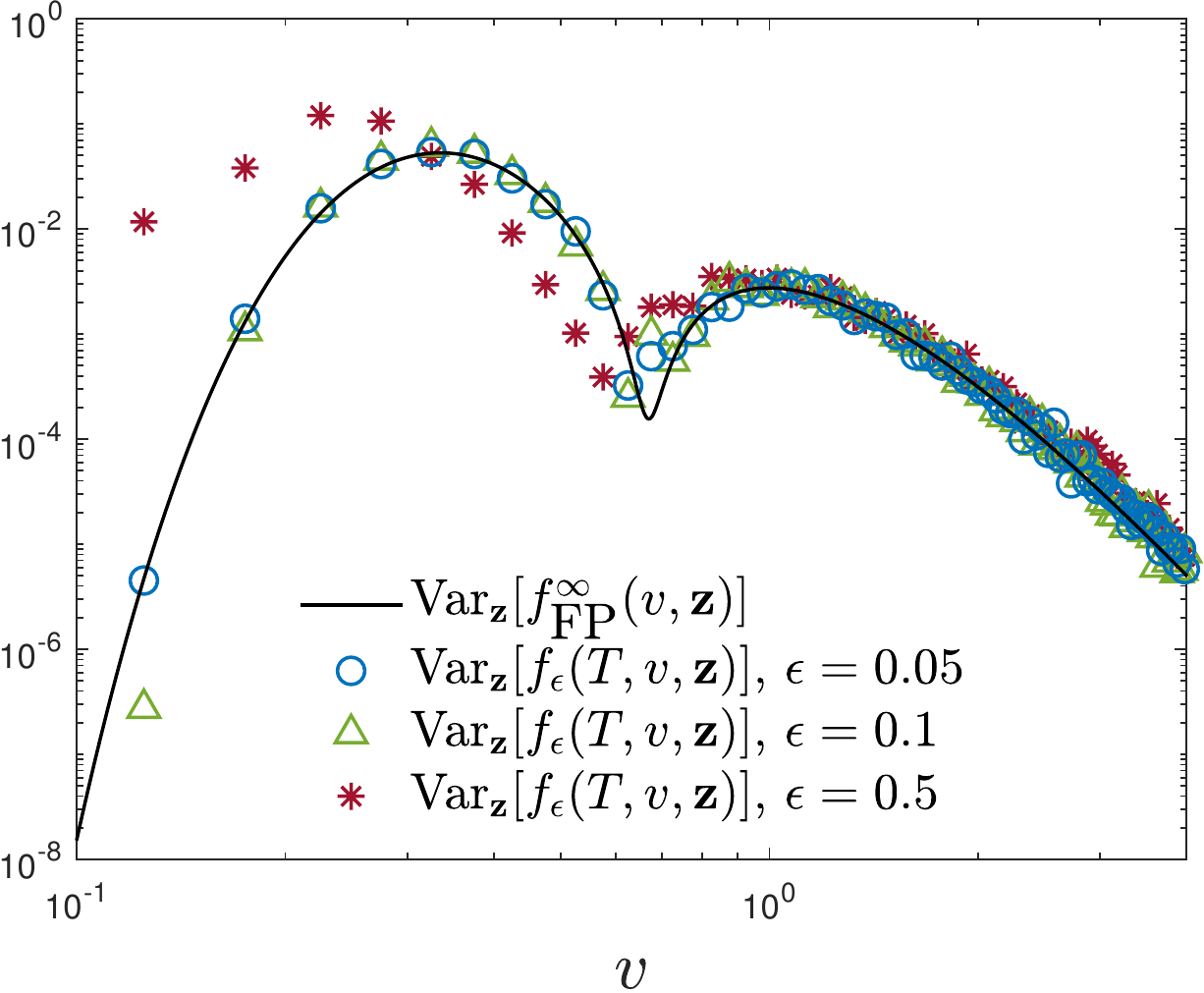}
	\includegraphics[width = 0.31\linewidth]{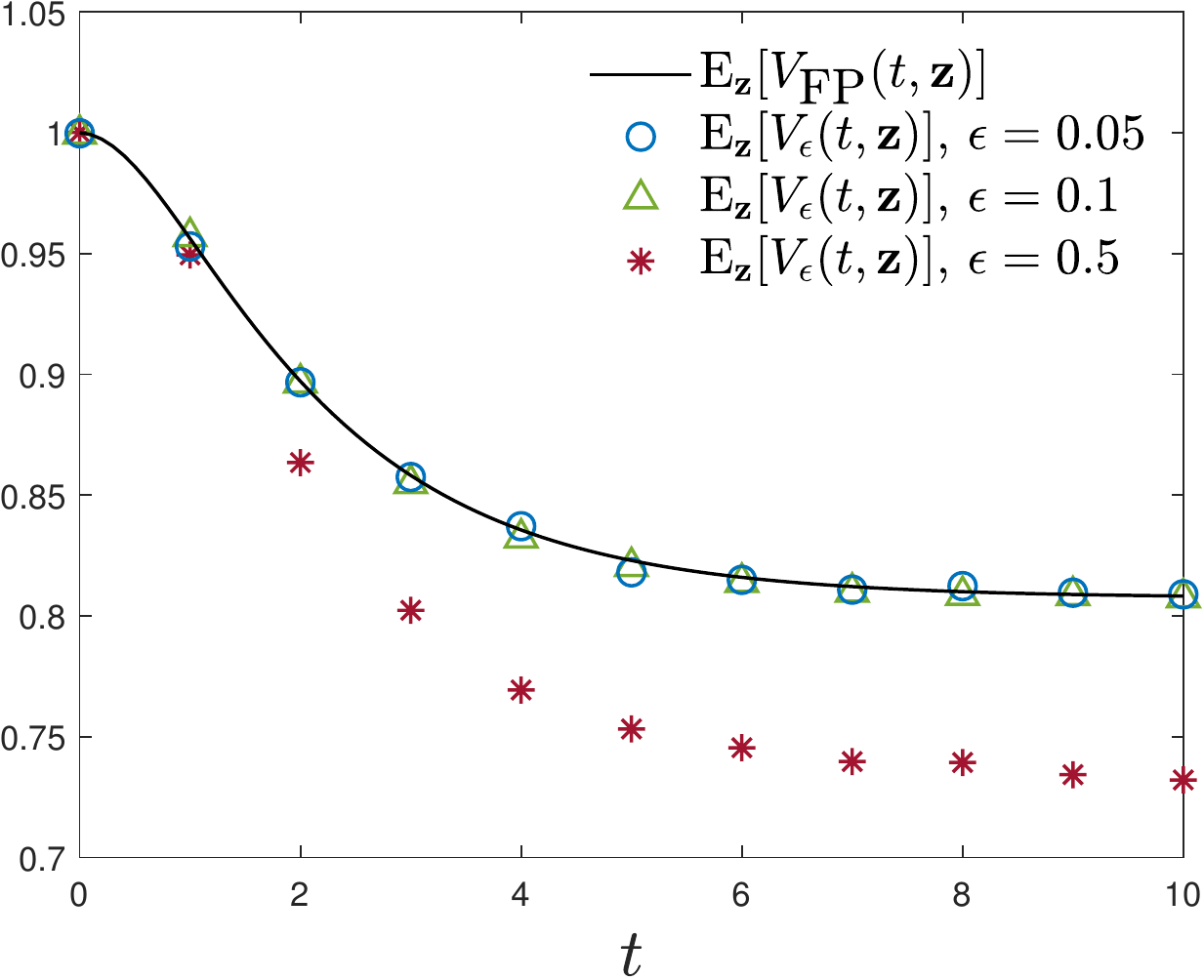}
	\caption{\textbf{Test 2}. Comparison of $f^\infty_{\textrm{FP}}(v,\z)$ and the DSMC-sG approximation of $f_\epsilon(T,v,\z)$ at time $T = 10$, in log-log scale, in terms of expectation (left) and variance (center) for several $\epsilon = 5 \times 10^{-2}, 10^{-1}, 5 \times 10^{-1}$. We report (right) the time evolution of $\mathbb E_\z[V_\epsilon(t,\z)]$ for several $\epsilon$ and $\mathbb E_{\z}[V_{\textrm{FP}}(t,\z)]$. We consider $N=10^5$ particles with $M=5$ Galerkin projections and $\Delta t = \epsilon/10$. The random parameter is uniformly distributed $\delta(\z) = \z\sim\mathcal{U}([0,1])$, and we fix $\kappa=1$, $\lambda=\sigma^2=0.5$. The background $\mathcal{E}$ is a uniform distribution $\mathcal{U}([0.9, 1.1])$.}
	\label{fig:test1_wealth}
\end{figure}

In order to show spectral convergence property of the scheme, we consider a reference DSMC-sG evolution of $E_\epsilon(t,\z)$ obtained with $\epsilon = 0.1$, $N = 10^5$, $\Delta t = 0.1$ and sG scheme up to order $M = 50$. We store the collisional tree generating the reference solution and we check the $L^2$ convergence of the scheme. In Figure \eqref{fig:test2_wealth} we present the decay of the $L^2$ error for increasing $M$ obtained from the initial distribution \eqref{eq:f0_t1}. If we consider the original binary dynamics \eqref{eq:interactionDSMC}, even if the expectation is well described, it can be observed that the spectral accuracy of the method is lost due to discontinuity of the indicator function $\chi (\cdot)$. The same test performed for the binary dynamics \eqref{eq:bin_int_sigmoid} recovers spectral accuracy. For increasing $\beta\gg 0$ the convergence of the scheme is deteriorated, since we approximate a step function. 
%Anyway, the resulting approximation is not coherent with the original model in terms of the main physical quantities. 
%Then we check for the convergence of the scheme in the space of the random parameter, looking at the $L^2$ error in the evaluation of the second order moment. Adopting the same choices and approach of the previous section, we may observe from Figure \ref{fig:test2_wealth} that the indicator function deteriorates the convergence (left panel) and the regularisation allows us to recover the spectral convergence for decreasing values of $\beta$ (right panel).
\begin{figure}
	\centering
	\includegraphics[width = 0.45\linewidth]{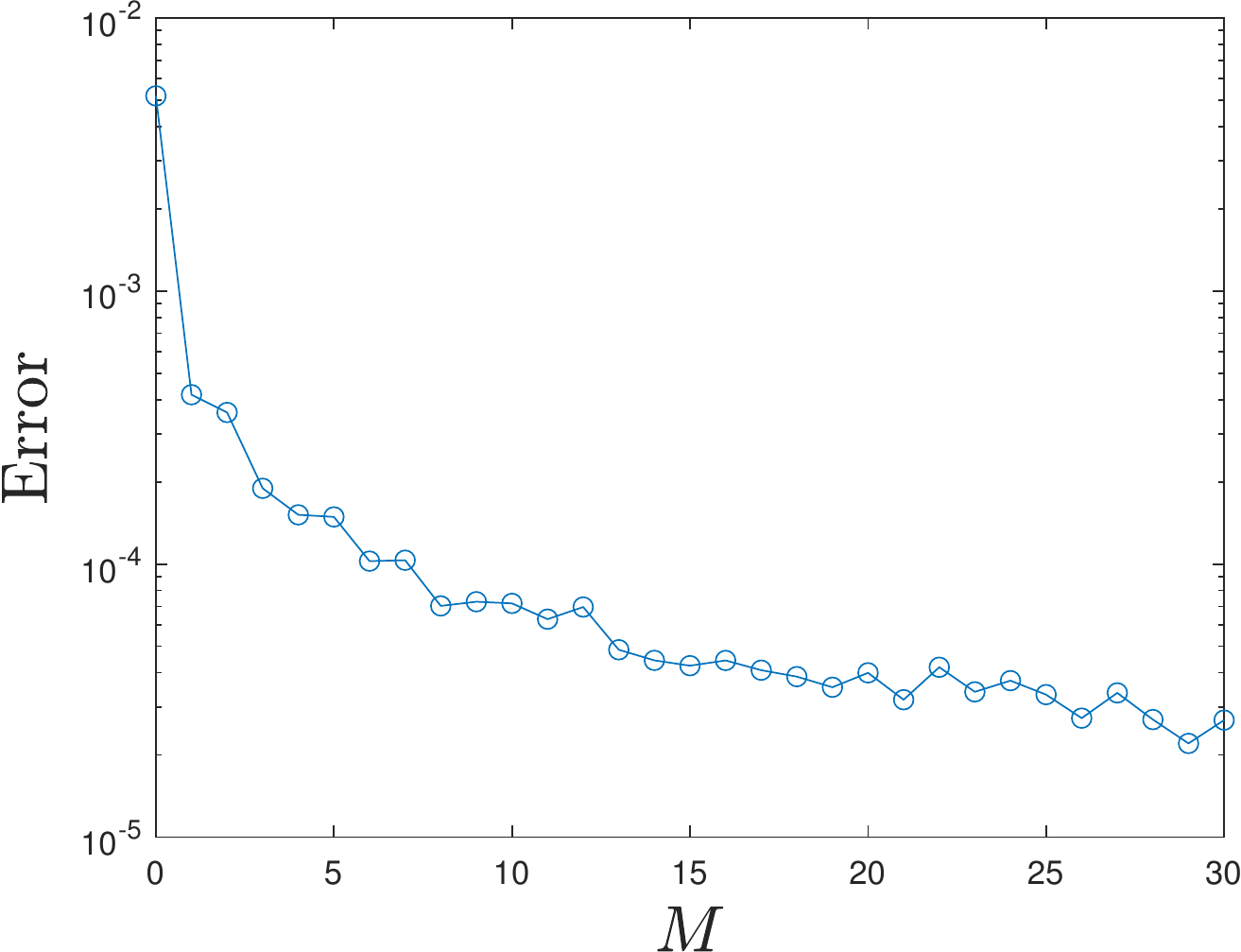}
	\includegraphics[width = 0.45\linewidth]{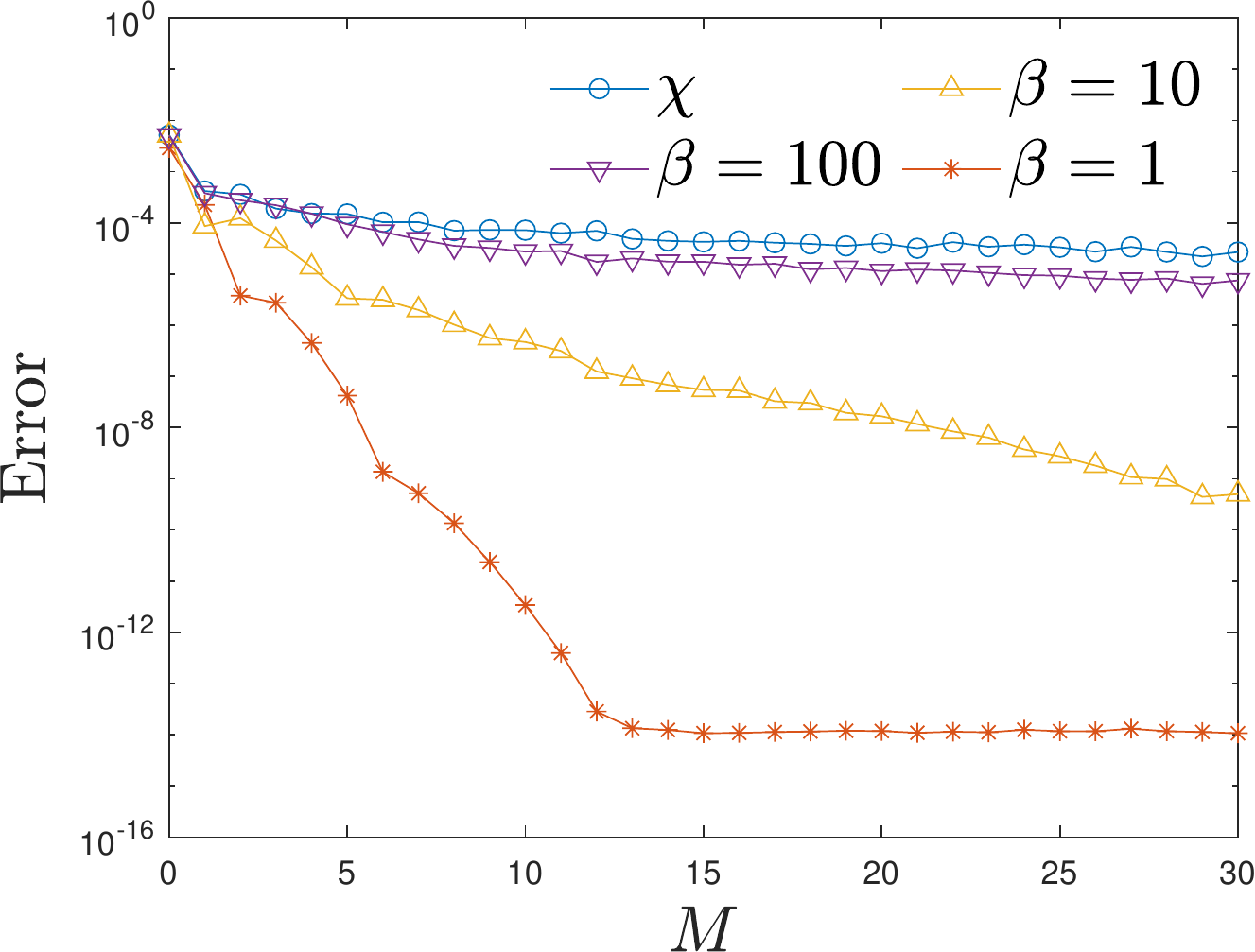}
	\caption{\textbf{Test 2}. Convergence of the $L^2$ error of the DSMC-sG scheme, where the binary interaction dynamics are given by \eqref{eq:interactionDSMC} (left) or by \eqref{eq:bin_int_sigmoid} (right), in the case of model for wealth distribution with uncertain kernel. We consider $N=10^5$, $\Delta t = 0.01$ and $\epsilon=0.1$. We fix $\kappa=1$, $\lambda=\sigma^2=0.5$. Reference solution computed with $M= 50$.}
	\label{fig:test2_wealth}
\end{figure}

Coupling now \eqref{eq:bin_int_sigmoid} with the process \eqref{eq:rescaling}, we recover qualitatively consistent approximation of the evolution of relevant quantities of interest together with spectral convergence for moderate values of $\beta >0$, see Figure \ref{fig:test3_wealth}. In this test we solve the Fokker-Planck model \eqref{eq:FP_wealth}. 

\begin{figure}
	\centering
	\includegraphics[width = 0.31\linewidth]{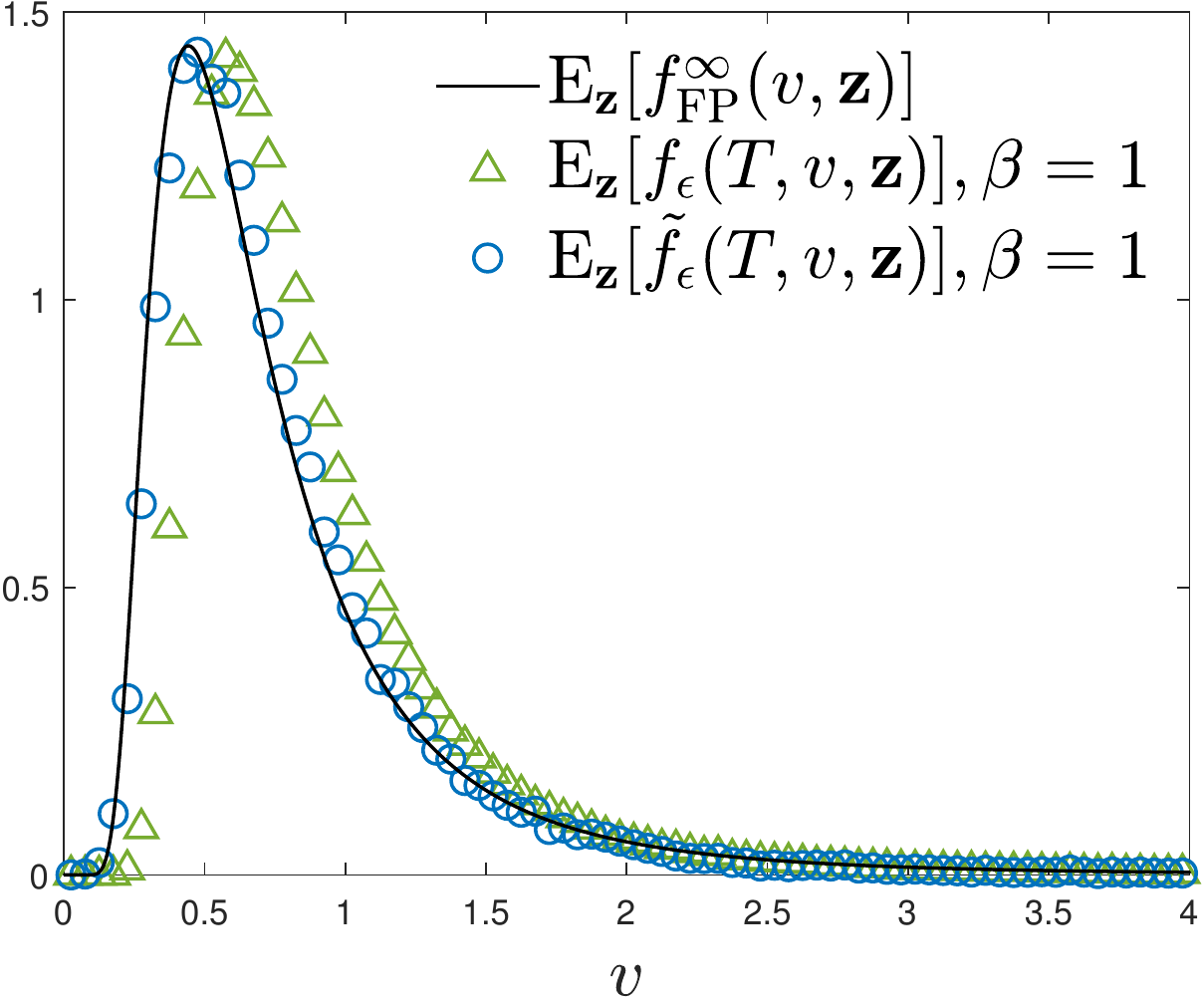}
	\includegraphics[width = 0.31\linewidth]{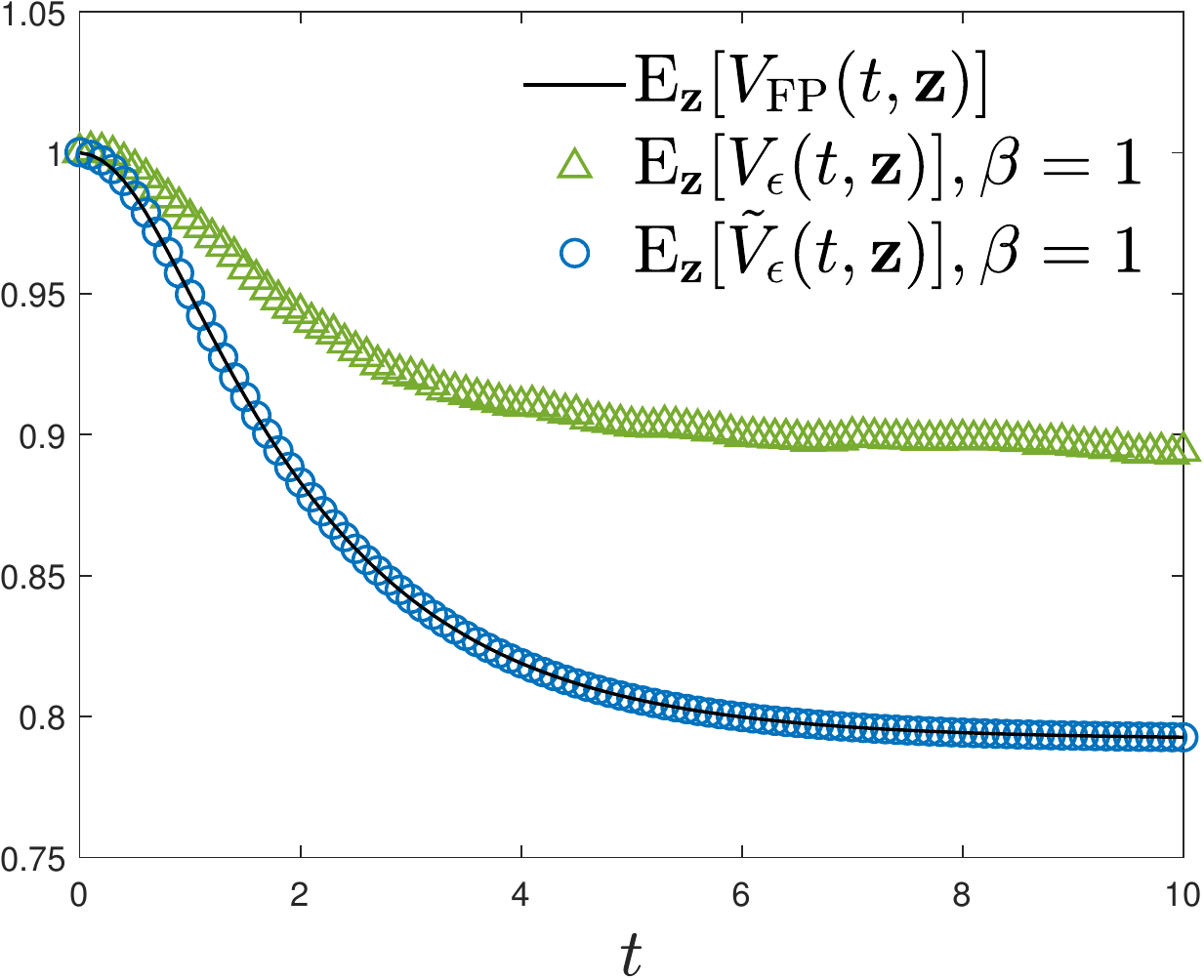}
	\includegraphics[width = 0.31\linewidth]{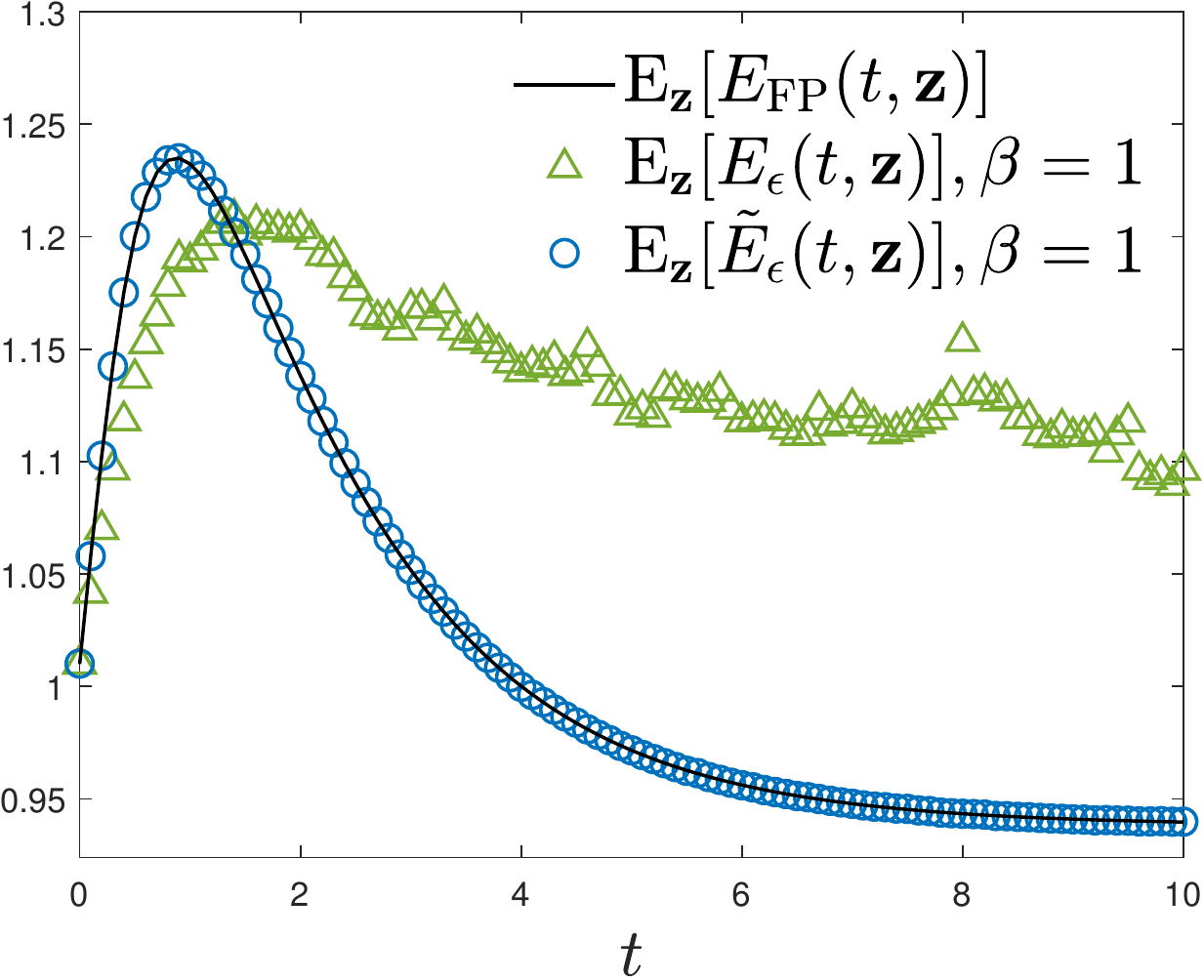}
	\caption{\textbf{Test 2}. Left: comparison of $f^\infty_{\textrm{FP}}$ in \eqref{eq:equilibrium_wealth} with the DSMC-sG approximation of $f_\epsilon(t,v,\z)$ (regularization without rescaling) and of $\tilde{f}_\epsilon(t,v,\z)$ (regularization with rescaling), in terms of expectation in $\z$.  Center and right: comparison of $V_{\textrm{FP}}, E_{\textrm{FP}}$ with the DSMC-sG approximation of $V_\epsilon(t,\z), E_\epsilon(t,\z)$ (regularization without rescaling) and of $\tilde{V}_\epsilon(t,\z),\tilde{E}_\epsilon(t,\z)$ (regularization with rescaling). We consider $N=10^5$, $M = 5$, $\Delta t = 0.01$ and $\epsilon=0.1$. We fix $\kappa=1$, $\lambda=\sigma^2=0.5$ }
	\label{fig:test3_wealth}
\end{figure}

%In the left panels of Figure \ref{fig:test3_wealth}, we display the expectation of the distribution function and the time evolution of the expectations of the mean velocity and energy with the regularisation, without the rescaling process, for $\beta=1$. As we expected, the results deviate from their original dynamics. In the right panels of Figure \ref{fig:test3_wealth}, we show the same quantities with the scaling process introduced in Algorithm \ref{DSMC_sG_VHS_hybrid}. We note that in these cases the results are in agreement with the dynamics derived from the Fokker-Planck equation, coherent with the Boltzmann-type one in the quasi invariant interaction limit. In addition to this, the regularised and rescaled model with $\beta=1$ exhibits spectral convergence, as observed before.

\subsection{Test 3: traffic flow}
In this last test, we consider the traffic model of Section \ref{subsec:traffic}, affected by an uncorrelated 2D random parameter $\z = (z_1,z_2)$ with $p(\z) = p_1(z_1)p_2(z_2)$. In particular, we consider $z_1$ affecting $\mu(z_1)$  in the interaction function $I(v,w,z_1)$ defined in \eqref{eq:Inter_traf} and $z_2$ affecting $\alpha(z_2)$ in the kernel $B(|v-w|,z_2)$ defined in \eqref{eq:B_traf}. 
Under these assumptions, the gPC expansion of the velocities $v_{i}(t,z_1,z_2)$, $i = 1,\dots,N$ reads
\begin{equation}
	v_{i}^{M_1,M_2}(z_1,z_2,t)=\sum_{h=0}^{M_1} \sum_{k=0}^{M_2}\hat{v}_{i,h,k}(t)\Phi^{(1)}_{h}(z_{1})\Phi^{(2)}_{k}(z_{2}),
\end{equation}
being $\{\Phi^{(1)}_{h}(z_{1})\}_{h=0}^{M_1}$ and $\{\Phi^{(2)}_{k}(z_{2})\}_{k=0}^{M_2}$ the polynomials orthogonal with respect to the distributions $p_1(z_1)$ and $p_2(z_2)$, respectively. Substituting $v_{i}^{M_1,M_2}(z_1,z_2,t)$ into the binary interaction \eqref{eq:bin_int_traf} and and proceeding as in Section \ref{subsec:DSMC}, we obtain 
\begin{equation}
	\begin{split}
		&\hat{v}_{i,h,k}'(t)=\hat{v}_{i,h,k}(t)+\hat{V}^{h,k}_{i,j} \\
		&\hat{w}_{j,h,k}'(t)=\hat{w}_{j,h,k}(t),
	\end{split}
\end{equation}
with the following collision matrix
\begin{equation}
	\begin{split}
		\hat{V}^{h,k}_{i,j}=&\int_{I_\z} \chi\left(\Sigma\xi<B(v_{i}^{M_1,M_2},w_{j}^{M_1,M_2},z_{2})\right)\\
		&\left(\gamma I(v_{i}^{M_1,M_2},w_{j}^{M_1,M_2},z_{1}) + D(v_{i}^{M_1,M_2};\rho)\eta\right)\Phi^{(1)}_{h}(z_{1})\Phi^{(2)}_{k}(z_{2})p_{1}(z_{1})p_{2}(z_{2})dz_{1}dz_{2}.
	\end{split}
\end{equation}   

We consider the following deterministic initial distribution
\begin{equation}
\label{eq:f0_t3}
f(0,v,\z)=
\begin{cases}
1 & v \in [0,1] \\
0                 & \textrm{elsewhere}.  
\end{cases}
\end{equation}

To assess the impact of the single uncertain parameters on the dynamics we first consider the case $\mu(z_1) = 1 + 2z_1$ and $\alpha(z_2) = 2z_2 $ with uncorrelated uncertainties  $z_1,z_2\sim \mathcal U([0,1])$. In Figure \ref{fig:2D_traffic} we show the DSMC-sG approximation of the solution of Fokker-Planck model for traffic \eqref{FP_nocontrol} in terms of expected value and variance in $\z = (z_1,z_2)$ of the distribution function and of the macroscopic quantities. We considered two different densities $\rho = 0.4$ and $\rho = 0.6$ and the Fokker-Planck is solved on a grid of $N_v = 51$ points such that $\Delta v = 0.02$ and $\Delta t = \Delta v/2$. As before, the DSMC-sG provides a good approximation in the limit $\epsilon \ll 1$ of the solution of the Fokker-Planck model. 

We show the $L^2$ convergence of the DSMC-sG scheme in Figure \ref{fig:2D_conv}. In details, we considered the case with binary interactions  \eqref{eq:interactionDSMC} in the left plot, whereas the case with regularization of the step function as in \eqref{eq:bin_int_sigmoid}, with $\beta = 0.01$, is presented in the right plot. The error has been computed with respect to a reference DSMC-sG evolution of $E_\epsilon(t,\z)$ with $\Delta t = \epsilon = 0.1$, $N = 10^5$, and $M_1 = M_2 = 50$. As before, in this test we store the collisional tree of the reference solution and we check $L^2$ convergence for increasing $M_1,M_2$. The error is presented here in $\log_{10}$ and we can clearly observe spectral accuracy in the case with regularization.   

\begin{figure}
	\centering
	\includegraphics[width = 0.31\linewidth]{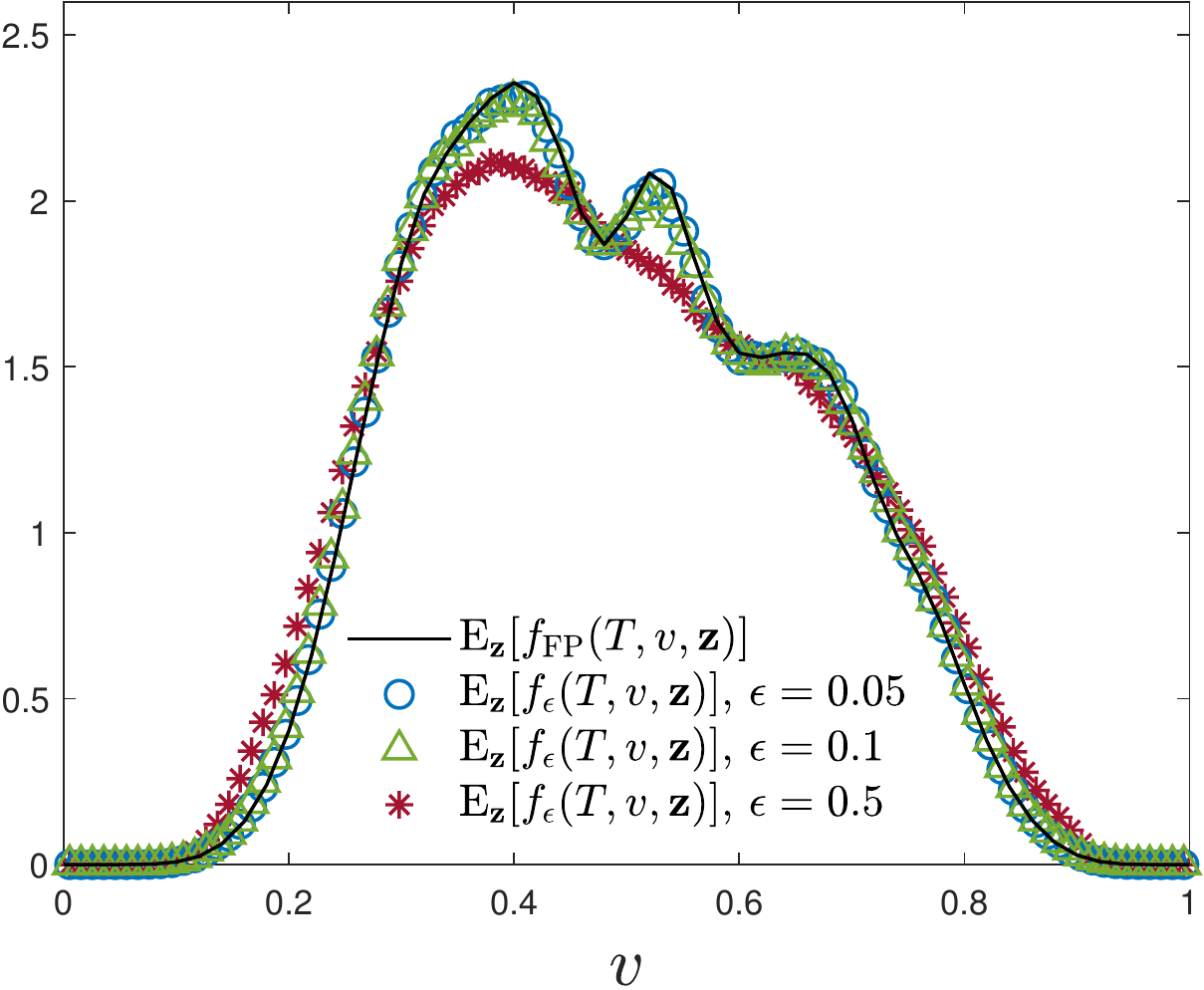}
	\includegraphics[width = 0.31\linewidth]{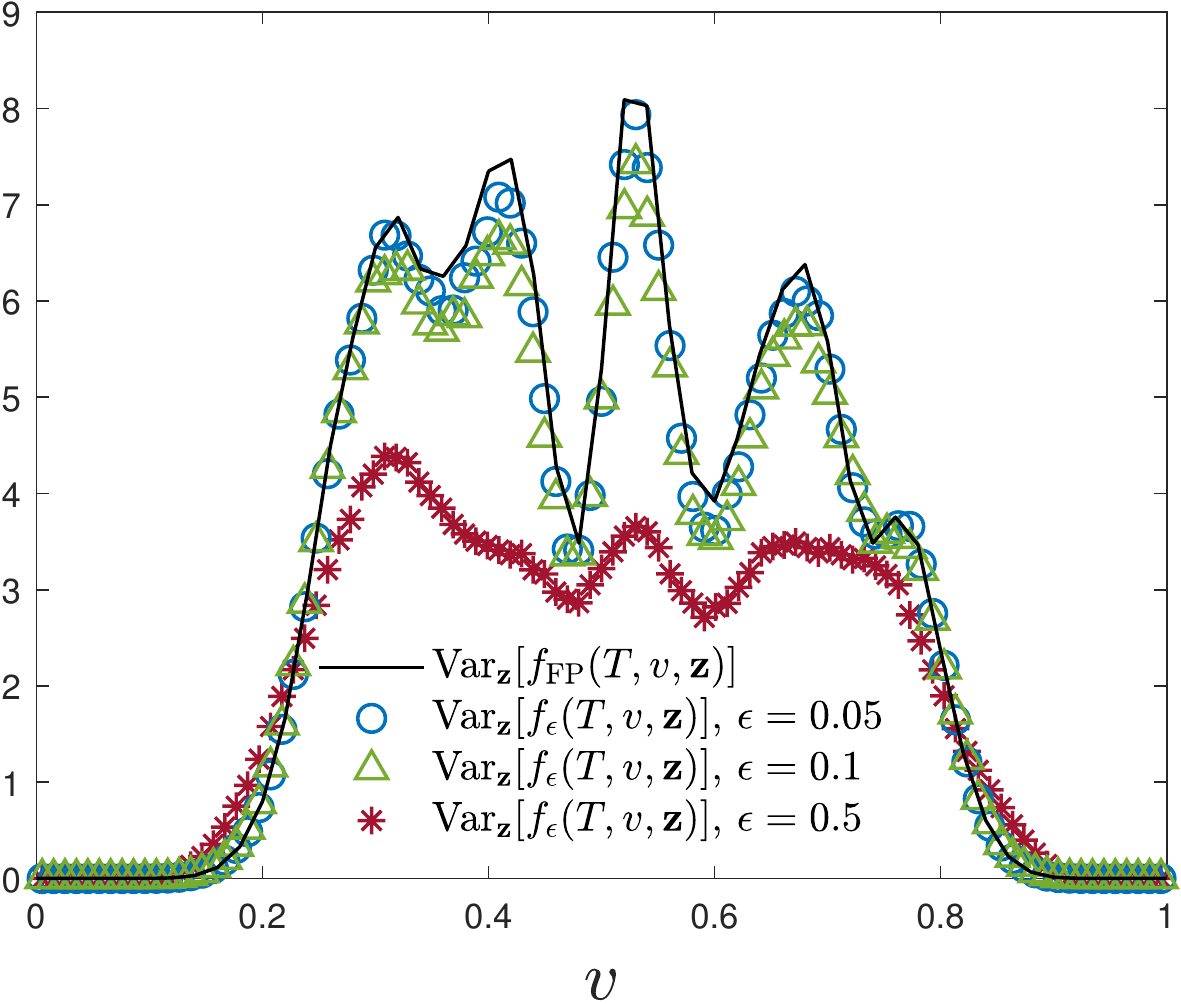}
	\includegraphics[width = 0.31\linewidth]{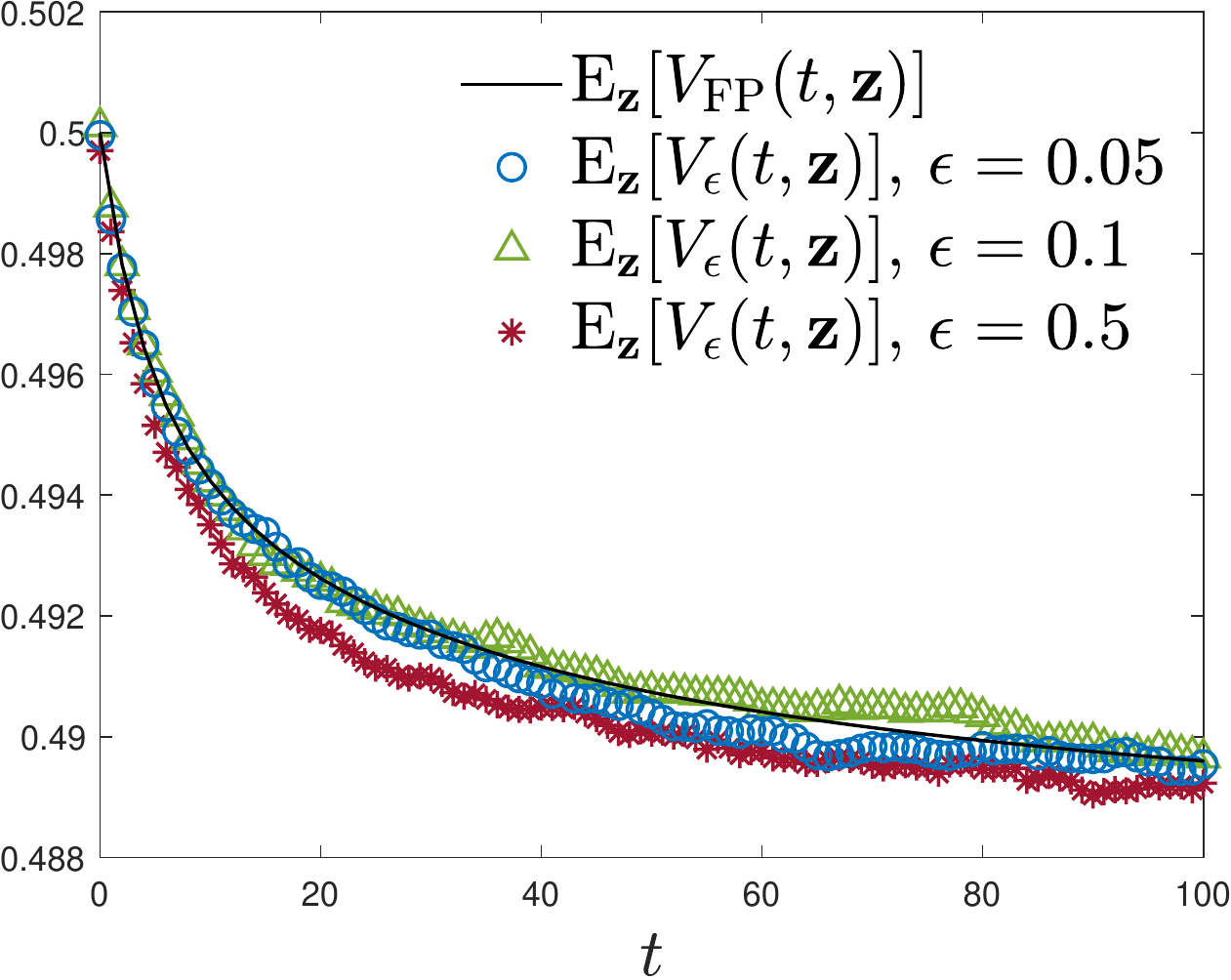} 
	\includegraphics[width = 0.31\linewidth]{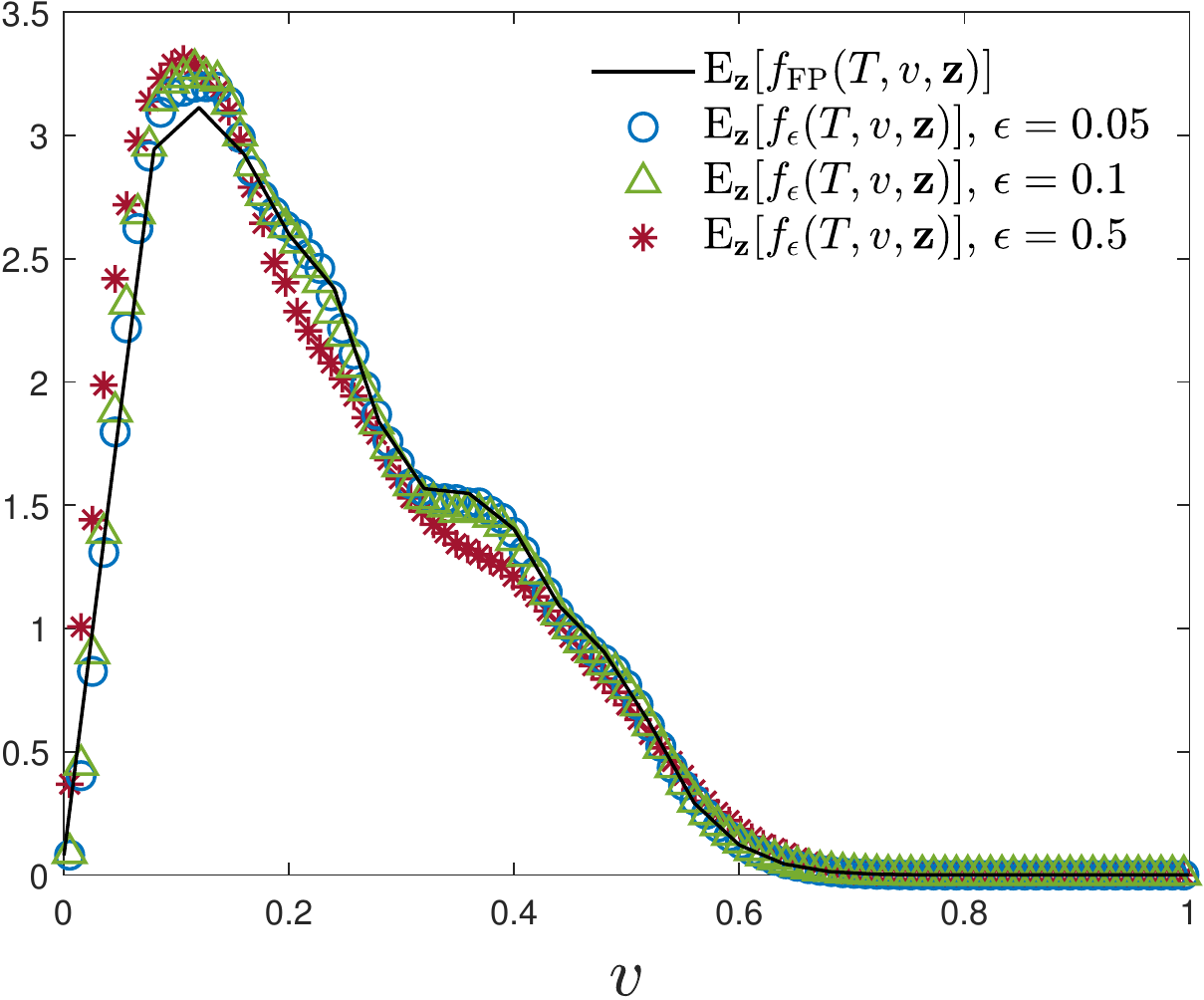}
	\includegraphics[width = 0.31\linewidth]{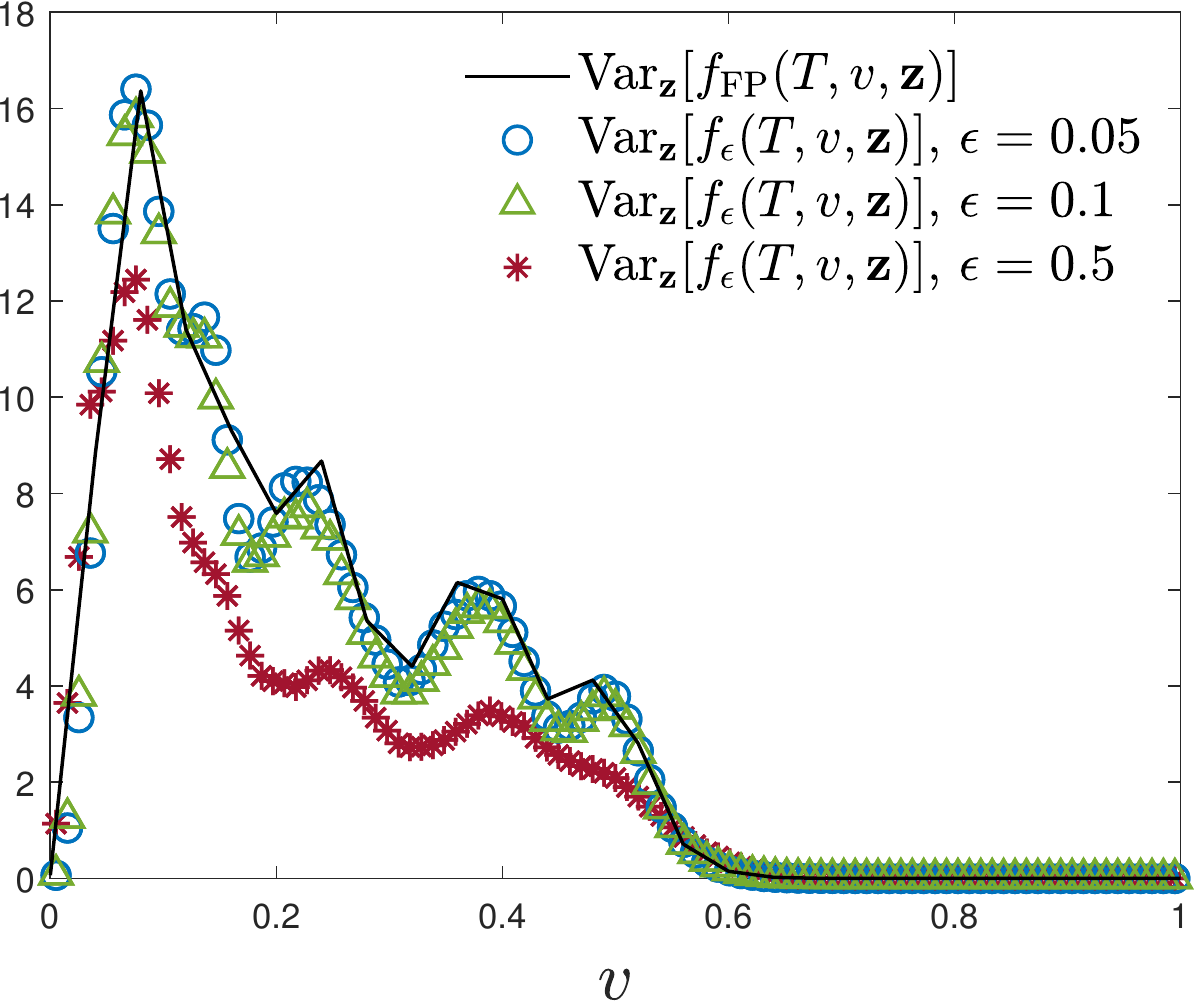}
	\includegraphics[width = 0.31\linewidth]{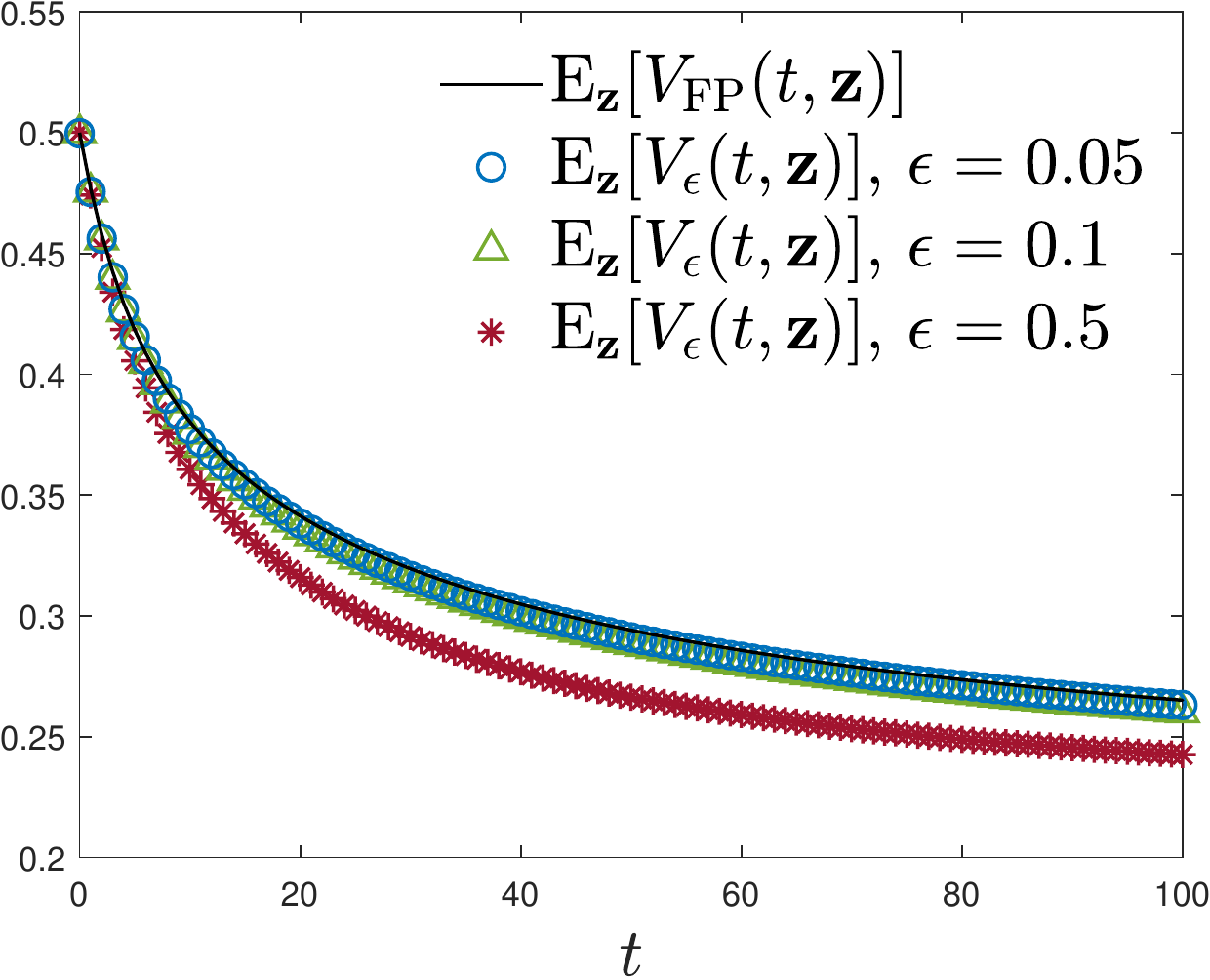}
	\caption{\textbf{Test 3}. Comparison of the numerical solution of the Fokker-Planck model \eqref{FP_nocontrol} with $\mu(z_1)= 1 + 2 z_1 $ and $\alpha(z_2) = 2z_2$, $z_1,z_2 \sim \mathcal U([0,1])$, with the DSMC-sG reconstruction of traffic distributions (first and second column) and of mean velocity (third column) for various $\epsilon = 0.05, 0.1,0.5$. We considered $\rho = 0.4$ (top row) and $\rho = 0.6$ (bottom row). We set for the DSMC-sG method $N = 10^5$, $M_1 = M_2 = 5$, $\Delta t = \epsilon$, and the deterministic solver for Fokker-Planck is such that $\Delta v = 0.02$ and $\Delta t = \Delta v/2$. The time frame is $t \in [0,T]$, $T = 300$. }
\label{fig:2D_traffic}
\end{figure}

%\begin{figure}
%	\centering
%	%\includegraphics[scale = 0.4]{Immagini/error_sG_z1_traffic_no_reg}
%	\includegraphics[scale = 0.4]{Immagini/error_sG_z2_traffic_no_reg}
%	\includegraphics[scale = 0.4]{Immagini/error_sG_z1_z2_traffic_no_reg}
%	\caption{\textcolor{red}{da sinistra destra: errore L2 sulla valutazione dell'energia nel caso con incertezza nell'interazione, con incertezza nel nucleo, e con doppia incertezza ($\rho=0.4$) }}
%	\label{fig:2D_traffic}
%\end{figure}

\begin{figure}
	\centering
	\includegraphics[width = 0.45\linewidth]{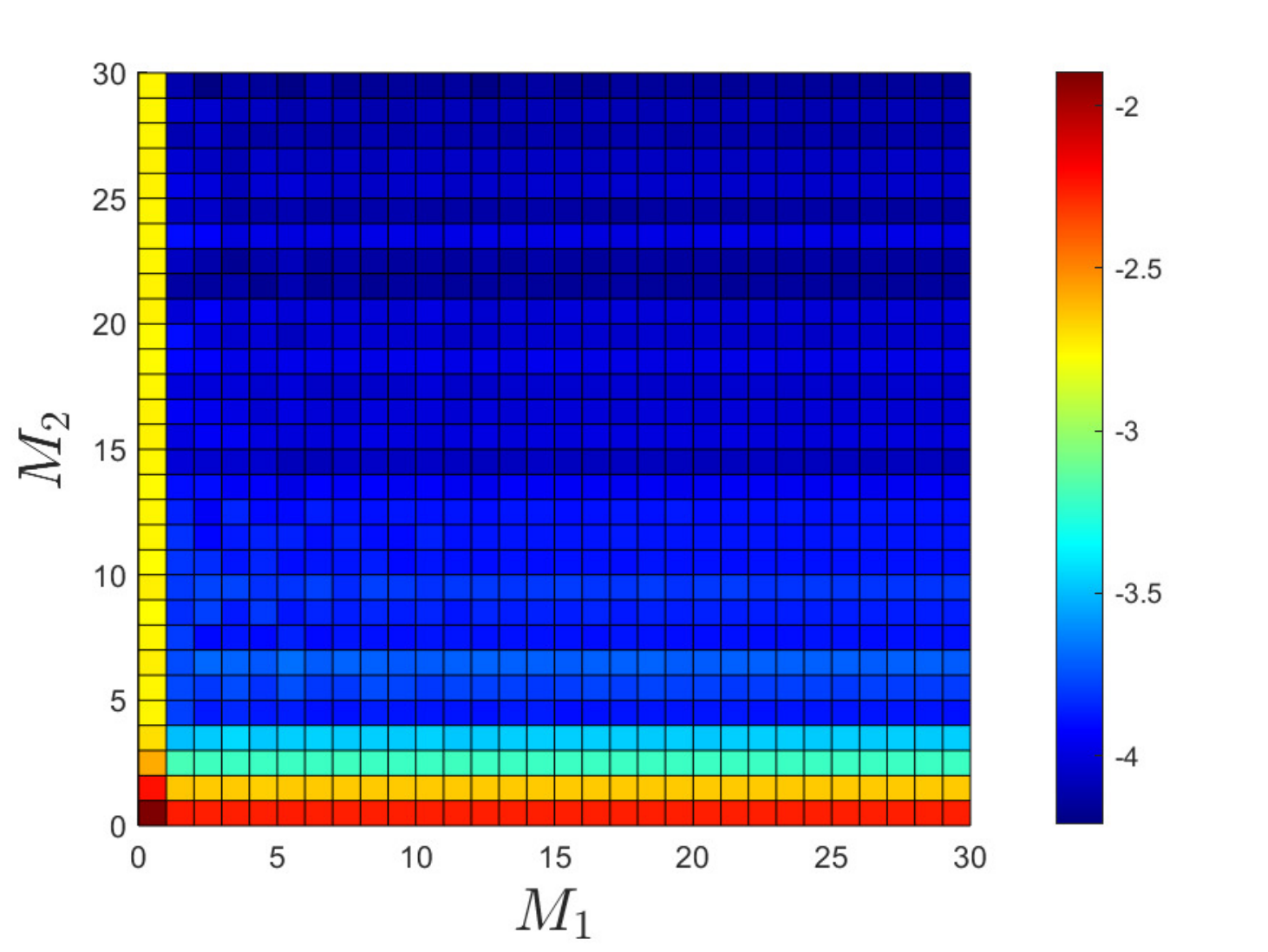}
	\includegraphics[width = 0.45\linewidth]{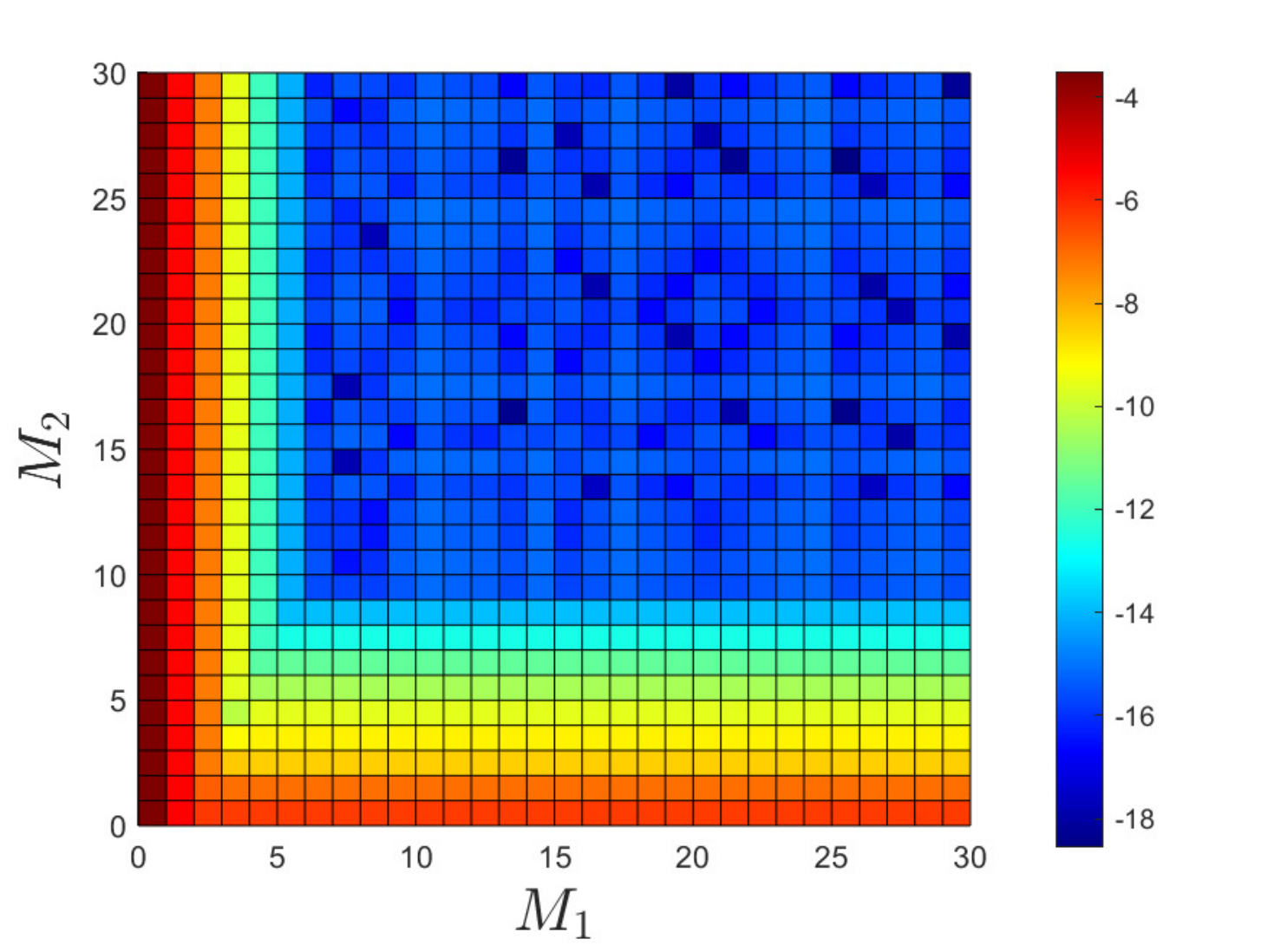}
	\caption{\textbf{Test 3}. Convergence of the $L^2$ error of the DSMC-sG scheme where binary dynamics are given by \eqref{eq:interactionDSMC} (left) or by \eqref{eq:bin_int_sigmoid} where $\beta = 0.01$, in the case of model for traffic flow with 2D uncorrelated  uncertainty in interactions and kernel. We consider $N = 10^5$, $\Delta t = \epsilon = 0.1$. We fix $\rho = 0.4$, $\mu(z_1) = 1 + 2 z_1$ and $\alpha(z_2)=2 z_2$ with $z_1,z_2 \sim \mathcal U([0,1])$. Reference solution computed with $M_1 = M_2 = 50$.}
	\label{fig:2D_conv}
\end{figure}

Finally, coupling \eqref{eq:bin_int_sigmoid} with the process \eqref{eq:rescaling}, we recover a qualitatively consistent approximation of the evolution of relevant quantities of interest in the case of traffic flow model, see Figure \ref{fig:reg_traffic}. 

\begin{figure}
\centering
\includegraphics[width = 0.31\linewidth]{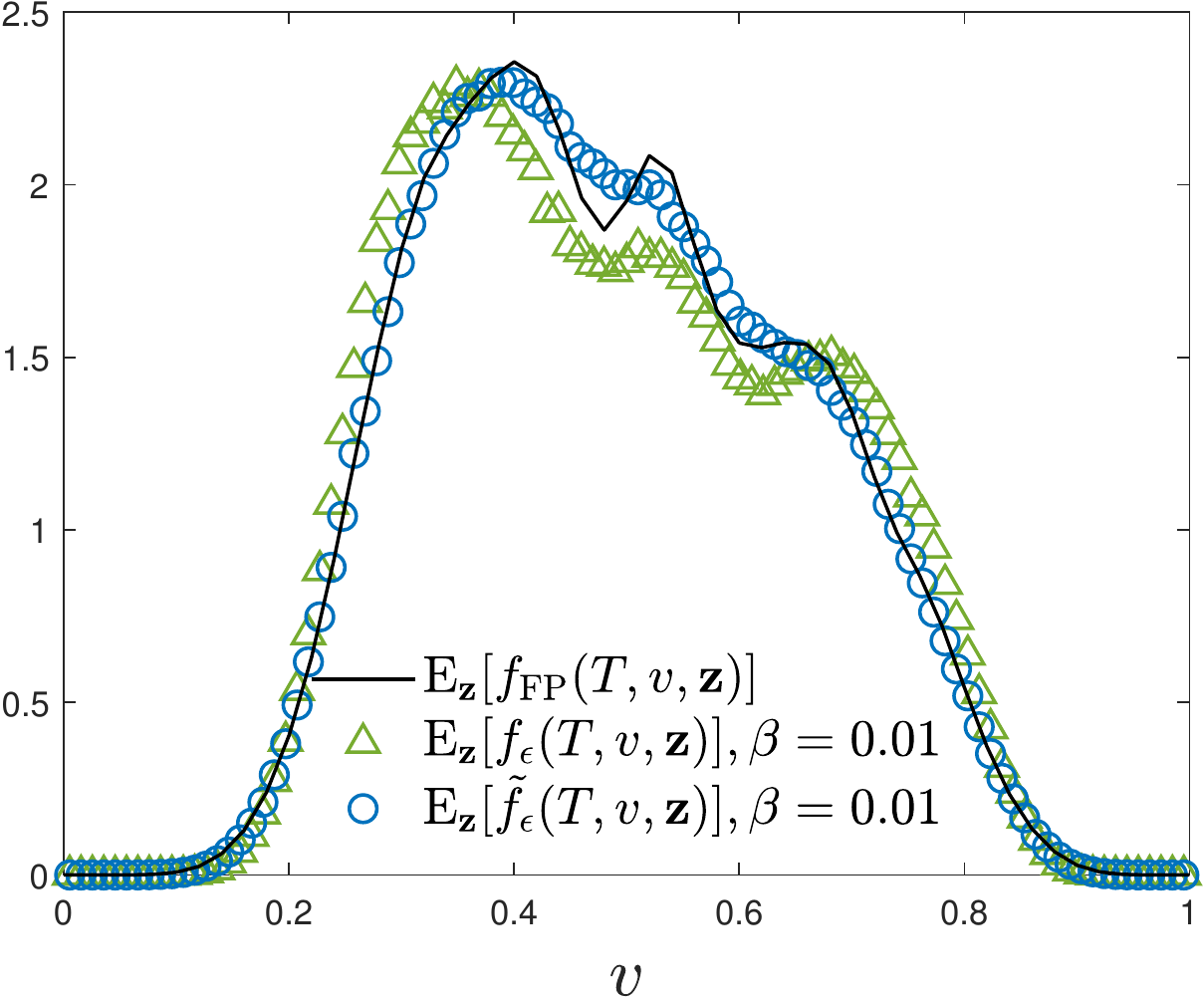}
\includegraphics[width = 0.31\linewidth]{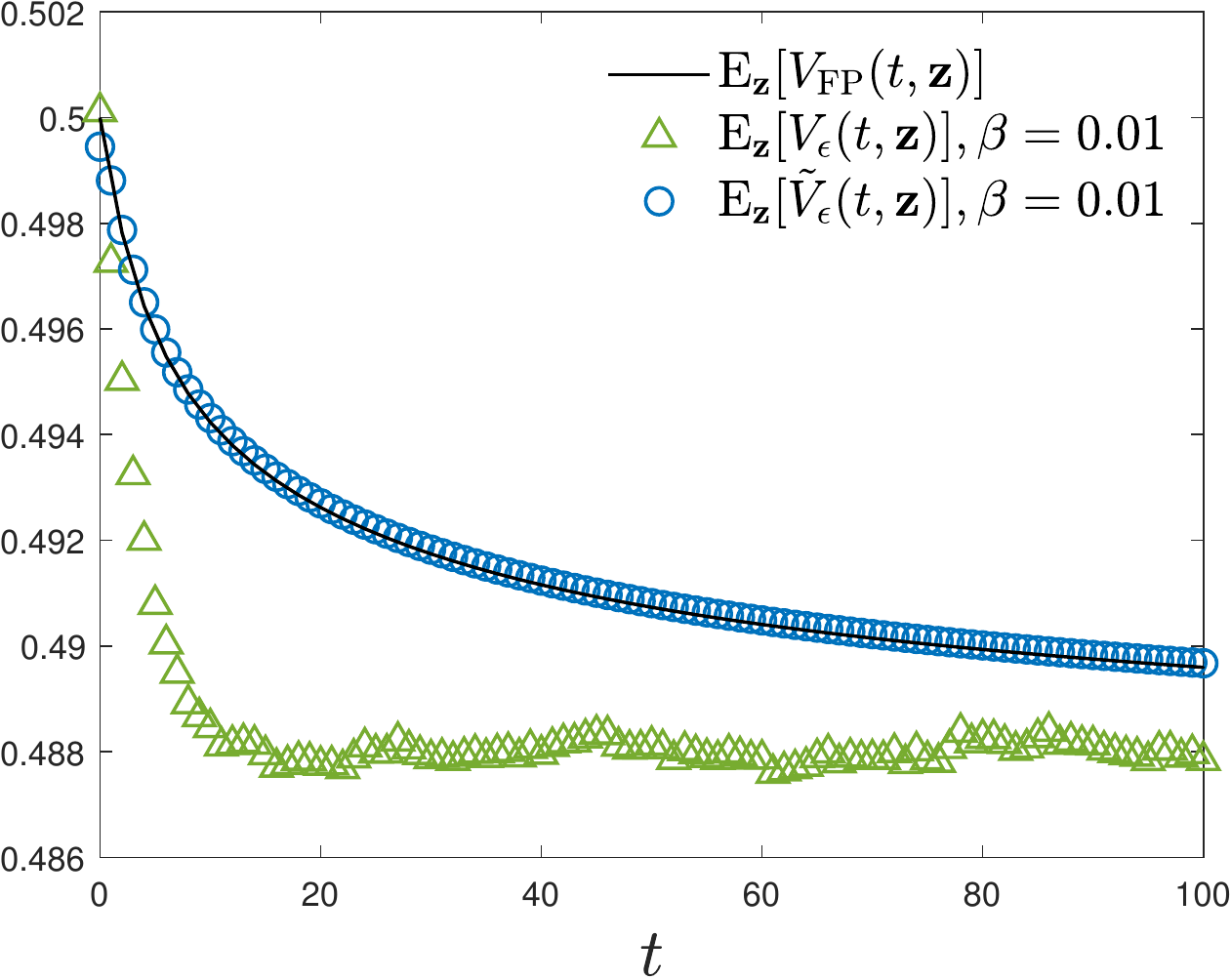}
\includegraphics[width = 0.31\linewidth]{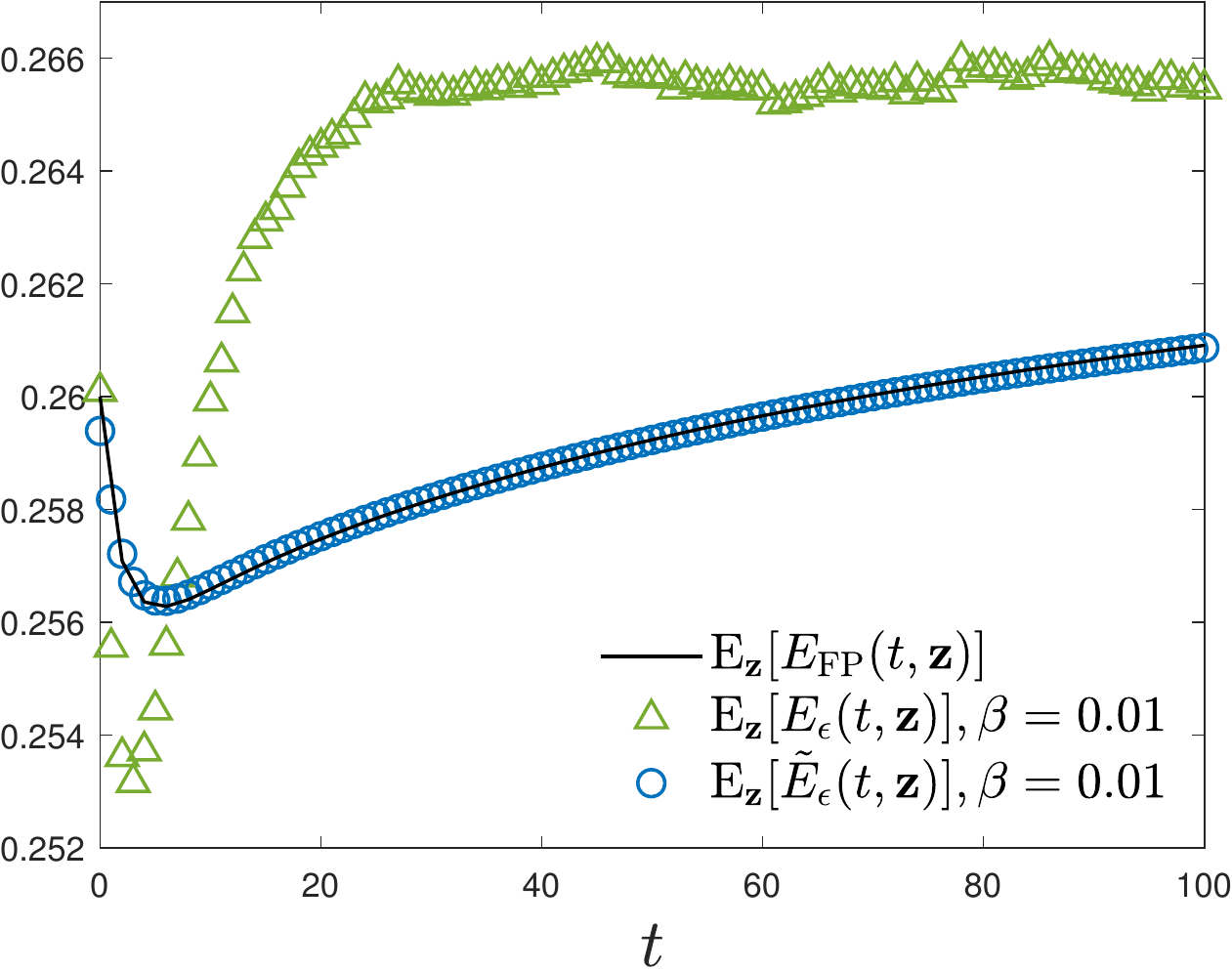} 
\includegraphics[width = 0.31\linewidth]{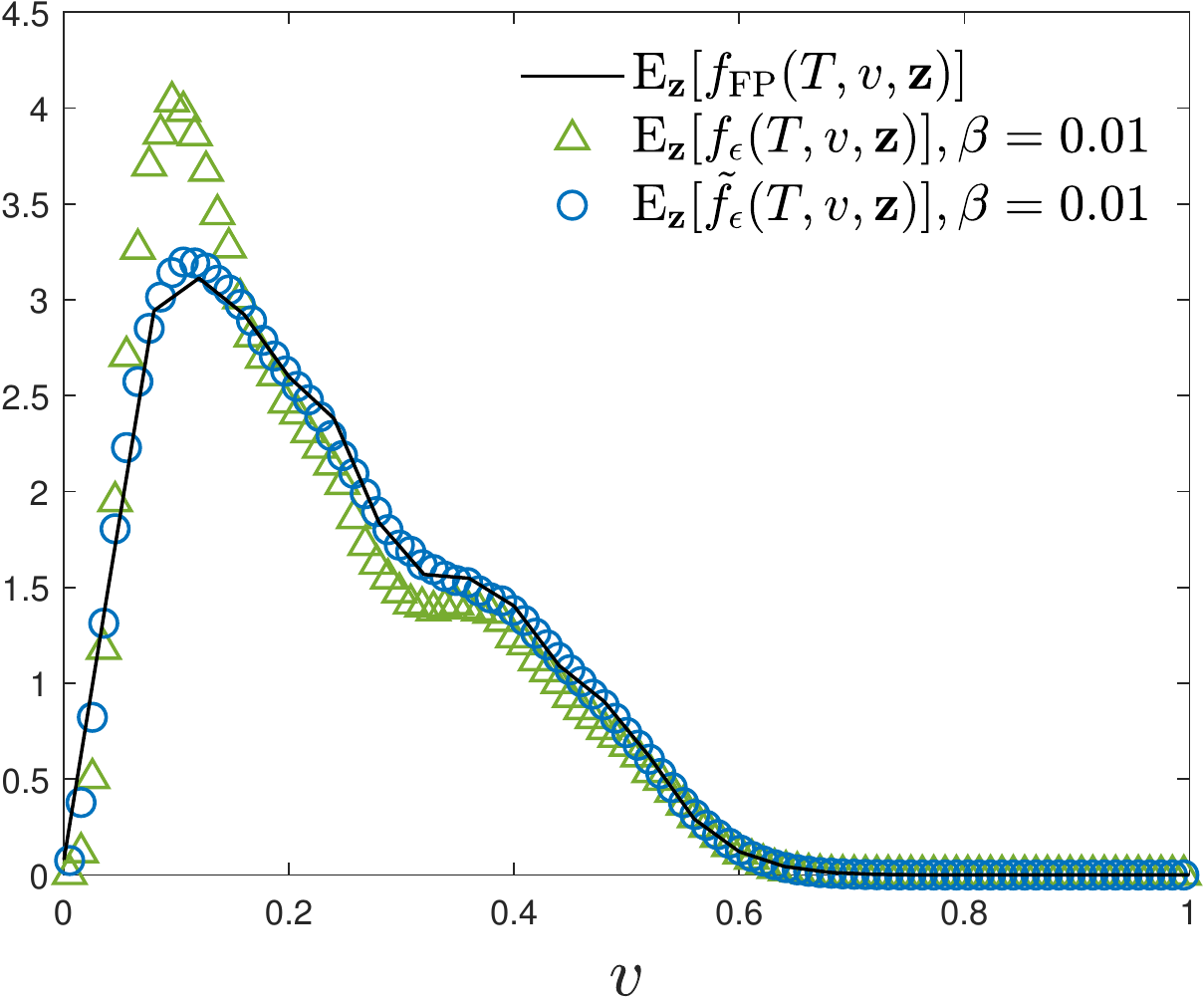}
\includegraphics[width = 0.31\linewidth]{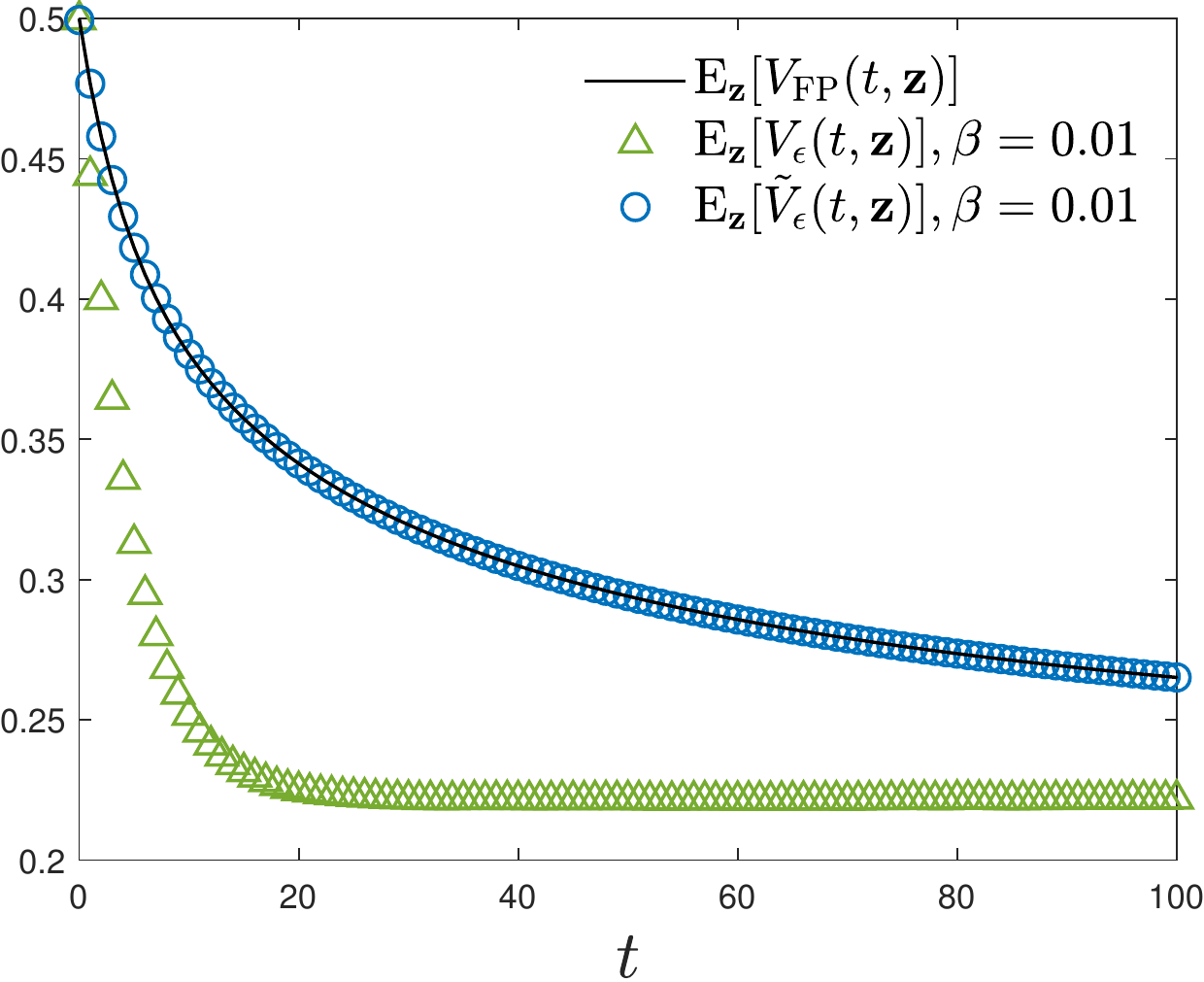}
\includegraphics[width = 0.31\linewidth]{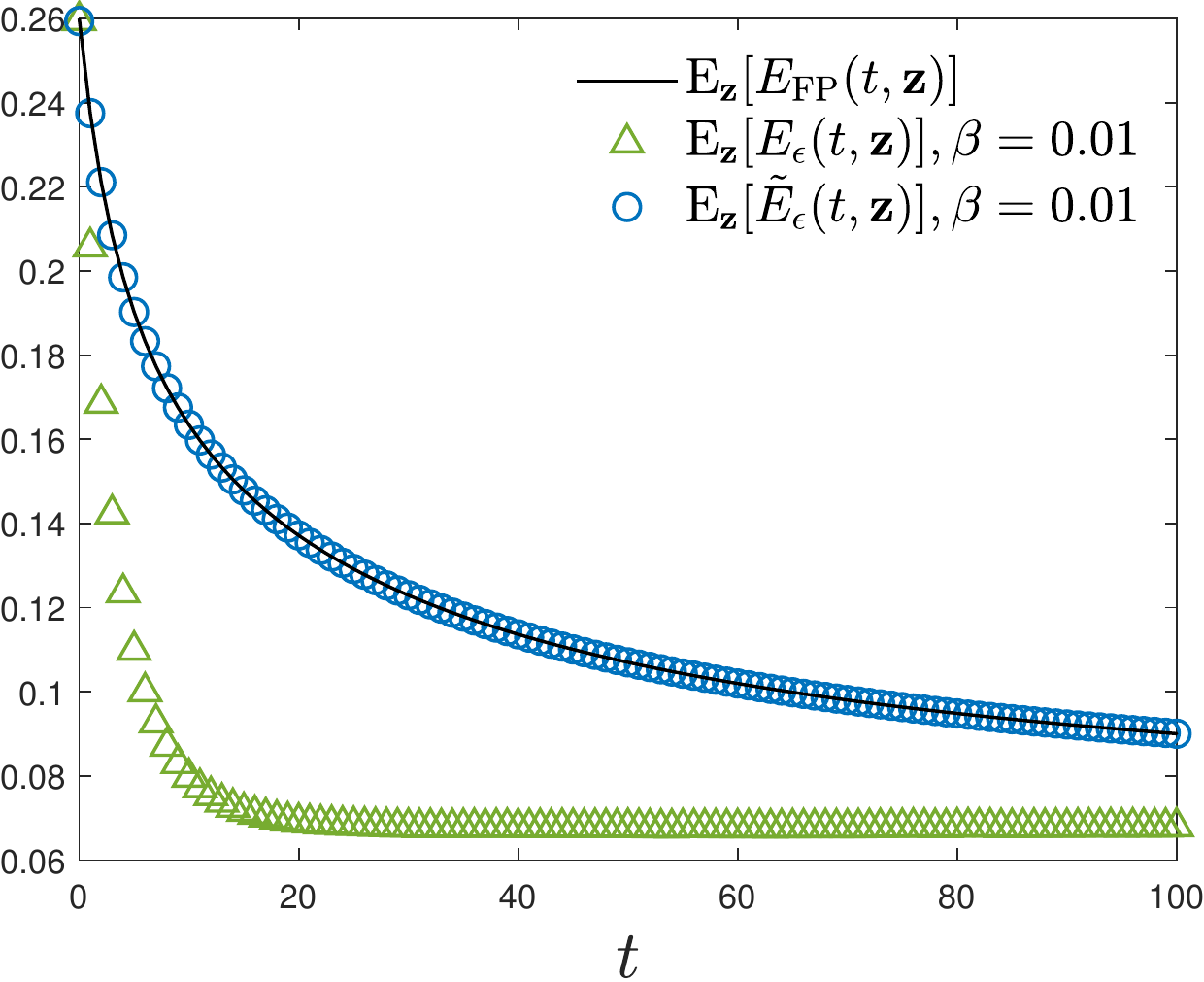}
\caption{\textbf{Test 3}. First column: comparison of numerical $f_{\textrm{FP}}(T,v,\z)$ with the DSMC-sG approximation of $f_\epsilon(T,v,\z)$ (regularization without rescaling) and of $\tilde{f}_\epsilon(T,v,\z)$ (regularization with rescaling) in terms of the expectation in $\z$. Central and right columns:  comparison of $V_{\textrm{FP}}$, $E_{\textrm{FP}}$ with the DSMC-sG approximation without rescaling $V_\epsilon(t,\z)$, $E_\epsilon(t,\z)$ and with rescaling $\tilde{V}_\epsilon(t,\z)$, $\tilde{E}_\epsilon(t,\z)$. Top row: $\rho = 0.4$, bottom row: $\rho = 0.6$. We consider $N = 10^5$, $M_1 = M_2 = 5$ and $\Delta t = \epsilon = 0.1$.   }
\label{fig:reg_traffic}
\end{figure}

\section*{Conclusion}

In this work, we studied an extension of a recently introduced DSMC-sG hybrid approach \cite{Carrillo2019,pareschi2020JCP} for uncertainty quantification of kinetic equations to non-Maxwellian Boltzmann-type models for multi-agent systems. The proposed method combines a DSMC solver in the physical space with a stochastic Galerkin method in the random space and is based on a generalized Polynomial Chaos expansion of statistical samples of a DSMC solver. The DSMC-sG solution of non-Maxwellian models with uncertainties requires a suitable reformulation of classical DSMC solvers. The class of kinetic models of interest can be formally approximated by surrogate Fokker-Planck-type models in the quasi-invariant regime. For these models, the regularity in the random space has been investigated. In particular, exploiting this observation we guarantee spectral accuracy of the method in the random space. Several examples based on existing models of multi-agent systems have been investigated numerically. The extension of the DSMC-sG methods to non-homogeneous equations of collective phenomena is currently under investigation.

\section*{Acknowledgements}
This work has been written within the activities of the GNFM group of INdAM (National Institute of High Mathematics). A.T. and M.Z. acknowledge partial support of MUR-PRIN2020 Project (No. 2020JLWP23) "Integrated mathematical approaches to socio-epidemiological dynamics". The research of M.Z. was partially supported by MUR, Dipartimenti di Eccellenza Program (2018–2022), and Department of Mathematics “F. Casorati”, University of Pavia.  The research of A.T. was partially supported by MUR, Dipartimenti di Eccellenza Program (2018–2022), and Department of Mathematical Sciences “G. L. Lagrange”, Politecnico di Torino. 

\section*{Competing interests}
The authors declare no competing interests. 

\section*{Data availability statement}
The dataset generated during the current study is available from the corresponding author on reasonable request.

\appendix

\section{Non-Maxwellian traffic model}
\label{append:A}

Let us consider first the case $\alpha(\z) = 1$. From \eqref{Vmean} we have
\begin{equation}
	\label{V_NM}
	\begin{split}
		\left|\frac{1}{\epsilon} \frac{dV(t,\z)}{dt} \right| & \le \underbrace{  \left|  \int_{0}^{1} \int_{0}^{v}(v-w)I(v,w,\z)f(t,v,\z)f(t,w,\z)dwdv \right|}_{A} \\
		&\quad+  \underbrace{\dfrac{1}{2}\left| \int_{0}^{1} \int_{0}^{1}(w-v)I(v,w,\z)f(t,v,\z)f(t,w,\z)dwdv \right|. }_{B}
	\end{split}
\end{equation}
In particular, a direct integration of the second term gives
\begin{equation*}
	B = \dfrac{1}{2}(E(t,\z)-V^2(t,\z))(P(\rho,\z)-P^2(\rho,\z)+1)>0
\end{equation*}
being $P \in [0,1]$ for all $\z$ and where $E(t,\z) = \int_0^1 v^2 f(t,v,\z)dv$ is the energy.  On the other hand, thanks to triangular inequality, we have
\begin{equation}
	\label{eq2}
	\begin{split}
		&A
		\leq \underbrace{ \int_{0}^{1} \int_{0}^{v}(v-w) |P(\rho,\z)(1-v)| f(t,v,\z)f(t,w,\z)dwdv}_{\textrm{A}_1} \\
		&\quad+\underbrace{\int_{0}^{1} \int_{0}^{v}(v-w) |(1-P(\rho,\z))(P(\rho,\z)w-v)|f(t,v,\z)f(t,w,\z)dwdv.}_{\textrm{A}_2}
	\end{split}
\end{equation}  
Since $P\ge 0$ and $v\in [0,1]$ we have 
\begin{equation*}
	\begin{split}
		 \textrm{A}_1 
		&=\int_{0}^{1} P(\rho,\z)(1-v)\left[\int_{0}^{v}(v-w)f(t,w,\z)dw\right]f(t,v,\z)dv\\
		&\leq\int_{0}^{1} P(\rho,\z)(1-v)vf(t,v,\z)dv=P(\rho,\z)(V(t,\z)-E(t,\z)).
	\end{split}
\end{equation*}
Similarly, for the second term of \eqref{eq2}  we have
\begin{equation*}
	%\label{eq5}
	\begin{split}
		\textrm{A}_2 &= \int_{0}^{1} \int_{0}^{v}(v-w) (1-P(\rho,\z)) v f(t,v,\z)f(t,w,\z)dwdv  \\
		&\leq\int_{0}^{1} (1-P(\rho,\z))v^{2}f(t,v,\z)dv=(1-P(\rho,\z))E(t,\z),
	\end{split}
\end{equation*}
From the obtained inequalities we conclude that 
\begin{equation}
\begin{split}
	\left| \frac{dV(t,\z)}{dt} \right| \leq& \frac{\epsilon}{2} \Big[ 2P(\rho,\z)V(t,\z)-V^2(t,\z)(P(\rho,\z)-P^2(\rho,\z)+1)  \\
	&+E(t,\z)(3-3P(\rho,\z)-P^2(\rho,\z)) \Big].
\end{split}
\end{equation} 
Since $v \in [0,1]$ we have $E(t,\z)\leq V(t,\z)$ and, introducing the notation $\beta(\z) = 3-3P(\rho,\z)-P^2(\rho,\z)$, we get
\begin{equation*}
	\begin{cases} 
		\left| \dfrac{dV(t,\z)}{dt} \right| \leq \frac{\epsilon}{2} \left[V(t,\z)(3-P(\rho,\z)-P^2(\rho,\z))-V^2(t,\z)(P(\rho,\z)-P^2(\rho,\z)+1)\right] & \beta > 0 \\ 
		\left| \dfrac{dV(t,\z)}{dt} \right| \leq \frac{\epsilon}{2} \left[2V(t,\z)P(\rho,\z)-V^2(t,\z)(P(\rho,\z)-P^2(\rho,\z)+1)\right] &  \beta \leq 0, 
	\end{cases}
\end{equation*} 
that are both Bernoulli-type ODEs. 

Hence, we get $V^-(t,\z) \le V(t,\z)\le V^+(t,\z)$ where
\begin{equation}
	\label{bernoulli2}
	\begin{split}
		& V^{+}(t,\z)=\left[\frac{C_2(\z)}{C_1(\z)}+e^{-C_1(\z)t}\left(\frac{1}{V_{0}}-\frac{C_2(\z)}{C_1(\z)}\right) \right]^{-1} \\ 
		& V^{-}(t,\z)=\left[\frac{C_2(\z)}{C_1(\z)}+e^{C_1(\z)t}\left(\frac{1}{V_{0}}-\frac{C_2(\z)}{C_1(\z)}\right) \right]^{-1} 
	\end{split}
\end{equation} 
where 
\begin{equation}
	\label{bernoulli}
	\begin{split}
	\begin{cases} 
		C_1(\z)=\frac{\epsilon}{2}(3-P(\rho,\z)-P^2(\rho,\z)) & \beta>0\\
		C_1(\z)= \epsilon P(\rho,\z) & \beta \le 0
	\end{cases}
	\\
	\begin{cases} 
		C_2(\z)=\frac{\epsilon}{2}(1+P(\rho,\z)-P^2(\rho,\z)) & \beta> 0  \\ 
		C_2(\z)=\frac{\epsilon}{2}(1+P(\rho,\z)-P^2(\rho,\z)) & \beta\leq 0.
	\end{cases}
	\end{split}
\end{equation} 

Hence, we consider the case $\alpha(\z) =2 $. We define the following constants $C_1(\z)=\epsilon P(\rho,\z)\geq0$ and $C_2(\z)=\frac{\epsilon}{2}(P^2(\rho,\z)+1-P(\rho,\z))>0$ and from \eqref{Vmean} we get
\begin{equation}
	\frac{dV(t,\z)}{dt}=C_1(\z)\left(E(t,\z)-V^2(t,\z)\right)+C_2(\z)\left(E(t,\z)V(t,\z) - M_3(t,\z)\right),
\end{equation}
where $M_3(t,\z) = \int_0^1 v^3 f(t,v,\z)dv$. Hence, from the triangular inequality we find 
\begin{equation}
	\left|\frac{dV(t,\z)}{dt}\right|\leq C_1(\z)\left(E(t,\z)-V^2(t,\z)\right)+C_2(\z)\left|E(t,\z)V(t,\z) - M_3(t,\z)\right|.
\end{equation}
Arguing as before, we get
\begin{equation}
\begin{split}
	C_1(\z)\left(E(t,\z)-V^2(t,\z)\right) &\leq C_1(\z)\left(V(t,\z)-V^2(t,\z)\right), \\
	C_2(\z)\left|E(t,\z)V(t,\z) - M_3(t,\z)\right| &\leq C_2(\z) V(t,\z).
\end{split}
\end{equation}
Therefore we get $V^-(t,\z)\le V(t,\z) \le V^+(t,\z)$ where
\begin{equation}
	\begin{split}
		 V^{+}(t,\z)&=\left[\frac{C_1(\z)}{C_1(\z)+C_2(\z)}+e^{-(C_1(\z)+C_2(\z))t}\left(\frac{1}{V_{0}}-\frac{C_1(\z)}{C_1(\z)+C_2(\z)}\right) \right]^{-1} \\ 
		 V^{-}(t,\z)&=\left[\frac{C_1(\z)}{C_1(\z)+C_2(\z)}+e^{(C_1(\z)+C_2(\z))t}\left(\frac{1}{V_{0}}-\frac{C_1(\z)}{C_1(\z)+C_2(\z)}\right) \right]^{-1},
	\end{split}
\end{equation} 
with $ V(0,\z) = V_0 \in [0,1]$.

\bibliographystyle{plain}
\bibliography{UQ_kernel_traffic.bib}

\begin{thebibliography}{10}

\bibitem{babovsky1989}
H.~Babovsky and R.~Illner.
\newblock A convergence proof for nanbu's simulation method for the full
  boltzmann equation.
\newblock {\em SIAM J. Numer. Anal.}, 26:45--65, 1989.

\bibitem{babovsky1986}
H.~Babovsky and H.~Neunzert.
\newblock On a simulation scheme for the boltzmann equation.
\newblock {\em Math. Meth. Appl. Sci.}, 8:223--233, 1986.

\bibitem{Bassetti2010}
F.~Bassetti and G.~Toscani.
\newblock Explicit equilibria in a kinetic model of gambling.
\newblock {\em Phys. Rev. E}, 81:066115, 2010.

\bibitem{bertaglia_etal}
G.~Bertaglia, L.~Liu, L.~Pareschi, and X.~Zhu.
\newblock Bi-fidelity stochastic collocation methods for epidemic transport
  models with uncertainties.
\newblock {\em Netw. Heterog. Media}, in press.

\bibitem{caflisch1998monte}
R.~E. Caflisch.
\newblock Monte {C}arlo and quasi {M}onte {C}arlo methods.
\newblock {\em Acta numerica}, 7:1--49, 1998.

\bibitem{CDOP}
J.~A. Carrillo, M.~R. D'Orsogna, and V.~Panferov.
\newblock Double milling in self-propelled swarms from kinetic theory.
\newblock {\em Kinetic. Relat. Models}, 2(2):363--378, 2009.

\bibitem{CFRT}
J.~A. Carrillo, M.~Fornasier, J.~Rosado, and G.~Toscani.
\newblock Asymptotic flocking dynamics for the kinetic {C}ucker-{S}male model.
\newblock {\em SIAM J. Math. Anal.}, 42(1):218--236, 2010.

\bibitem{Carrillo2019}
J.~A. Carrillo, L.~Pareschi, and M.~Zanella.
\newblock Particle based g{PC} methods for mean-field models of swarming with
  uncertainty.
\newblock {\em Commun. Comput. Phys.}, 25(2):508--531, 2019.

\bibitem{cercignani1994BOOK}
C.~Cercignani, R.~Illner, and M.~Pulvirenti.
\newblock {\em The Mathematical Theory of Dilute Gases}, volume 106 of {\em
  Applied Mathematical Sciences}.
\newblock Springer, 1994.

\bibitem{choi2021}
Y.-P. Choi and S.-B. Yun.
\newblock Existence and hydrodynamic limit for a {P}averi-{F}ontana type
  kinetic traffic model.
\newblock {\em SIAM J. Math. Anal.}, 53:2631--2659, 2021.

\bibitem{cordier2005JSP}
S.~Cordier, L.~Pareschi, and G.~Toscani.
\newblock On a kinetic model for a simple market economy.
\newblock {\em J. Stat. Phys.}, 120(1):253--277, 2005.

\bibitem{DM08}
P.~Degond and S.~Motsch.
\newblock Continuum limit of self-driven particles with orientation
  interaction.
\newblock {\em Math. Models Meth. Appl. Sci.}, 18(supp01):1193--1215, 2008.

\bibitem{desvillettes}
L.~Desvillettes.
\newblock Boltzmann's kernel and the spatially homogeneous {B}oltzmann
  equation.
\newblock {\em Riv. Mat. Univ. Parma}, 6(4):1--22, 2001.

\bibitem{dimarco_acta}
G.~Dimarco and L.~Pareschi.
\newblock Numerical methods for kinetic equations.
\newblock {\em Acta Numerica}, 23:369--520, 2014.

\bibitem{dimarco2019}
G.~Dimarco and L.~Pareschi.
\newblock Multi-scale control variate methods for uncertainty quantification in
  kinetic equations.
\newblock {\em J. Comp. Phys.}, 388:63--89, 2019.

\bibitem{dimarco2020}
G.~Dimarco and L.~Pareschi.
\newblock Multiscale variance reduction methods based on multiple control
  variates for kinetic equations with uncertainties.
\newblock {\em Multiscale Model. Simul.}, 18(1):351--382, 2020.

\bibitem{Dimarco2017}
G.~Dimarco, L.~Pareschi, and M.~Zanella.
\newblock Uncertainty quantification for kinetic models in socio-economic and
  life sciences.
\newblock In S.~Jin and L.~Pareschi, editors, {\em Uncertainty quantification
  for Hyperbolic and Kinetic Equations}, volume~14 of {\em SEMA-SIMAI Springer
  Series}, pages 151--191. Springer, 2017.

\bibitem{DPTZ}
G.~Dimarco, B.~Perthame, G.~Toscani, and M.~Zanella.
\newblock Kinetic models for epidemic dynamics with social heterogeneity.
\newblock {\em J. Math. Biol.}, 83(4), 2021.

\bibitem{Dragulescu2000}
A.~Dragulescu and V.~M. Yakovenko.
\newblock Statistical mechanics of money.
\newblock {\em Eur. Phys. J. B}, 17:723--729, 2000.

\bibitem{during2021}
B.~D{\"u}ring, M.~Fischer, and M.-T. Wolfram.
\newblock An {E}lo-type rating model for players and teams of variable
  strength.
\newblock {\em Philos. Trans. R. Soc. Lond. Ser. A Phys. End. Sci.}, 2021.

\bibitem{during2019}
B.~D{\"u}ring, M.~Torregrossa, and M.-T. Wolfram.
\newblock Boltzmann and {F}okker-{P}lanck equations modelling the {E}lo rating
  system with learning effects.
\newblock {\em J. Nonlinear Sci.}, 29(3):1095--1128, 2019.

\bibitem{during2015}
B.~D{\"u}ring and M.-T. Wolfram.
\newblock Opinion dynamics: inhomogeneous {B}oltzmann-type equations modelling
  opinion leadership and political segregation.
\newblock {\em Proc. R. Soc. A}, 417(2182), 2015.

\bibitem{Toscani_nonMax}
G.~Furioli, A.~Pulvirenti, E.~Terraneo, and G.~Toscani.
\newblock Non-{M}axwellian kinetic equations modelling the evolution of wealth
  distribution.
\newblock {\em Math. Models Meth. Appl. Sci.}, 30(4):685--725, 2020.

\bibitem{gamba2021}
I.~Gamba, S.~Jin, and L.~Liu.
\newblock Error estimate of a bi-fidelity method for kinetic equations with
  random parameters and multiple scales.
\newblock {\em Int. J. Uncertain. Quantif.}, 11(5):57--75, 2021.

\bibitem{helbing1996}
D.~Helbing.
\newblock Gas-kinetic derivation of {N}avier-{S}tokes-like traffic equations.
\newblock {\em Phys. Rev. E}, 53(3):2366--2381, 1996.

\bibitem{Herty2005}
M.~Herty, A.~Klar, and L.~Pareschi.
\newblock General kinetic models for vehicular traffic flow and {M}onte {C}arlo
  methods.
\newblock {\em Comput. Meth. Appl. Math.}, 5:155--169, 2005.

\bibitem{Herty2010}
M.~Herty and L.~Pareschi.
\newblock {F}okker-{P}lanck asymptotics for traffic flow models.
\newblock {\em Kinet. Relat. Mod.}, 3(1):165--179, 2010.

\bibitem{hu2019}
J.~Hu, S.~Jin, and R.~Shu.
\newblock On stochastic galerkin approximation of the nonlinear {B}oltzmann
  equation with uncertainty in the fluid regime.
\newblock {\em J. Comp. Phys.}, 397:108838, 2019.

\bibitem{hu2021}
J.~Hu, L.~Pareschi, and Y.~Wang.
\newblock Uncertainty quantification for the {BGK} model of the boltzmann
  equation using multilevel variance reduced monte carlo methods.
\newblock {\em SIAM/ASA J. Uncert. Quantif.}, 9(2):650--680, 2021.

\bibitem{JLM}
S.~Jin, J.~G. Liu, and Z.~Ma.
\newblock Uniform spectral convergence of the stochastic galerkin method for
  the linear transport equations with random inputs in diffusive regime and a
  micro--macro decomposition-based asymptotic-preserving method.
\newblock {\em Res. Math. Sci.}, 4(15), 2017.

\bibitem{Jin2017}
S.~Jin and L.~Pareschi, editors.
\newblock {\em Uncertainty quantification for hyperbolic and kinetic
  equations}, volume~14 of {\em SEMA-SIMAI Springer Series}.
\newblock Springer, 2017.

\bibitem{kac59}
M.~Kac.
\newblock {\em Probability and {R}elated {T}opics in the {P}hysical
  {S}ciences}.
\newblock New York Interscience, 1959.

\bibitem{QL}
Q.~Li and L.~Wang.
\newblock Uniform regularity for linear kinetic equations with random input
  based on hypocoercivity.
\newblock {\em SIAM/ASA J. Uncert. Quantif.}, 5(1):1193--1219, 2017.

\bibitem{LJ18}
L.~Liu and S.~Jin.
\newblock Hypocoercivity based sensitivity analysis and spectral convergence of
  the stochastic {G}alerkin approximation to collisional kinetic equations with
  multiple scales and random inputs.
\newblock {\em Multiscale Model. Simul.}, 16(3):1085--1114, 2018.

\bibitem{nanbu1980}
K.~Nanbu.
\newblock Direct simulation scheme derived from the {B}oltzmann equation. i.
  monocomponent gases.
\newblock {\em J. Phys. Soc. Jpn.}, 49:2042--2049, 1980.

\bibitem{pareschi2001ESAIMP}
L.~Pareschi and G.~Russo.
\newblock An introduction to {M}onte {C}arlo methods for the {B}oltzmann
  equation.
\newblock {\em ESAIM: Proc.}, 10:35--75, 2001.

\bibitem{pareschi2013BOOK}
L.~Pareschi and G.~Toscani.
\newblock {\em Interacting {M}ultiagent {S}ystems: {K}inetic equations and
  {M}onte {C}arlo methods}.
\newblock Oxford University Press, 2013.

\bibitem{pareschi2022}
L.~Pareschi, T.~Trimborn, and M.~Zanella.
\newblock Mean-field control variate methods for kinetic equations with
  uncertainties and applications to socio-economic sciences.
\newblock {\em Int. J. Uncertain. Quantif.}, 12(1):61--84, 2022.

\bibitem{Pareschi2018}
L.~Pareschi and M.~Zanella.
\newblock Structure preserving schemes for nonlinear {F}okker-{P}lanck
  equations and applications.
\newblock {\em J. Sci. Comput.}, 74:1575--1600, 2018.

\bibitem{pareschi2020JCP}
L.~Pareschi and M.~Zanella.
\newblock {M}onte {C}arlo stochastic {G}alerkin methods for the {B}oltzmann
  equation with uncertainties: {S}pace-homogeneous case.
\newblock {\em J. Comp. Phys.}, 423:109822, 2020.

\bibitem{paveri1975TR}
S.~L. Paveri-Fontana.
\newblock On {B}oltzmann-like treatments for traffic flow: a critical review of
  the basic model and an alternative proposal for dilute traffic analysis.
\newblock {\em Transportation Res.}, 9(4):225--235, 1975.

\bibitem{perthame}
B.~Perthame.
\newblock {\em Transport Equations in Biology}.
\newblock Frontiers in Mathematics. Birkh{\"a}user Basel, 2007.

\bibitem{prigogine1971BOOK}
I.~Prigogine and R.~Herman.
\newblock {\em Kinetic {T}heory of {V}ehicular {T}raffic}.
\newblock American Elsevier Publishing Co., New York, 1971.

\bibitem{temam88}
R.~Temam.
\newblock {\em Infinite-Dimensional Dynamical Systems in Mechanics and
  Physics}.
\newblock Springer-Verlag, 1988.

\bibitem{toscani2006CMS}
G.~Toscani.
\newblock Kinetic models of opinion formation.
\newblock {\em Commun. Math. Sci.}, 4(3):481--496, 2006.

\bibitem{tosin2018CMS}
A.~Tosin and M.~Zanella.
\newblock Boltzmann-type models with uncertain binary interactions.
\newblock {\em Commun. Math. Sci.}, 16(4):963--985, 2018.

\bibitem{tosin2019MMS}
A.~Tosin and M.~Zanella.
\newblock Kinetic-controlled hydrodynamics for traffic models with
  driver-assist vehicles.
\newblock {\em Multiscale Model. Simul.}, 17(2):716--749, 2019.

\bibitem{tosin2021MCRF}
A.~Tosin and M.~Zanella.
\newblock Uncertainty damping in kinetic traffic models by driver-assist
  controls.
\newblock {\em Math. Control Relat. Fields}, 2021.

\bibitem{Visconti2017}
G.~Visconti, M.~Herty, G.~Puppo, and A.~Tosin.
\newblock Multivalued fundamental diagrams of traffic flow in the kinetic
  {F}okker-{P}lanck limit.
\newblock {\em Multiscale Model. Simul.}, 15:1267--1293, 2017.

\bibitem{Xiu2010}
D.~Xiu.
\newblock {\em Numerical Methods for Stochastic Computations}.
\newblock Princeton University Press, 2010.

\bibitem{Xiu2002}
D.~Xiu and G.~E. Karniadakis.
\newblock The {W}iener-{A}skey polynomial chaos for stochastic differential
  equations.
\newblock {\em SIAM J. Sci. Comput.}, 24(2):614--644, 2002.

\bibitem{zanella2020MCS}
M.~Zanella.
\newblock Structure preserving stochastic {G}alerkin methods for
  {F}okker-{P}lanck equations with background interactions.
\newblock {\em Math. Comput. Simulation}, 168:28--47, 2020.

\bibitem{Zhu2017}
Y.~Zhu and S.~Jin.
\newblock The {V}lasov-{P}oisson-{F}okker-{P}lanck system with uncertainty and
  a one-dimensional asymptotic-preserving method.
\newblock {\em Multiscale Model. Simul.}, 15(4):1502--1529, 2017.

\end{thebibliography}

\end{document}